\documentclass[onefignum,onetabnum,final]{siamart250211}
\usepackage{amsmath,amssymb,amsfonts,latexsym,stmaryrd}
\usepackage{graphicx,float,epsfig} 
\usepackage{mathrsfs} 
\usepackage{color}
\usepackage{setspace} % Packages and macros go here
\usepackage{lipsum}
\usepackage{graphicx,float}
\usepackage{color}
\usepackage{epstopdf}
\usepackage{multirow}
\usepackage{svg}
\usepackage{url}
\usepackage{diagbox}
\ifpdf
  \DeclareGraphicsExtensions{.eps,.pdf,.png,.jpg}
\else
  \DeclareGraphicsExtensions{.eps}
\fi
\newcommand\err{\texttt{err}}
\newcommand\eff{\texttt{eff}}

\def\ds{\displaystyle}
\def\G{\Gamma}
\def\O{\Omega}

\def\l{\lambda}

\def\mV{\mathcal{V}}

\renewcommand\sp{\mathop{\mathrm{Sp}}\nolimits}
\newcommand{\set}[1]{\lbrace #1 \rbrace}

\newcommand{\jump}[1]{\llbracket #1 \rrbracket}
\newcommand{\mean}[1]{\{#1\}}

\newcommand{\norm}[1]{\lVert#1\rVert}

\newcommand{\n}{\boldsymbol{n}}
\usepackage{booktabs}
\usepackage{float}

\newcommand\bu{\boldsymbol{u}}
\newcommand\bv{\boldsymbol{v}}
\newcommand\bw{\boldsymbol{w}}

\newcommand\bn{\boldsymbol{n}}

\def\hdel{\widehat{\delta}}%\def\VK{V^{\E}}
\def\CM{\mathcal{X}}
\def\CN{\mathcal{Y}}

\newcommand\bF{\boldsymbol{f}}
\newcommand\bbK{\mathbb{K}}
\newcommand\bI{\boldsymbol{I}}

\newcommand\bT{\boldsymbol{T}}

\newcommand\0{\mathbf{0}}

\newcommand{\cF}{\mathcal{F}}

\newcommand\cP{\mathcal{P}}
\newcommand\cT{\mathcal{T}}
\def\CT{{\mathcal T}}

\newcommand\bcE{\boldsymbol{\mathcal{E}}}

\newcommand\btau{\boldsymbol{\tau}}

\renewcommand\H{\mathrm{H}}

\renewcommand\L{\mathrm{L}}
\renewcommand\O{\Omega}

\newcommand{\vertiii}[1]{{\left\vert\kern-0.25ex\left\vert\kern-0.25ex\left\vert #1 
		\right\vert\kern-0.25ex\right\vert\kern-0.25ex\right\vert}}

\renewcommand\div{\mathop{\mathrm{div}}\nolimits}

\renewcommand\sp{\mathop{\mathrm{sp}}\nolimits}

\newsiamremark{remark}{Remark}
\newsiamremark{hypothesis}{Hypothesis}
\newsiamthm{problem}{Problem}
\newsiamthm{claim}{Claim}

\headers{DG methods for the Stokes-Brinkman eigenvalue problem}{Felipe Lepe,  Gonzalo Rivera and Jesus Vellojin}

\title{Discontinuous Galerkin approximation for a Stokes-Brinkman-type formulation for the eigenvalue problem in porous media\thanks{Submitted to the editors DATE.
\funding{The second author was partially supported by ANID-Chile through FONDECYT project 1231619.  The third author was partially supported by ANID-Chile through FONDECYT Postdoctorado project 3230302.
}}}

\author{Felipe Lepe\thanks{GIMNAP-Departamento de Matem\'atica, Universidad del B\'io - B\'io, Casilla 5-C, Concepci\'on, Chile. \email{flepe@ubiobio.cl}.}
\and Gonzalo Rivera\thanks{Departamento de Ciencias Exactas,
	Universidad de Los Lagos, Casilla 933, Osorno, Chile. \email{gonzalo.rivera@ulagos.cl}.}
\and Jesus Vellojin\thanks{GIMNAP-Departamento de Matem\'atica, Universidad del B\'io - B\'io, Casilla 5-C, Concepci\'on, Chile. \email{jvellojin@ubiobio.cl}.}}

\usepackage{amsopn}

\ifpdf
\hypersetup{
  pdftitle={The Stokes-Brinkman eigenvalue problem},
  pdfauthor={Felipe Lepe, Gonzalo Rivera and Jesus Vellojin}
}
\fi

\def\CT{{\mathcal T}}

\begin{document}

\maketitle

% REQUIRED
\begin{abstract}
We introduce a family of discontinuous Galerkin methods to approximate the eigenvalues and eigenfunctions of a Stokes-Brinkman  type of problem based in the interior penalty strategy. Under the standard assumptions on the meshes and a suitable norm, we prove the stability of the discrete scheme. Due to the non-conforming nature of the method, we use the well-known non-compact operators theory to derive convergence and error estimates for the method. We present an exhaustive computational analysis where we compute the spectrum with different stabilization parameters with the aim of study its influence when the spectrum is approximated.
\end{abstract}

% REQUIRED incpor

% REQUIRED

\section{Introduction}\label{sec:intro}
The approximation of eigenvalues and eigenfunctions of certain systems involving partial differential equations 
 is a subject in constant progress on the numerical analysis community, where different methods and techniques have emerged through the years to approximate accurately these quantities. In \cite{MR2652780}  we found the different approaches to study numerically eigenvalue problems arising from partial differential equations from the finite element method point of view, where prima and mixed formulations are analyzed. However, these techniques are possible to be extended for other  type of methods, particularly discontinuous Galerkin methods (DG) whose interesting features as for example, the  meshes admit hanging-nodes with no restriction and elemental polynomial bases consisting of locally variable polynomial degrees are also admissible, owing to the lack of pointwise continuity requirements across the mesh-skeleton. Since the elements of the mesh does not share  degrees of freedom (DOFs) at their interfaces, the DG method is intrinsically parallelizable, making it very efficient for large-scale computations. This provides a crucial benefit that offsets the greater number of DOFs it uses relative to the finite element method (FEM). Also, an important number of solvers are now available to solve the resulting linear systems when the IPDG is implemented in two and three dimensions (see \cite{MR2323689,MR3246098,MR2522959} for instance).
 
The application of DG for eigenvalue problems, in particular the interior penalization approach (IPDG) was introduced, for the best of the authors's knowledge, for the Laplace eigenvalue problem in \cite{MR2220929},  where the authors proved spectral correctness and error estimates  under suitable norms. These developments have been  also useful to tackle other eigenvalue problems as \cite{MR2324460,MR4542511, MR4673997,MR3962898,MR4623018, MR4077220} where the Maxwell's equations, elasticity, and flow equations, in primal and mixed formulations have been considered. On these references, one of the main contribution besides the mathematical and numerical analysis, is related to the computational experimentation. More precisely, since the use of the IPDG method requieres the introduction of a stabilization parameter intrinsic to penalize the jumps on interments, which is positive and chosen proportionally to the square of the polynomial degree as is proposed in \cite{MR2324460}. If  this parameter is not correctly chosen on the implementation, may introduce undesirable eigenvalues  with no physical meaning. These are the so-called spurious eigenvalues, and the effects of this parameter must be analyzed for IPDG methods. On the other hand, the IPDG has shown in the previously mentioned references a high accuracy on the approximation of eigenvalues for the operators in two and three dimensions and different geometries and boundary conditions. This makes the IPDG method a suitable alternative to solve eigenvalue problems arising from  continuum mechanics.

In the present paper, we continue with our research program related to DG to solve numerically eigenvalue problems now  focusing our attention on the Stokes-Brinkman eigenvalue problem. This problem, already studied in for example \cite{williamson2019posteriori} for the load problem,   has the particularity of interpolate between pure Stokes flow and damped flow in porous media (similar to Brinkman or Darcy–Stokes). Now, the eigenvalue problem version of  this problem is stated as follows: let  $\O\subset\mathbb{R}^d$,  with $d\in\{2,3\}$,  be an open bounded domain with Lipschitz boundary $\partial\O$. Let $\Gamma_1$ and $\Gamma_2$ be disjoint open subset of $\partial\O$ such that $\partial \O=\overline{\Gamma}_1\cup\overline{\Gamma}_2$. Considering a steady-state of the balance laws for linear fluid equations, the problem to be studied is given as follows 
\begin{align}
	\bbK^{-1}\bu - \nu \Delta \bu + \nabla p & = \lambda \bu & \text{in $\Omega$},\label{eq:Brinkman1}\\
	\div\bu &= 0 & \text{in $\Omega$},\label{eq:Brinkman2}\\
	\bu &= \boldsymbol{0}& \text{on $\Gamma_1$}, \label{bc:Sigma} \\
	(\nu\nabla \bu - p\mathbb{I}) \bn   &=\boldsymbol{0} & \text{on $\Gamma_2$} \label{bc:Gamma},
\end{align}
where $\bu$ represents the fluid velocity, $p$ is the pressure, $\mathbb{I}\in\mathbb{R}^{d\times d}$ denotes  the identity matrix, the tensor $\mathbb{K}$ is the parameter associated to the permeability of the domain,  $\nu$ is the viscosity of the fluid, and $\bn$ represents the unit normal. For tensor  $\mathbb{K}$ we assume that is  bounded, symmetric, and positive definite. Let us observe that if  $\mathbb{K}^{-1}\rightarrow \boldsymbol{0}$, system \eqref{eq:Brinkman1}--\eqref{bc:Gamma} is nothing but a Stokes eigensystem,  while if $\nu=0$, we have Darcy's law, which we do not consider for the analysis of an  eigenvalue problem. Regarding the domain $\O$, we split it in two media $\overline{\Omega}:=\overline{\Omega}_S\cup\overline{\Omega}_D$,  where $\Omega_S$ and $\Omega_D$ represent subdomains where there is free flow and porous media, respectively. For the free-flow domain, we take $\mathbb{K}^{-1}=\boldsymbol{0}$, while $\mathbb{K}^{-1}\gg\boldsymbol{0}$ is considered in $\Omega_D$.  This allows to study the eigenmodes on domains where we have membrane-like behavior or internal filters.

System \eqref{eq:Brinkman1}--\eqref{bc:Gamma} has been previously studied in \cite{lepe2025jsc} using inf-sup stable finite elements. In that work, convergence of the solution operators is established, along with both a priori and a posteriori error analyses. On this reference is analyzed how eigenvalues—and consequently, eigenfunctions—are affected by variations in the permeability of the medium in both two and three dimensions. Although the finite element method (FEM) is a well-known and effective technique for approximating the spectrum of operators, it remains somewhat restrictive compared to more general methods—particularly IPDG methods, which offer greater flexibility, as previously mentioned. Now our goal is to extend the results of \cite{lepe2025jsc} to a more general method as the IPDG method.  Unlike \cite{lepe2025jsc}, where the analysis is based on the theory of compact operators, the analysis of nonconforming methods such as IPDG requires the theoretical framework developed in \cite{MR483400, MR483401} in order to establish convergence and derive error estimates, as shown, for instance, in \cite{MR2220929, MR4623018}. Moreover, since we are dealing with varying permeabilities, not only can geometric singularities affect the regularity of certain eigenfunctions, but the heterogeneity of the permeability coefficient itself may also impact them in different ways. This motivates the development of an a posteriori error estimator that is reliable, efficient, and fully computable within the IPDG framework. Also, these analyzes need to be supported with numerical tests. In this sense, as we discussed above in the influence of the stabilization parameter, we theoretically prove that the solution operator associated to the velocity is well defined when the stabilization is sufficiently large compared with quantities depending on physical parameters such as viscosity and permeability (cf. Lemma \ref{lmm:disc_ellipticity}) and this must be confirmed by the theory, motivating the design of experiments where the influence of the stabilization plays an important role on the convergence and the appearance of spurious eigenvalues, as is studied in  \cite{ MR4623018, MR3962898, MR4077220, MR4673997} for  IPDG methods in different contexts.

% 
%This motivates the development  of robust numerical methods and particularly, the  study of IPDG methods. As the literature dictates,  we need to prove that the method, in one hand stable, where this stability depends on the stabilization parameter and how it must be chosen, whereas on the other hand, we prove that the method is convergent and spurious free. To do this task, it is necessary to resort to the theory of \cite{MR483400,MR483401} as is proposed in \cite{MR2220929} to provide convergence and error estimates which will depend on the regularity of the solution and the polynomial degree of the approximation. Also, on this paper we propose a fully computable  a posteriori error estimator which results to be reliable and efficient.     

 \subsection{Outline of the paper} The notations for Lebesgue spaces, norms, and inner products are the standard along our manuscript. In section \ref{sec:model_problem} we summarize some details about the model problem, solutions operators, well posedness of the load problem and regularity results. Immediately in section \ref{sec:DG} we introduce the definitions of the elements of the mesh, norms, seminorms, polynomial spaces in order to define  the IPDG methods. Here we introduce the discrete bilinear forms and hence, the IPDG discretizations of the model problem. The discrete solutions operators are introduced with the corresponding discrete load problems which we prove are well posed. With this discrete framework at hand, in section \ref{sec:conv_error} we analyze convergence and error estimates for the discontinuous numerical schemes. Section  \ref{sec:apost} contains the analysis for an a posteriori error estimator of the residual type and, finally in section \ref{sec:numerics}, we report a complete experimental analysis for  in order to assess the performance of the IPDG methods when the spectrum of the model problem is approximated.
 
 %The interior penalty discontinuous Galerkin methods (IPDG), which first intention was to solve numerically source problems, 

%%%%%%%%%%%%%%%%%%%%%%%%%%%%%%%%%%%%%%%%%%%%%%%%%%%%%%%%%%%%%
%%%%%%%%%%%%%%%%%%%%%%%%%%%%%%%%%%%%%%%%%%%%%%%%%%%%%%%%%%%%%
%\subsection{Notation}\label{sec:notation}
%Let $\mathcal{O}$ be a subset of $\mathbb{R}^d$ with Lipschitz boundary $\partial\mathcal{O}$.   For $r \geq 0$ and $p \in[1, \infty]$, we denote by $\mathrm{L}^p(\mathcal{O})$ the usual Lebesgue space of maps from $\mathcal{O}$ to $\mathbb{R}$ endowed with the norm $\|\cdot\|_{\mathrm{L}^p(\mathcal{O})}$, while $\boldsymbol{\H}^r(\mathcal{O})$ denotes a Hilbert space. Vectors spaces and vector-valued functions will be written in bold letters. For instance, for $r\geq 0$, we simply write $\boldsymbol{\H}^r(\O)$ instead of $[\H^r(\O)]^d$.  As usual, we write $|\cdot|_{r, \mathcal{O}}$ and $\Vert \cdot\Vert_{r,\Omega}$ to denote the seminorm and norm in Hilbert spaces. We define $\boldsymbol{\H}_0^1(\mathcal{O})$ as the vector space of functions in $\boldsymbol{\H}^1(\mathcal{O})$ with vanishing trace on $\partial \mathcal{O}$, and $\mathrm{L}_0^2(\mathcal{O})$ as the space of $\mathrm{L}^2(\mathcal{O})$ functions with vanishing mean value over $\mathcal{O}$. Finally, $\mathbf{0}$ denotes a generic null vector or tensor. 

\section{Functional framework and variational formulation}
\label{sec:model_problem}

Let us establish the functional framework in which we will operate. The space where we seek the  velocity is
%The Sobolev spaces for the velocity and pressure are
$$
	\boldsymbol{\H}_{\Gamma_1}(\Omega):= \{ \bv \in \boldsymbol{\H}^1(\Omega): \bv = \boldsymbol{0} \text{ on } \Gamma_1\},
$$
whereas $\mathrm{L}^2(\Omega)$ is the space for the pressure. If $\Gamma_2 = \emptyset$, then the pressure is defined up to a constant. Hence, we take $\boldsymbol{\H}_0^1(\Omega)$ for the velocity space and $\L_0^2(\Omega)$ for the pressure.

Throughout this work, we assume that the permeability tensor is positive definite for all $\bv\in\boldsymbol{\H}_{\Gamma_1}(\Omega)$. More precisely, there exist positive constants $\mathbb{K} _*,\mathbb{K}^*>0$ such that
\begin{equation*}
	\label{eq:kappa-bounds}
	0<\mathbb{K}_*\Vert\bv\Vert_{0,\O}^2 \leq (\mathbb{K}^{-1}\bv,\bv)_{0,\Omega}\leq \mathbb{K}^*\Vert\bv\Vert_{0,\O}^2.
\end{equation*}
% Check on the proofs if we need $0 < \nu \le 1$.

A variational formulation for system \eqref{eq:Brinkman1}--\eqref{bc:Gamma} is the following: find $\lambda\in\mathbb{R}^+$, the velocity $\boldsymbol{0}\neq \bu\in\boldsymbol{\H}_{\Gamma_1}(\Omega)$, and the pressure $0\neq p\in\mathrm{L}^2(\O)$ such that
\begin{align*}
	\int_{\Omega} \mathbb{K}^{-1} \bu \cdot \bv + \nu \int_{\Omega} \nabla \bu :\nabla \bv - \int_{\Omega} p \div \bv & = \lambda \int_{\Omega} \bu \cdot \bv &\forall\bv\in\boldsymbol{\H}_{\Gamma_1}(\Omega), \\
	- \int_{\Omega} \div \bu\,  q &= 0  &\forall q\in \mathrm{L}^2(\Omega).
\end{align*}

Defining   the continuous bilinear forms $a:\boldsymbol{\H}_{\Gamma_1}(\Omega)\times\boldsymbol{\H}_{\Gamma_1}(\Omega)\rightarrow\mathbb{R}$ and $b:\boldsymbol{\H}_{\Gamma_1}(\Omega)\times \mathrm{L}^2(\Omega)\rightarrow \mathbb{R}$ as follows
$$
\begin{aligned}
	a(\bu,\bv)&:=\int_{\Omega} \mathbb{K}^{-1} \bu \cdot \bv + \nu \int_{\Omega} \nabla \bu :\nabla \bv,\quad\text{and}\quad
	b(\bv,q)&:= - \int_{\Omega} q \div\bv,
\end{aligned}
$$
the weak formulation  is rewritten in the following abstract form:
\begin{problem}\label{prob:continuous}
	Find $\lambda\in\mathbb{R}^+$ and $(\boldsymbol{0},0)\neq(\bu,p)\in\boldsymbol{\H}_{\Gamma_1}(\Omega)\times \mathrm{L}^2(\Omega)$ such that
\begin{equation*}
	\label{eq:porous_system_continuous}
	\left\{
	\begin{aligned}
		a(\bu,\bv) + b(\bv,p) &= \lambda(\bu,\bv)_{0,\Omega},&\forall \bv\in\boldsymbol{\H}_{\Gamma_1}(\Omega),\\
		b(\bu,q)&=0,&\forall q\in \L^2(\Omega).
	\end{aligned}
	\right.
\end{equation*}
where $(\cdot,\cdot)_{0,\Omega}$ denotes the usual $\mathrm{L}^2$ inner product.
\end{problem}

%In the forthcoming analysis, we need a suitable weighted norm which depends on the physical parameters. With this in mind, given $(\bv,q)\in\boldsymbol{\H}_{\Gamma_1}(\Omega)\times \mathrm{L}^2(\Omega)$, we define the following norm:
%\begin{equation}
%	\label{eq:weigthed-norm}
%	\vertiii{(\bv,q)}^2:= \|\mathbb{K}^{-1/2} \bv\|_{0,\Omega}^2 +   \|\nu^{1/2} \nabla \bv \|_{0,\Omega}^2 + \| \div \bv \|_{0,\Omega}^2 + \Vert q\Vert_{0,\O}^2.
%\end{equation}
%\cblue{\begin{equation}
%	\label{eq:weigthed-norm}
%	\vertiii{(\bv,q)}^2:= \|\mathbb{K}^{-1/2} \bv\|_{0,\Omega}^2 +   \|\nu^{1/2} \nabla \bv \|_{0,\Omega}^2  + \Vert q\Vert_{0,\O}^2.
%\end{equation}}

We denote by  $\boldsymbol{\mathcal{K}}$ the kernel of $b(\cdot,\cdot)$ which si defined by 
$$\boldsymbol{\mathcal{K}}:= \{ \bv \in \boldsymbol{\H}_{\Gamma_1}(\Omega): b(\bv, q)=0 \quad \forall q \in \mathrm{L}^2(\Omega) \} = \{ \bv \in\boldsymbol{\H}_{\Gamma_1}(\Omega): \div \bv = 0 \}. $$

With this space at  hand, it is direct to prove that  bilinear form $a(\cdot,\cdot)$ satisfies
\begin{align*}
	a(\bu, \bv)  &\le \max\{{\mathbb{K}^*},\nu\}\Vert \bu\Vert_{1,\Omega}\Vert\bv\Vert_{1,\O}\quad \forall \bu,  \bv \in\boldsymbol{\H}_{\Gamma_1}(\Omega), \\
	a(\bv, \bv) &\ge \min\{{\mathbb{K}_*},\nu\}\Vert \bv\Vert_{1,\Omega}^2 \quad \forall \bv \in \boldsymbol{\mathcal{K}}.
\end{align*}

On the other hand,  we have that there exists $\beta>0$ such that the following inf-sup condition holds
\begin{equation}
\label{eq:cont_inf_sup}
\sup_{\boldsymbol{0} \neq\bv \in \boldsymbol{\H}_{\Gamma_1}(\Omega)}\frac{b(\bv, q)}{\Vert\bv\Vert_{1,\Omega}}  \ge \beta \norm{q}_{0,\Omega}, \qquad \forall q \in \mathrm{L}^2(\Omega).
\end{equation}
Let us  define the following continuous bilinear form
$$
A((\bu,p),(\bv,q)):= a(\bu,\bv) + b(\bv,p) + b(\bu,q),
$$
which allows to rewrite Problem \ref{prob:continuous}  in the following manner:
\begin{problem}\label{prob:continuous-A}
	Find $\lambda\in\mathbb{R}^+$ and $(\boldsymbol{0},0)\neq(\bu,p)\in\boldsymbol{\H}_{\Gamma_1}(\Omega)\times \mathrm{L}^2(\Omega)$ such that
	\begin{equation*}
		\label{eq:porous_system_continuous-A}
		A((\bu,p),(\bv,q))= \lambda(\bu,\bv)_{0,\Omega},\qquad\forall (\bv,q)\in\boldsymbol{\H}_{\Gamma_1}(\Omega)\times \L^2(\Omega).
	\end{equation*}
	\end{problem}
%The following result establish a general inf-sup condition for the bilinear form $A$.

%\begin{lemma}
%	\label{lemma:elliptic}
%	For all $\boldsymbol{0}\neq (\bu,p)\in\boldsymbol{\H}_{\Gamma_1}(\Omega)\times \L^2(\Omega)$, there exists $(\bv,q)\in\boldsymbol{\H}_{\Gamma_1}(\Omega)\times \L^2(\Omega)$ with $\vertiii{(\bv,q)}\leq C_1\vertiii{(\bu,p)}$ such that
%	$$
%	A((\bu,p),(\bv,q))\geq C_2 \vertiii{(\bu,p)}^2,
%	$$
%	where $C_1$, $C_2$ are constants depending on the physical parameters, and the triple norm is defined in \eqref{eq:weigthed-norm}.
%\end{lemma}

	For the analysis of  Problem \ref{prob:continuous} (namely Problem \ref{prob:continuous-A}) it is necessary to introduce the corresponding solution operators associated to the velocity and the pressure as in \cite{MR2652780}. With this idea in mind, let us denote by  $\bT$ the operator associated to the velocity and    $\mathcal{S}$ the one associated to the pressure, which are defined by 
\begin{align*}
\bT&:\boldsymbol{\L}^2(\O)\rightarrow\boldsymbol{\H}_{\Gamma_1}(\Omega),\,\,\,\bF\mapsto\bT\bF:=\widetilde{\bu}, \\
\mathcal{S}&:\boldsymbol{\L}^2(\O)\rightarrow \L^2(\O),\,\,\,\bF\mapsto\mathcal{S}\bF:=\widetilde{p}, 
\end{align*}
%For the analysis, we introduce the so-called solution operator, which we denote by $\bT$, and defined by 
%$$
%\bT:\boldsymbol{\L}^2(\O)\rightarrow\boldsymbol{\H}_{\Gamma_1}(\Omega),\,\,\,\bF\mapsto\bT\bF:=\widetilde{\bu}, 
%$$
where the pair  $(\widetilde{\bu},\widetilde{p})\in\boldsymbol{\H}_{\Gamma_1}(\Omega)\times\L^2(\O)$ solves the following source problem 
\begin{equation}
	\label{eq:porous_system_continuous_source}
	\left\{
	\begin{aligned}
		a(\widetilde{\bu},\bv) + b(\bv,\widetilde{p}) &= (\boldsymbol{f},\bv)_{0,\Omega},&\forall \bv\in\boldsymbol{\H}_{\Gamma_1}(\Omega),\\
		b(\widetilde{\bu},q)&=0,&\forall q\in \L^2(\Omega),
	\end{aligned}
	\right.
\end{equation}
which is well posed due to \eqref{eq:cont_inf_sup},  the coercivity of $a(\cdot,\cdot)$ on $\boldsymbol{\mathcal{K}}$, and the Babu\v ska-Brezzi theory. This implies that $\bT$ and $\mathcal{S}$ are well defined. 
%is the solution of the following source problem
%\begin{equation} 
%	\label{eq:source}
%	A((\widetilde{\bu},\widetilde{p}),(\bv,q))=(\bF,\bv)_{0,\O},\qquad \forall(\bv,q)\in \boldsymbol{\H}_{\Gamma_1}(\Omega)\times \L^2(\Omega),\\
%\end{equation}
%which is well posed thanks to  the coercivity of the bilinear form $a(\cdot,\cdot)$ and the inf-sup  condition \eqref{eq:cont_inf_sup}.  %Observe that $\bT$ is a compact operator thanks to the compact inclusion of $\H^1(\O)^n$ onto $\L^2(\O)^n$. Also, the symmetry of $\mathcal{A}$ allows to
Moreover, let $\mu$ be a real number such that  $\mu\neq 0$. Notice that $(\mu,\boldsymbol{u})\in \mathbb{R}^+\times\boldsymbol{\H}_{\Gamma_1}(\Omega)$ is an eigenpair of $\bT$ if and only if  there exists $p\in \L^2(\Omega)$ such that,  $(\l,(\boldsymbol{u},p))$ solves Problem \ref{prob:continuous-A}  with $\mu:=1/\l$.

Let us remark that trivially, we can consider the following source problem: find $(\widetilde{\bu},\widetilde{p})\in\boldsymbol{\H}_{\Gamma_1}(\Omega)\times\L^2(\O)$ such that 
\begin{equation} 
	\label{eq:source}
	A((\widetilde{\bu},\widetilde{p}),(\bv,q))=(\bF,\bv)_{0,\O},\qquad \forall(\bv,q)\in \boldsymbol{\H}_{\Gamma_1}(\Omega)\times \L^2(\Omega).
\end{equation}

Since  $(\mathbb{K}^{-1}\bu,\bv)_{0,\O}$ is well defined, from the well known Stokes regularity results (see \cite{MR975121,MR1600081} for instance), we have  the following additional regularity result for the solution of the source problem \eqref{eq:source}.
\begin{theorem}\label{th:regularidadfuente}
There exists $s>0$  that for all $\boldsymbol{f} \in \boldsymbol{\L}^2(\O)$, the solution $(\widetilde{\bu},\widetilde{p})\in\boldsymbol{\H}_{\Gamma_1}(\Omega)\times \mathrm{L}^2(\Omega)$ of problem \eqref{eq:source}, satisfies for the velocity $\widetilde{\bu}\in  \boldsymbol{\H}^{1+s}(\Omega)$, for the pressure $\widetilde{p}\in \H^s(\Omega)$, and
 \begin{equation*}
\|\widetilde{\bu}\|_{1+s,\O}+\|\widetilde{p}\|_{s,\O}\leq C \|\boldsymbol{f}\|_{0,\O},
 \end{equation*}
 where $C>0$ is a constant  depending on  the physical parameters.
\end{theorem}
\begin{remark}
It is worth mentioning that the estimate provided in Theorem \ref{th:regularidadfuente} holds for the solutions of the load problem. Furthermore, for the eigenfunctions, there exists $r>0$ and  a constant 
$C>0$, which depends on the physical parameters and the eigenvalue $\lambda$, such that
\begin{equation}\label{eq_regularityofev}
\|\bu\|_{1+r,\O}+\|p\|_{r,\O}\leq C \|\bu\|_{0,\O}.
 \end{equation}
\end{remark}

Finally, since the embedding $\boldsymbol{\H}^{1+s}(\O)\hookrightarrow\boldsymbol{ \L}^2(\O)$ holds, we
conclude that $\bT$ is compact and the following  spectral characterization of $\bT$ holds.
\begin{lemma}(Spectral Characterization of $\bT$).
The spectrum of $\bT$ is such that $\sp(\bT)=\{0\}\cup\{\mu_{k}\}_{k\in{N}}$ where $\{\mu_{k}\}_{k\in\mathbf{N}}$ is a sequence of real  eigenvalues that converge to zero, according to their respective multiplicities. 
\end{lemma}

We end this section with a result that establishes a general inf-sup condition for the bilinear form $A(\cdot,\cdot)$, which is essential for proving the reliability of the a posteriori estimator and its proof is available in \cite[Lemma 4.3]{Paul2005347}.
\begin{lemma}\label{lemma:elliptic}
	For all $(\boldsymbol{0},0)\neq (\bu,p)\in\boldsymbol{\H}_{\Gamma_1}(\Omega)\times \L^2(\Omega)$, there exists $(\bv,q)\in\boldsymbol{\H}_{\Gamma_1}(\Omega)\times \L^2(\Omega)$ with $\vertiii{(\bv,q)}\leq C_1\vertiii{(\bu,p)}$ such that
	$$
	A((\bu,p),(\bv,q))\geq C_2 \vertiii{(\bu,p)}^2,
	$$
	where $\vertiii{(\bv,p)}^2=\|\bv\|_{1,\O}^2+\|q\|_{0,\O}^2$ and $C_1$ and $C_2$ are constants depending on the physical parameters.
\end{lemma}
%\begin{proof}
%The proof is followed by the same arguments as \cite[Lemma 4.3]{Paul2005347}
%	\end{proof}
%	

\section{The DG method}
\label{sec:DG}
Let us introduce the IPDG methods. With this purpose, we need to 
state the framework in which our analysis will be performed, implying the introduction of definitions and 
notations that we will use through the work. We begin with the elements of the mesh.
%The forthcoming analysis is inspired by \cite{MR2220929}.
%\subsection{Preliminaries}
Let $\mathcal{T}_h$ be a shape regular family of meshes which subdivide the domain $\bar \Omega$ into  
triangles/tetrahedra that we denote by $K$. Let us denote by $h_K$
the diameter of any element $K\in\mathcal{T}_h$ and let $h$ be the maximum of the diameters of all the
elements of the mesh, i.e. $h:= \max_{K\in \cT_h} \{h_K\}$.

Let $F$ be a closed set. We say that  $F\subset \overline{\Omega}$
is an interior edge/face if $F$ has a positive $(d-1)$-dimensional 
measure and if there are distinct elements $K$ and $K'$
such that $F =\bar K\cap \bar K'$. A closed 
subset $F\subset \overline{\Omega}$ is a boundary edge/face if
there exists $K\in \cT_h$ such that $F$ is an edge/face
of $K$ and $F =  \bar K\cap \G$. 
Let $\cF_h^0$ and $\cF_h^\partial$ be  the sets  of interior edges/faces
and  boundary edges/face, respectively.
We assume that the boundary mesh $\cF_h^\partial$ is
compatible with the partition $\G = \G_1 \cup \G_2$, namely,
\[
\bigcup_{F\in \cF_h^1} F = \G_1 \qquad \text{and} \qquad \bigcup_{F\in \cF_h^2} F = \G_2,
\]
where $\cF_h^1:= \set{F\in \cF_h^\partial; \quad F\subset \G_1}$ and 
$\cF_h^2:= \set{F\in \cF_h^\partial: \quad F\subset \G_2}$.
Also we denote   $\cF_h := \cF_h^0\cup \cF_h^\partial$ and $\cF^*_h:= \cF_h^{0} \cup \cF_h^{1}$.
Also, for any element $K\in \cT_h$, we introduce the set
$\cF(K):= \set{F\in \cF_h:\quad F\subset \partial K}$  of edges/faces composing the boundary of $K$.

For any $t\geq 0$, we define the following broken Sobolev space 
\[
 \boldsymbol{\H}^t(\cT_h):=
 \set{\bv \in \boldsymbol{\L}^2(\O): \quad \bv|_K\in \boldsymbol{\H}^t(K)\quad \forall K\in \cT_h}.
\]
%
%Now, for $\bv:=\set{\bv_K}\in H^t(\cT_h)^n$ 
%the components $\bv_K$   represent the
%restrictions $\bv|_K$, when it is
%convenient, we will drop the subscript for these restrictions.
Also, the space of the skeletons of the triangulations
$\cT_h$ is defined  by $ \L^2(\cF_h):= \prod_{F\in \mathcal{F}_h} \L^2(F).$

In the forthcoming analysis, $h_\cF\in \L^2(\cF_h)$
will represent the piecewise constant function 
defined by $h_\cF|_F := h_F$ for all $F \in \cF_h$,
where $h_F$ denotes the diameter of edge/face $F$.

Let  $\cP_m(\cT_h)$ be  the space of piecewise polynomials respect with to
 $\cT_h$  of degree at most $m\geq 0$;
namely,  
\[
 \cP_m(\cT_h) :=\left\{ v\in \L^2(\O):\, \bv|_K \in \cP_m(K), \forall K\in \cT_h \right\}. 
\]
Let $k\geq 1$.  To approximate the velocity we define the following space 
\begin{equation*}
\boldsymbol{\mV}_h:=\{\bv_h\in \boldsymbol{\L}^2(\O)\,:\quad \bv_h|_K\in \boldsymbol{\cP}_k(K),\,\,\,\forall K\in\CT_h\},
\end{equation*}
where $\boldsymbol{\cP}_k(K):=\cP_k(K)^d$, whereas for the approximation of the pressure, we consider the following space
\begin{equation*}
\mathcal{Q}_h:=\{q_h\in \L^2(\O)\,:\quad q_h|_K\in \cP_{k-1}(K),\,\,\,\forall K\in\CT_h\}.
\end{equation*}

Given a scalar field $q$, we define the average $\mean{q}\in \L^2(\cF_h)$ and the jump $\jump{q}\in \L^2(\cF_h)$ by  
$$\mean{q} := (q_K + q_{K'})/2,\qquad \jump{q} := q_K\n_K + q_{K'}\n_{K'},$$
respectively,  where $\n_K$ is the outward unit normal vector to $\partial K$ and $q_K$ represents the restriction $q|_K$. 
Similarly, for a vector field $\bv$, the average $\mean{\bv}\in \boldsymbol{\L}^2(\cF_h)$  and scalar jump $\jump{\underline{\bv}}\in \L^2(\cF_h)$ are given by 
$$\mean{\bv} := (\bv_K + \bv_{K'})/2\qquad \jump{\underline{\bv}} := \bv_K \cdot\n_K + \bv_{K'}\cdot\n_{K'},$$
respectively, while the tensor (or total) jump  $\jump{\bv}\in [\L^2(\cF_h)]^{d\times d}$ is defined by 
$$\jump{\bv} := \bv_K \otimes\n_K + \bv_{K'}\otimes\n_{K'}.$$ 
Finally, if $\btau$ is a tensor field, we define the corresponding average and jump as $$\mean{\btau} := (\btau_K + \btau_{K'})/2 \in [\L^2(\cF_h)]^{d\times d},\qquad \jump{\btau\bn} := \btau_K\n_K + \btau_{K'}\n_{K'}\in \boldsymbol{\L}^2(\cF_h),$$ respectively. If  $K\in \CT_h$ is such that a facet $\cF$ satisfies $\cF\in \cF_h^\partial$, we can obtain the definition of average and jump in the domain boundary by taking $K=K'$ and $K'=0$ in the above definitions, respectively.

%We define averages $\mean{\bv}\in [L^2(\cF_h)]^n$
%and jumps $\jump{\bv}\in L^2(\cF_h)$ 
%as follows
%\[
% \mean{\bv}_F := (\bv_K + \bv_{K'})/2 \quad \text{and} \quad \jump{\bv}_F := \bv_K \cdot\n_K + \bv_{K'}\cdot\n_{K'} 
% \quad \forall F \in \cF(K)\cap \cF(K'),
%\]
%where $\n_K$ is the outward unit normal vector to $\partial K$ and $\bv_K$ represents the restriction $\bv|_K$. We remark that when it is convenient, we will drop the subscript for this restriction.
%Also, on the boundary  $\partial\O$ and for all  $F \in \cF(K)\cap \DO$,
%the averages and jumps are defined by $ \mean{\bv}_F := \bv_K $
%and  $\jump{\bv}_F := \bv_K \cdot\n $, respectively.

Motivated by \cite{MR2220929}, let us define the space  $\boldsymbol{\boldsymbol{\mathcal{V}}}(h):=\boldsymbol{\boldsymbol{\mathcal{V}}}+\boldsymbol{\boldsymbol{\mathcal{V}}}_h$ which we endow with the following norm
%\cblue{\begin{equation*}
%\|\bv\|_{\boldsymbol{\boldsymbol{\mathcal{V}}}(h)}^2= \|\nabla_h\bv\|_{0,\O}^2+\|h^{-1/2}\jump{\bv}\|_{0,\mathcal{F}_h}^2.
%\end{equation*}}
\begin{equation*}
\|\bv\|_{\boldsymbol{\boldsymbol{\mathcal{V}}}(h)}^2=\|\bv\|_{0.\O}^2+\|\nabla_h\bv\|_{0,\O}^2+\|h_{\cF}^{-1/2}\jump{\bv}\|_{0,\mathcal{F}_h}^2.
\end{equation*}
 Finally, given a function $\xi(x)\in \L^{\infty}(\O)$, we introduce the following trace inequality \cite[Section 3]{MR4673997}
\begin{equation}
\label{eq:dg_trace}
\|h^{1/2}\set{\xi(x)\bv}\|_{0,\mathcal{F}}\leq \overline{C}_\xi \|\bv\|_{0,\Omega}\quad\forall  \bv\in \boldsymbol{\cP}_k(\cT_h),
\end{equation}
where $\overline{C}_\xi$ is a positive constant independent of the mesh but depending on $\xi(x)$.

%which coincides with the natural norm of $\boldsymbol{\boldsymbol{\mathcal{V}}}.$

%%%%%%%%%%%%%%%%%%%%%%%%%%%%%%%%%%%
%%%%%%%%%%%%%%%%%%%%%%%%%%%%%%%%%%%
%%%%%%%%%%%%%%%%%%%%%%%%%%%%%%%%%%%
%%%%%%%%%%%%%%%%%%%%%%%%%%%%%%%%%%%
%%%%%%%%%%%%%%%%%%%%%%%%%%%%%%%%%%%
\subsection{Symmetric and nonsymmetric DG schemes}

With the discrete spaces  previously defined, we introduce the discrete counterpart of Problem \ref{prob:continuous} as follows: Find $\lambda_h\in\mathbb{C}$ and $(\boldsymbol{0},0)\neq(\bu_h,p_h)\in\boldsymbol{\mV}_h\times\mathcal{Q}_h$ such that
\begin{equation}\label{eq:weak_stokes_system_dg}
\begin{array}{rcll}
 a_h(\bu_h,\bv_h)+b_h(\bv_h,p_h) & = & \lambda_h(\bu_h,\bv_h)_{0,\O} &\forall\ \bv_h\in \boldsymbol{\boldsymbol{\mathcal{V}}}_h,\\
\ds b_h(\bu_h,q_h) & = & 0 &\forall\ q_h\in \mathcal{Q}_h,
\end{array}
\end{equation}
where the continuous  sesquilinear form $a_h:\boldsymbol{\mV}_h\times\boldsymbol{\mV}_h\rightarrow\mathbb{C}$ is defined, for all $(\bu_h,\bv_h)\in\boldsymbol{\mV}_h\times\boldsymbol{\mV}_h$, by $a_h(\bu_h,\bv_h):=a^{\mathbb{K}}(\bu_h,\bv_h)+a_h^{\nabla}(\bu_h,\bv_h)$, where 
\begin{equation*}
a^{\mathbb{K}}(\bu_h,\bv_h):=\int_{\Omega} \mathbb{K}^{-1} \bu_h \cdot \bv_h,\qquad\forall (\bu_h,\bv_h)\in \boldsymbol{\mV}_h\times\boldsymbol{\mV}_h,
\end{equation*}
and
\begin{multline}
\label{eq:Ah}
a_h^{\nabla}(\bu_h,\bv_h):=\int_{\O}\nu\nabla_h\bu_h:\nabla_h\bv_h+ \int_{\cF^*_h}\frac{\texttt{a}_S}{h_{\cF}}\nu\jump{\bu_h}:\jump{\bv_h}\\
- \int_{\cF^*_h}\mean{\nu\nabla_h\bu_h}:\jump{\bv_h}
-\varepsilon \int_{\cF^*_h}\mean{\nu\nabla_h\bv_h}:\jump{\bu_h},\qquad\forall (\bu_h,\bv_h)\in \boldsymbol{\mV}_h\times\boldsymbol{\mV}_h.
\end{multline}
In \eqref{eq:Ah} the parameter  $\texttt{a}_S>0$, which is commonly named as the \textit{stabilization parameter},  is independent of the mesh size and will have an important influence on the computation of the spectrum as we will see in the forthcoming analys and more precisely on the numerical tests. On the other hand, the parameter $ \varepsilon\in\{-1,0,1\}$  dictates if the IPDG methods result to be  symmetric or non-symmetric. More precisely, 
if  $\varepsilon=-1$ we obtain the non-symmetric interior penalty method (NIP)
and if $\varepsilon=0$ the incomplete interior penalty method (IIP). 
%The difference between these IPDG methods will be observed on the numerical section.

It is easy to check that  $a_h(\cdot,\cdot)$ is a continuous  sesquilinear form, i.e., there exists a positive constant $C^{\star}$ such that for all $\bv_h,\bw_h\in\boldsymbol{\mathcal{V}}(h)$ there holds
\begin{equation}
\label{eq:bound_ah}
|a_h(\bv_h,\bw_h)|=|a_h^{\nabla}(\bv_h,\bw_h)+a^{\mathbb{K}}(\bv_h,\bw_h)|\leq C^{\star}\|\bv_h\|_{\boldsymbol{\mathcal{V}}(h)}\|\bw_h\|_{\boldsymbol{\mathcal{V}}(h)},
\end{equation}
where $C^{\star}:=\max\{\nu_{\max},\texttt{a}_S\nu_{\max}, |\varepsilon|\overline{C}_\nu,\overline{C}_{\nu},\mathbb{K}^*\}$.

Finally we define the bounded sesquilinear form   $b_h:\boldsymbol{\mathcal{V}}_h\times\mathcal{Q}_h\rightarrow \mathbb{C}$  by
$$
b_h(\bv_h,q_h):=-\int_{\O}\div_h\bv_h q_h+\int_{\cF_h^*}\mean{q_h}\jump{\underline{\bv_h}},\quad\forall\bv\in\boldsymbol{\mathcal{V}}_h,\,\forall q_h\in\mathcal{Q}_h.
$$

Defining the sesquilinear for $A_h(\cdot,\cdot)$ by 
\begin{equation*}
A_h((\bu_h,p_h),(\bv_h,q_h)):= a_h(\bu_h,\bv_h)+b_h(\bv_h,p_h) +b_h(\bu_h,q_h),
\end{equation*}
 for all $(\bu_h,p_h),(\bv_h,q_h)\in\boldsymbol{\mathcal{V}}_h\times \mathcal{Q}_h$, we rewrite \eqref{eq:weak_stokes_system_dg} as follows: Find $\lambda_h\in\mathbb{C}$ and $(\bu_h,p_h)\in\boldsymbol{\mathcal{V}}_h\times \mathcal{Q}_h$
such that
\begin{equation*}
A_h((\bu_h,p_h),(\bv_h,q_h))=\lambda_h(\bu_h,\bv_h)\quad\forall(\bv_h,q_h)\in\boldsymbol{\mathcal{V}}_h\times \mathcal{Q}_h.
\end{equation*}
Observe that $A_h(\cdot,\cdot)$ is a bounded sesquilinear form due to the boundedness of $a_h(\cdot,\cdot)$ and $b_h(\cdot,\cdot)$.

Let us recall the following discrete inf-sup condition \cite[Proposition 10]{MR1886000}: there exists a constant $\widehat{\beta}>0$, independent of $h$, such that 
\begin{equation}
\label{eq:disc_infsup}
\displaystyle\sup_{\btau_h\in\boldsymbol{\mV}_h}\frac{b_h(\btau_h,q_h)}{\|\btau_h\|_{\boldsymbol{\boldsymbol{\mathcal{V}}}(h)}}\geq\widehat{\beta}\|q_h\|_{0,\O}\quad\forall q_h\in\mathcal{Q}_h.
\end{equation}

On the other hand, let us define the discrete kernel $\boldsymbol{\mathcal{K}}_h$ of $b_h(\cdot,\cdot)$ as follows
\begin{equation*}
\boldsymbol{\mathcal{K}}_h:=\{\btau_h\in\boldsymbol{\mV}_h\,:\,  b_h(\btau_h,\bv_h)=0\,\,\,\,\forall\bv_h\in\boldsymbol{\boldsymbol{\mathcal{V}}}_h\}.
\end{equation*}

With this kernel at hand, now we prove that  $a_h(\cdot,\cdot)$ is  $\boldsymbol{\mathcal{K}}_h$- coercive.
\begin{lemma}[ellipticity of $a_h(\cdot,\cdot)$]
\label{lmm:disc_ellipticity}
For any $\varepsilon\in\{-1,0,1\}$, there exists a positive parameter $\texttt{a}^*$ such that for all $\texttt{a}_S\geq \texttt{a}^*$ there holds
$$a_h(\bv_h,\bv_h)\geq \widehat{\alpha}\|\bv_h\|_{\boldsymbol{\boldsymbol{\mathcal{V}}}(h)}^2\quad\forall\bv_h\in \boldsymbol{\mathcal{K}}_h,$$
where $\widehat{\alpha}>0$ is independent of $h$.
\end{lemma}
\begin{proof}
Let $\bv_h\in\boldsymbol{\mathcal{K}}_h$. Then, we  have $a_h(\bv_h,\bv_h)=a^{\mathbb{K}}(\bv_h,\bv_h)+a_h^{\nabla}(\bv_h,\bv_h)$. We observe that the estimate for $a_h^{\nabla}(\bv_h,\bv_h)$ is proved in \cite[Lemma 1]{MR4623018}. In fact, for $\tau>0$ we have that
$$a_h^{\nabla}(\bv_h,\bv_h)\geq\nu^*\underbrace{\left(1-C_t\frac{(1+\varepsilon)}{2\tau}\right)}_{C_1}\|\nabla_h\bv_h\|_{0,\O}^2+\nu^*\underbrace{\left(\texttt{a}_S-\frac{\tau(1+\varepsilon)}{2}\right)}_{C_2}\|h_{\cF}^{-1/2}\jump{\bv_h}\|_{0,\cF_h^*}^2,$$
where $C_t>0$ is the constant provided by \eqref{eq:dg_trace}. According to  \cite[Lemma 1]{MR4623018},  $C_1>0$ if and only if $\tau$ is such that  $\tau>(1+\varepsilon)C_t/2$. On the other hand,  $\texttt{a}_S$ is chosen in such a way that $\texttt{a}_S>\texttt{a}^*:=(1+\varepsilon)\tau/2$, where $\tau$ has been previously determined. Now, for $a^{\mathbb{K}}(\cdot,\cdot)$ we have $a^{\mathbb{K}}(\bv_h,\bv_h)\geq\displaystyle\mathbb{K}_{*}\|\bv_h\|_{0,\O}^2$. Hence, $a_{h}(\cdot,\cdot)$ is $\boldsymbol{\mathcal{K}}$-elliptic with constant $\widehat{\alpha}:=\nu^*\min\{C_1,C_2,\mathbb{K}_{*}\}$. This concludes the proof.
\end{proof}

As a consequence of Lemma \ref{lmm:disc_ellipticity} and the definition of $A_h(\cdot,\cdot)$, is possible to conclude that for all $(\bv_h,q_h)\in\boldsymbol{\mathcal{K}}_h\times\mathcal{Q}_h$ there holds
\begin{equation*}
A_h((\bv_h,q_h),(\bv_h,q_h))\geq \widehat{\alpha}\|\bv_h\|_{\boldsymbol{\mathcal{V}}(h)}^2,
\end{equation*}
where $\widehat{\alpha}>0$ is the constant previously determined.

Its is important to observe that the constant $\widehat{\alpha}$ of Lemma \ref{lmm:disc_ellipticity} depends on the choice of 
the penalization parameter $\texttt{a}_S$ which depends on the geometry of $\Omega$ and the physical parameters. This means that the well posedness of the source problem at discrete level  is directly related to the configuration of the these parameters.
 
Let us  introduce the discrete solution operators for the approximations of the velocity and pressure.  More precisely, we define these operators in the following manner
\begin{align*}
\bT_h&:\boldsymbol{\L}^2(\O)\rightarrow\boldsymbol{\mV}_h,\,\,\,\bF\mapsto\bT_h\bF:=\widetilde{\bu}_h, \\
\mathcal{S}_h&:\boldsymbol{\L}^2(\O)\rightarrow\mathcal{Q}_h,\,\,\,\bF\mapsto\mathcal{S}_h\bF:=\widetilde{p}_h.
\end{align*}
%\begin{equation*}
%\bT_h:\boldsymbol{\mV}\rightarrow\boldsymbol{\mV}_h,\quad\boldsymbol{f}\mapsto\bT_h\boldsymbol{f}:=\widetilde{\bu}_h,
%\end{equation*}
where  the pair $(\widetilde{\bu}_h,\widetilde{p}_h)\in\boldsymbol{\boldsymbol{\mathcal{V}}}_h\times\mathcal{Q}_h$ corresponds to the solution of the following discrete source problem
\begin{equation}\label{eq:weak_stokes_system_dg_source}
\begin{array}{rcll}
 a_h(\widetilde{\bu}_h,\bv_h)+b_h(\bv_h,\widetilde{p}_h) & = & (\boldsymbol{f},\bv_h)_{0,\O} &\forall\ \bv_h\in \boldsymbol{\boldsymbol{\mathcal{V}}}_h,\\
\ds b_h(\widetilde{\bu}_h,q_h) & = & 0 &\forall\ q_h\in \mathcal{Q}_h,
\end{array}
\end{equation}
which equivalently is written as the following problem
\begin{equation*}
A_h((\widetilde{\bu}_h,\widetilde{p}_h),(\bv_h,q_h))=(\boldsymbol{f},\bv_h)_{0,\O}\quad\forall(\bv_h,q_h)\in\boldsymbol{\mathcal{V}}_h\times \mathcal{Q}_h,
\end{equation*}

From \eqref{eq:disc_infsup} and Lemma \ref{lmm:disc_ellipticity}, together with the Babu\v ska-Brezzi theory, we conclude that \eqref{eq:weak_stokes_system_dg_source} has a unique solution $(\widetilde{\bu}_h,\widetilde{p}_h)\in\boldsymbol{\boldsymbol{\mathcal{V}}}_h\times\mathcal{Q}_h$  and as a consequence,  operators $\bT_h$ and $\mathcal{S}_h$ are
well defined. 

For the solutions of the continuous and discrete source problems, the following C\'ea estimate holds
\begin{equation*}
\label{eq:cea}
\|(\widetilde{\bu}-\widetilde{\bu}_h,\widetilde{p}-\widetilde{p}_h)\|_{\boldsymbol{\mathcal{V}}(h)\times\mathcal{Q}_h}\\
%\lesssim \left(1+\frac{C_*}{\beta}\right)\inf_{(\widetilde{\bv}_h,\widetilde{q}_h)\in\boldsymbol{\mathcal{V}}_h\times\mathcal{Q}_h}\|(\widetilde{\bu},\widetilde{p})-(\widetilde{\bv}_h,\widetilde{q}_h)\|_{\boldsymbol{\mathcal{V}}(h)\times\mathcal{Q}_h}
\leq C\inf_{(\widetilde{\bv}_h,\widetilde{q}_h)\in\boldsymbol{\mathcal{V}}_h\times\mathcal{Q}_h}\|(\widetilde{\bu},\widetilde{p})-(\widetilde{\bv}_h,\widetilde{q}_h)\|_{\boldsymbol{\mathcal{V}}(h)\times\mathcal{Q}_h},
\end{equation*}
where $C$, depends on the continuity constant of $a_h(\cdot,\cdot)$ and the discrete inf-sup constant given in \eqref{eq:disc_infsup}.
Then we have the following error convergence result for the source problem (see \cite[Corollary 6.22]{MR2882148})
 \begin{equation*}
\|\widetilde{\bu}-\widetilde{\bu}_h\|_{\mathcal{V}(h)}+\|\widetilde{p}-\widetilde{p}_h\|_{0,\O}\leq C h^s(\|\widetilde{\bu}\|_{1+s}+\|\widetilde{p}\|_{s})\leq C h^s\|\boldsymbol{f}\|_{0,\O},
\end{equation*}
where $C$ depends on the physical constants, but is independent of $h$.

\section{Convergence and error estimates}
\label{sec:conv_error}
Now our goal is to  prove of convergence of the IPDG scheme and error estimates. To do this task, we resort to the theory of \cite{MR483400,MR483401}, due to the non-conformity of the proposed methods. let us introduce some preliminary definitions and notations. We denote by $\norm{\cdot}_{\mathcal{L}(\boldsymbol{\boldsymbol{\mathcal{V}}}(h),\boldsymbol{\boldsymbol{\mathcal{V}}}(h))}$
the corresponding norm acting from $\boldsymbol{\boldsymbol{\mathcal{V}}}(h)$ into the same space.
In addition, we will denote by $\norm{\cdot}_{\mathcal{L}(\boldsymbol{\boldsymbol{\mathcal{V}}}_h,\boldsymbol{\boldsymbol{\mathcal{V}}}(h))}$
the norm of an operator restricted to the discrete subspace $\boldsymbol{\boldsymbol{\mathcal{V}}}_h$;
namely, if $\boldsymbol{L}:\boldsymbol{\boldsymbol{\mathcal{V}}}(h)\to \boldsymbol{\boldsymbol{\mathcal{V}}}(h)$, then
\begin{equation*}
\norm{\boldsymbol{L}}_{\mathcal{L}(\boldsymbol{\boldsymbol{\mathcal{V}}}_h, \boldsymbol{\boldsymbol{\mathcal{V}}}(h))}:=\sup_{\0\neq\btau_h\in\boldsymbol{\boldsymbol{\mathcal{V}}}_h}
\frac{\norm{\boldsymbol{L}\btau_h}_{\boldsymbol{\boldsymbol{\mathcal{V}}}(h)}}{\norm{\btau_h}_{\boldsymbol{\boldsymbol{\mathcal{V}}}(h)}}.
\end{equation*}

According to \cite{MR483400}, to establish spectral correctness we need  to prove the following properties
\begin{itemize}
\item P1. $\norm{\bT-\bT_h}_{\mathcal{L}(\boldsymbol{\boldsymbol{\mathcal{V}}}_h, \boldsymbol{\boldsymbol{\mathcal{V}}}(h))}\to0$ as $h\to0$.
\item P2. $\forall\btau\in\boldsymbol{\boldsymbol{\mathcal{V}}}$, there holds
$$\inf_{\btau_h\in \boldsymbol{\boldsymbol{\mathcal{V}}}_h}\norm{\btau-\btau_h}_{\boldsymbol{\boldsymbol{\mathcal{V}}}(h)}\,\,\to0\quad\text{as}\quad h\to0.$$
\end{itemize}
We observe that property P2 is immediate as a consequence of the density of continuous piecewise degree $k$ polynomial functions in $\boldsymbol{\boldsymbol{\mathcal{V}}}$. On the other hand, P1 is not direct, and our goal is to prove it.

\subsection{Convergence}
The IPDG methods considered for the  Stokes-Brinkman eigenvalue problem does not differ from the Stokes eigenvalue problem studied in \cite{MR4623018} except for the permeability term associated to the Brinkman equations. This implies that the convergence analysis for the Stokes-Brinkman  eigenvalue problem is the same as in \cite{MR4623018}. However, and for completeness of the analysis, we summarize the results that are possible to derive. The first result that is needed is the following (see \cite[Lemma 2]{MR4623018}).
\begin{lemma}
\label{lmm:TTh}
For all $\boldsymbol{f}\in \boldsymbol{\mathcal{V}}$, we define  $\widetilde{\bu}:=\bT\boldsymbol{f}$ and $\widetilde{p}:=\mathcal{S}\boldsymbol{f}$ as the solutions of \eqref{eq:porous_system_continuous_source}, whereas   $\widetilde{\bu}_h:=\bT_h\boldsymbol{f}$ and $\widetilde{p}_h:=\mathcal{S}_h\boldsymbol{f}$ are the solutions of \eqref{eq:weak_stokes_system_dg_source}. Then, the following estimate holds
\begin{equation*}
\|(\bT-\bT_h)\boldsymbol{f}\|_{\boldsymbol{\boldsymbol{\mathcal{V}}}(h)}\leq C h^{\min\{s,k\}}\|\boldsymbol{f}\|_{0,\O},
\end{equation*}
where $s>0$ and the hidden constant is independent of $h$ and $C>0$  depends on the physical constants.
\end{lemma}

This result is directly extended for discrete sources as is stated in \cite[Corollary 1]{MR4623018}. 
%This is a key point, since the theory of non-compact operators \cite{MR483400} requieres precisely this type of sources to prove property P1. 

\begin{corollary}
\label{lmm:P1}
There  following estimate holds 
\begin{equation*}\label{P1}
\norm{\boldsymbol{T}-\bT_h}_{\mathcal{L}(\boldsymbol{\boldsymbol{\mathcal{V}}}_h, \boldsymbol{\boldsymbol{\mathcal{V}}}(h))}\leq C h^{\min\{s,k\}},
\end{equation*}
where the hidden constant is independent of $h$.
\end{corollary}
%\begin{proof}
%The proof is derived applying Lemma \ref{lmm:TTh} but considering $\boldsymbol{f}_h\in\boldsymbol{\mathcal{V}}_h$.
%\end{proof}

The following result indicates a gain of one additional order in the approximation of the error in the $\L^2$ norm for the solution operators $\bT$ and $\bT_h$. The proof follows the 
techniques in  \cite{MR2085402}. Moreover, for the proof we will assume, $s>1/2$ which is perfectly reasonable, according to the results presented in \cite{MR4231512,MR3145765}.
\begin{theorem}\label{thm_mejor_orden}
Let $\O$ be a convex domain. Let us assume $s>1/2$. Then, under the hypotheses of Lemma \ref{lmm:TTh}, there exists a constant $C$, independent of $h$,  such that
$$\|\widetilde{\bu}-\widetilde{\bu}_h\|_{0,\O}\leq Ch^{2\min\{s,k\}}(|\widetilde{\bu}|_{1+s,\O}+|\widetilde{p}|_{s,\O})\leq Ch^{2\min\{s,k\}}\|\boldsymbol{f}\|_{0,\O},$$
where $C>0$ depends on the physical constants and independent of $h$.
\end{theorem} 
\begin{proof}
Let us  define the following dual problem: Find $(\boldsymbol{\varphi},\psi)\in \boldsymbol{\H}_{\Gamma_1}(\Omega)\times\L^2(\O)$ such that
\begin{align*}
	\bbK^{-1}\boldsymbol{\varphi} - \nu \Delta \boldsymbol{\varphi} + \nabla \psi & = \widetilde{\bu}-\widetilde{\bu}_h & \text{in $\Omega$},\\
	\div\boldsymbol{\varphi} &= 0 & \text{in $\Omega$},\\
	\boldsymbol{\varphi} &= \boldsymbol{0}& \text{on $\Gamma_1$},  \\
	(\nu\nabla \boldsymbol{\varphi} - \psi\mathbb{I}) \bn   &=\boldsymbol{0} & \text{on $\Gamma_2$}
\end{align*}
Clearly the above solution satisfies the following data dependence
$$\|\boldsymbol{\varphi}\|_{1+s,\O}+\|\psi\|_{s,\O}\leq C_{\nu,\bbK}\|\widetilde{\bu}-\widetilde{\bu}_h\|_{0,\O},$$
which holds for $s>1/2$. 

On the other hand, let us consider  the following identity
\begin{equation*}
\|\widetilde{\bu}-\widetilde{\bu}_h\|_{0,\O}^2=\sum_{T\in\CT_h}\int_T\left(\bbK^{-1}\boldsymbol{\varphi} - \nu \Delta \boldsymbol{\varphi} + \nabla \psi\right)\cdot(\widetilde{\bu}-\widetilde{\bu}_h),
%=\sum_{T\in\CT_h}\left(\int_T\bbK^{-1}\boldsymbol{\varphi}\cdot(\widetilde{\bu}-\widetilde{\bu}_h)+\int_T\nu\nabla\boldsymbol{\varphi}:\nabla(\widetilde{\bu}-\widetilde{\bu}_h)-\int_T\div(\widetilde{\bu}-\widetilde{\bu}_h)\psi\right)\\
%\sum_{T\in\CT_h}\left(\int_{\partial T}(\widetilde{\bu}-\widetilde{\bu}_h)\cdot \boldsymbol{n}\psi-\int_{\partial T}\nu\nabla \boldsymbol{\varphi}\boldsymbol{n}\cdot(\widetilde{\bu}-\widetilde{\bu}_h)\right)
\end{equation*}
applying  integration by parts, and using that $(\boldsymbol{\varphi},\psi)\in \boldsymbol{\H}^{1+s}(\Omega)\times\H^s(\Omega)$, we have 
\begin{multline}\label{eq_control}
\|\widetilde{\bu}-\widetilde{\bu}_h\|_{0,\O}^2\\%=%\sum_{T\in\CT_h}\int_T\left(\bbK^{-1}\boldsymbol{\varphi} - \nu \Delta \boldsymbol{\varphi} + \nabla \psi\right)\cdot(\widetilde{\bu}-\widetilde{\bu}_h)\\
=\sum_{T\in\CT_h}\left(\int_T\bbK^{-1}\boldsymbol{\varphi}\cdot(\widetilde{\bu}-\widetilde{\bu}_h)+\int_T\nu\nabla\boldsymbol{\varphi}:\nabla(\widetilde{\bu}-\widetilde{\bu}_h)-\int_T\div(\widetilde{\bu}-\widetilde{\bu}_h)\psi\right)\\
%\sum_{T\in\CT_h}\int_{\partial T}(\widetilde{\bu}-\widetilde{\bu}_h)\cdot \boldsymbol{n}\psi-\int_{\partial T}\nu\nabla \boldsymbol{\varphi}\boldsymbol{n}\cdot(\widetilde{\bu}-\widetilde{\bu}_h)\\
+\int_{\cF^*_h}\mean{\psi}\jump{\underline{\widetilde{\bu}-\widetilde{\bu}_h}}-\int_{\cF^*_h}\mean{\nu\nabla_h\boldsymbol{\varphi}}:\jump{\widetilde{\bu}-\widetilde{\bu}_h}.
\end{multline}

Let us denote by $\Xi_h:\L^2(\O)\to\mathcal{Q}_h$  the classic $\L^2$-orthogonal projection and let $\boldsymbol{\Pi}_h\boldsymbol{\varphi}\in\boldsymbol{\H}^{1}(\Omega)$ be the Lagrange interpolator  of $\boldsymbol{\varphi}$. Now, testing Problem \ref{eq:porous_system_continuous_source} and Problem \ref{eq:weak_stokes_system_dg_source} with $(\boldsymbol{\Pi}_h\boldsymbol{\varphi},\Xi_h\psi)\in \boldsymbol{\H}^{1}(\Omega)\times\L^2(\Omega)$ we have the following identity
$$A_h((\widetilde{\bu}-\widetilde{\bu}_h,\widetilde{p}-\widetilde{p}_h),(\boldsymbol{\Pi}_h\boldsymbol{\varphi},\Xi_h\psi))=0.$$
By subtracting the above identity to \eqref{eq_control}, we obtain
\begin{multline}
\|\widetilde{\bu}-\widetilde{\bu}_h\|_{0,\O}^2=\underbrace{\sum_{T\in\CT_h}\int_T\bbK^{-1}(\boldsymbol{\varphi}-\boldsymbol{\Pi}_h\boldsymbol{\varphi})\cdot(\widetilde{\bu}-\widetilde{\bu}_h)}_{T_{1}}\\
+\underbrace{\sum_{T\in\CT_h}\int_T\nu\nabla(\boldsymbol{\varphi}-\boldsymbol{\Pi}_h\boldsymbol{\varphi}):\nabla(\widetilde{\bu}-\widetilde{\bu}_h)}_{T_{2}}
+\underbrace{\sum_{T\in\CT_h}\int_T\div(\widetilde{\bu}-\widetilde{\bu}_h)(\Xi_h\psi-\psi)}_{T_{3}}\\
+\underbrace{\sum_{T\in\CT_h}\int_T(\widetilde{p}-\widetilde{p}_h)\div \boldsymbol{\Pi}_h\boldsymbol{\varphi}}_{T_{4}}+\underbrace{\sum_{F\in\cF^*_h}\int_F\mean{\psi-\Xi_h\psi}\jump{\underline{\widetilde{\bu}-\widetilde{\bu}_h}}}_{T_{5}}\\
+\underbrace{(\varepsilon-1)\sum_{F\in\cF^*_h}\int_F\mean{\nu\nabla_h(\boldsymbol{\Pi}_h\boldsymbol{\varphi}-\boldsymbol{\varphi})}:\jump{\widetilde{\bu}-\widetilde{\bu}_h}}_{T_{6}}.
\end{multline}
The following task is to bound  each of the   terms $T_i$, for $ i=1,2,...,6$. 
First we note that to bound $T_4$ it is necessary that:
$$T_4=\sum_{T\in\CT_h}\int_T(\widetilde{p}-\widetilde{p}_h)\div( \boldsymbol{\Pi}_h\boldsymbol{\varphi}-\boldsymbol{\varphi})=\sum_{T\in\CT_h}\int_T(\widetilde{p}-\Xi_h\widetilde{p})\div( \boldsymbol{\Pi}_h\boldsymbol{\varphi}-\boldsymbol{\varphi}).$$
Using the approximation properties of $I_h$ and the projector $\Xi_h$, together with the regularity of the dual problem and the convergence order of the operators, we obtain that:
$$T_1+T_2+T_3+T_4\leq \max\{\bbK^*,\nu,1,C_{\nu,\bbK}\}h^{2\min\{s,k\}}(|\widetilde{\bu}|_{s,\O}+|\widetilde{p}|_{s,\O})\|\widetilde{\bu}-\widetilde{\bu}_h\|_{0,\O}.$$
To estimate $T_5$, it is necessary to use trace inequality and the properties of approximation 
\begin{multline}
T_5\leq \sum_{F\in\cF^*_h}\|h_F^{1/2}\jump{\psi-\Xi_h\psi}\|_{0,F}\|h_F^{-1/2}\jump{\widetilde{\bu}-\widetilde{\bu}_h}\|_{0,F}\leq h^s|\psi|_{s,\O}\|\widetilde{\bu}-\widetilde{\bu}_h\|_{{\boldsymbol{\boldsymbol{\mathcal{V}}}(h)}}\\
\leq C_{\nu,\bbK} h^{2\min\{s,k\}}\|\widetilde{\bu}-\widetilde{\bu}_h\|_{0,\O}(|\widetilde{\bu}|_{1+s,\O}+|\widetilde{p}|_{s,\O}).
\end{multline}
For $T_6$ we proceed analogously
$$
T_6\leq  C_{\nu,\bbK} |\varepsilon-1|h^{2\min\{s,k\}}\|\widetilde{\bu}-\widetilde{\bu}_h\|_{0,\O}|\widetilde{\bu}|_{1+s,\O}.
$$
Thus, substituting in \eqref{eq_control}, all the estimates obtained for bounding the different $T_i$ with $i = 1, 2, 3, 4, 5,6$ completes the proof.
\end{proof}

The goal now is to establish that the numerical schemes are spurious free. To do this task, first we recall the definition of the resolvent operator of $\bT$ and $\bT_h$ respectively:
\begin{gather*}
	(z\bI-\bT)^{-1}\,:\, \boldsymbol{\mathcal{V}} \to \boldsymbol{\mathcal{V}}\,, \quad z\in\mathbb{C}\setminus \sp(\bT), \\
	(z\bI-\bT_h)^{-1}\,:\, \boldsymbol{\mathcal{V}}_h \to \boldsymbol{\mathcal{V}}_h\,, \quad z\in\mathbb{C}\setminus\sp(\bT_h) .
\end{gather*}

 Let $\mathcal{D}$ denote the unit disk in the complex plane, defined as   $\mathcal{D}:=\{z\in\mathbb{C}\,:\, |z|\leq 1\}$. According to \cite[Lemma 4]{MR4623018}, the resolvent $(z \bI - \bT) \boldsymbol{f}$ is correctly bounded in the $\boldsymbol{\mathcal{V}}(h)$ norm,  in the sense that  exists a constant $C>0$ independent of $h$
such that for all  $z \in\mathcal{D}\setminus \sp(\bT)$  there holds
\[
\norm{(z \bI - \bT) \boldsymbol{f}}_{\boldsymbol{\mathcal{V}}(h)} \geq C|z|\,
\norm{\boldsymbol{f}}_{\boldsymbol{\mathcal{V}}(h)} \quad \forall \boldsymbol{f} \in \boldsymbol{\mathcal{V}}(h).\]
Moreover, on a compact subset $E$ of $\mathcal{D}\setminus \sp(\bT)$, the resolvent is invertible and bounded, i.e.,  for all  $z\in E$, there exists a constant $C>0$ such that 
\begin{equation*}\label{resDisc}
\displaystyle\norm{(z\bI-\bT)^{-1}}_{\mathcal{L}(\boldsymbol{\mathcal{V}}(h),\boldsymbol{\mathcal{V}}(h))}\leq C\qquad\forall z\in E.
\end{equation*}

On the other hand, the discrete resolvent is also bounded for sufficiently small values of  $h$ as stated in \cite[Lemma 5]{MR4623018}. More precisely, if $z \in\mathcal{D}\setminus \sp(\bT)$, there exist $h_0>0$  and $C>0$ independent of $h$ but depending  on $|z|$, such that for all $h\leq h_0$
\[
\norm{(z \bI - \bT_h) \boldsymbol{f}}_{\boldsymbol{\mathcal{V}}(h)} \geq C \, \norm{\boldsymbol{f}}_{\boldsymbol{\mathcal{V}}(h)} \quad \forall \boldsymbol{f}\in \boldsymbol{\mathcal{V}}(h),
\]

As we mention of the continuous resolvent, if $E$ is a compact subset 
of the complex plane such that $E\cap\sp(\bT)=\emptyset$ for $h$
small enough and for all $z\in E$,  there exists a positive constant
$C$ independent of $h$ such that $\norm{(z \bI - \bT_h)^{-1}}_{\mathcal{L}(\boldsymbol{\mathcal{V}}(h),\boldsymbol{\mathcal{V}}(h))}\leq C$
for all $z\in E$.  Hence, with all these ingredients at hand, we conclude that for $h$  small enough.
the numerical schemes are  spurious free.  This is summarized in the following result proved in \cite{MR483400}.
\begin{theorem}\label{free}
Let $E\subset\mathbb{C}$ be a compact subset not intersecting $\sp(\bT)$.
Then, there exists $h_0>0$ such that, if $h\leq h_0$, then $E\cap\sp(\bT_h)=\emptyset.$ 
\end{theorem}

 \subsection{A priori error estimates}
First we introduce the definition of the \textit{gap} $\hdel$ between two closed
subspaces $\CM$ and $\CN$ of $\L^2(\O)$:
$$
\hdel(\CM,\CN)
:=\max\big\{\delta(\CM,\CN),\delta(\CN,\CM)\big\},
$$
where
$$
\delta(\CM,\CN)
:=\sup_{x\in\CM:\ \left\|x\right\|_{0,\O}=1}
\left(\inf_{y\in\CN}\left\|x-y\right\|_{0,\O}\right).
$$

Let $\lambda$ be an isolated eigenvalue of $\bT$
and let $D$ be any open disk in the complex plane with boundary $\gamma$
such that $\lambda$ is the only eigenvalue of $\bT$ lying in $D$ and $\gamma\cap\sp(\bT)=\emptyset$.
We introduce the spectral projector corresponding to the continuous
and discrete solution operators $\bT$ and $\bT_h$, respectively
\begin{equation*}
\bcE:=\frac{1}{2\pi i}
\int_{\gamma}\left(z\bI-\bT\right)^{-1}\, dz:\boldsymbol{\mathcal{V}}(h)\longrightarrow \boldsymbol{\mathcal{V}}(h),
\end{equation*}
\begin{equation*}
\bcE_h:=\frac{1}{2\pi i}
\int_{\gamma}\left(z\bI-\bT_h\right)^{-1}\, dz:\boldsymbol{\mathcal{V}}(h)\longrightarrow \boldsymbol{\mathcal{V}}(h).
\end{equation*}
%where $\bcE_h$ is well-defined and bounded uniformly in 
%$h$ due to \eqref{resDisc}. 

The following approximation result for the spectral projections holds is derived according to \cite[Theorem 5.1]{MR2220929}.
\begin{lemma}
\label{eq:E-E_h}
There holds
\begin{equation*}
	\lim_{h\rightarrow 0}\|\bcE-\bcE_h\|_{\mathcal{L}(\boldsymbol{\mathcal{V}}_h,\boldsymbol{\mathcal{V}}(h))}=0.
\end{equation*}
\end{lemma}

We end this section with the a priori error estimates for the eigenfunctions and eigenvalues. These estimates depend on the IPDG under consideration, in the sense that for non-symmetric methods ($\varepsilon\in\{0,-1\}$) the orders of convergence are not optimal, whereas for the symmetric method ($\varepsilon=1$) the order of convergence is quadratic. We begin by recalling \cite[Lemma 7]{MR4623018}, which is straightforward for our eigenvalue problem.

\begin{lemma}
\label{lmm:error_eigenfunctions}
There exists a strictly positive constant $h_0$ such that, for $h<h_0$ there holds
\begin{equation*}
\label{eq:error_eigenfunction}
\widehat{\delta}_h(\bcE(\boldsymbol{\mathcal{V}}),\bcE_h(\boldsymbol{\mathcal{V}}_h))\leq C h^{\min\{r,k\}},
\end{equation*}
where $r>0$ is the same as in \eqref{eq_regularityofev} and  the hidden constant $C>0$ is independent of $h$.
\end{lemma}
\begin{remark}
\label{rmrk:11}
It is important to remark that this lemma, rigorously speaking, the proof of this result lies in the fact that the IPDG methods are consistent in the sense that for all $(\bv_h,q_h)\in\boldsymbol{\mathcal{V}}_h\times\mathcal{Q}_h$, the following identity holds
$$A_h((\bu-\bu_h,p-p_h),(\bv_h,q_h))=0,$$
with $(\bu,p)\in  \boldsymbol{\H}^{1+r}(\O)\times \H^r(\O)$ and $r>0$ and hence, the following C\'ea estimate for the eigenfunctions holds
\begin{equation*}
\|(\bu-\bu_h,p-p_h)\|_{\boldsymbol{\mathcal{V}}(h)\times\mathcal{Q}_h}\leq C\left(1+\frac{C_*}{\beta}\right)\inf_{(\bv_h,q_h)\in\boldsymbol{\mathcal{V}}_h\times\mathcal{Q}_h}\|(\bu,p)-(\bv_h,q_h)\|_{\boldsymbol{\mathcal{V}}(h)\times\mathcal{Q}_h},
\end{equation*}
where $C_*$ and $\beta$ are the continuity constant of $a_h(\cdot,\cdot)$ and the inf-sup constant of $b_h(\cdot,\cdot)$, respectively. Now, applying any suitable interpolant for $\bu$ on $\boldsymbol{\mathcal{V}}_h$ and the $\L^2$ orthogonal projection operator, together with the fact that $\bu\in\bcE(\boldsymbol{\mathcal{V}})\subset \boldsymbol{\H}^{1+r}(\O)$ with $r>0$, we have
\begin{equation*}
\|(\bu-\bu_h,p-p_h)\|_{\boldsymbol{\mathcal{V}}(h)\times\mathcal{Q}_h} \leq Ch^{\min\{r,k\}}(\|\bu\|_{1+r}+\|p\|_{r,\O}),
\end{equation*}
where $C$ is a constant depending on the physical constants and the corresponding eigenvalue.
\end{remark}

\begin{theorem}\label{theorem_ordendoble}
	There exists a strictly positive constant $h_0$ such that, for $h<h_0$ there holds
	\begin{enumerate}
	\item If the symmetric IPDG method is considered $(\varepsilon=1)$, then there holds
	\begin{equation}
	\label{eq:double_order}
		|\lambda-\lambda_h|\leq C h^{2\min\{r,k\}},
	\end{equation}
	\item If any of the non-symmetric IPDG methods are  considered $(\varepsilon\in\{-1,0\})$, then there holds
		\begin{equation}
		\label{eq:simple_order}
		|\lambda-\lambda_h|\leq C  h^{\min\{r,k\}},
	\end{equation}
	\end{enumerate}
	%where in each estimate the hidden constant is independent of $h$.
where $r > 0$ and $C$ is a constant depending on the physical constants and the corresponding eigenvalue given in \eqref{eq_regularityofev}.
\end{theorem}
\begin{proof}
We begin by noticing that  \eqref{eq:simple_order} is an immediate consequence of Lemma \ref{lmm:error_eigenfunctions}. The estimate \eqref{eq:double_order} follows from the well known algebraic   identity
\begin{equation}
\label{eq:padra}
A_h((\bu-\bu_h,p-p_h),(\bu-\bu_h,p-p_h))-\lambda(\bu-\bu_h,\bu-\bu_h)_{0,\O}=(\lambda_{h}-\lambda)(\bu_h,\bu_h)_{0,\O}.
\end{equation}

It is straightforward to prove that  there exists a positive constant  $\widetilde{C}$ such that  $(\bu_h,\bu_h)_{0,\O}>\widetilde{C}>0$. On the other hand, applying modulus on \eqref{eq:padra} we obtain
\begin{multline*}
\widetilde{C}|\lambda_{h}-\lambda|\leq |A_h((\bu-\bu_h,p-p_h),(\bu-\bu_h,p-p_h))|+\lambda|(\bu-\bu_h,\bu-\bu_h)_{0,\O}|\\
=|a_h(\bu-\bu_h,\bu-\bu_h)+2b_h(\bu-\bu_h,p-p_h)|+|\lambda||(\bu-\bu_h,\bu-\bu_h)_{0,\O}|\\
\leq |a_h(\bu-\bu_h,\bu-\bu_h)|+2|b_h(\bu-\bu_h,p-p_h)|+|\lambda||(\bu-\bu_h,\bu-\bu_h)_{0,\O}|\\
\leq C^*\|\bu-\bu_h\|_{\boldsymbol{\mathcal{V}}(h)}^2+2\|\bu-\bu_h\|_{\boldsymbol{\mathcal{V}}(h)}\|p-p_h\|_{0,\O}+C\|\bu-\bu_h\|_{0,\O}^2,\\
\leq \max\{C^*,C,1\}(\|\bu-\bu_h\|_{\boldsymbol{\mathcal{V}}(h)}^2+\|p-p_h\|_{0,\O}^2),
\end{multline*}
where the constant $C^*>0$ is the one involved in \eqref{eq:bound_ah}. Finally, the proof follows from Remark \ref{rmrk:11}.
\end{proof}

\section{A posteriori error analysis}
\label{sec:apost}
The aim of this section is to introduce a suitable   fully computable residual-based error in the sense that it depends only on quantities available from the DG solution. Then, we will show its equivalence with the error. The analysis is focused only on eigenvalues with simple multiplicity.

For $T\in\CT_h$, we introduce the local indicator $\boldsymbol{\eta}_T$ as follows
\begin{multline*}
\boldsymbol{\eta}_T^2:= h_T^2\|\lambda_h\bu_h+\nu\Delta\bu_h-\mathbb{K}^{-1}\bu_h-\nabla p_h\|_{0,T}^2+\|\div\bu_h\|_{0,T}^2\\
+\frac{h_F}{2}\sum_{F\in\cF_h^0}\|\jump{\nu\nabla\bu_h-p_h\mathbb{I}}\boldsymbol{n}\|_{0,F}^2 +\frac{h_F}{2}\sum_{F\in\cF_h^2}\Vert(\nu\nabla\bu_h-p_h\mathbb{I})\boldsymbol{n}\|_{0,F}^2\\
+\frac{h_F^{-1}}{2}\sum_{F\in\cF_h^0}\|\nu\jump{\bu_h}\|_{0,F}^2 +\frac{h_F^{-1}}{2}\sum_{F\in\cF_h^1}\|\nu\bu_h \otimes\bn\|_{0,F}^2.
\end{multline*}
We introduce the global a posteriori error estimator
\begin{equation*}\label{estimadorglobal}
\boldsymbol{\eta}=\left(\sum_{T\in\CT_h}\boldsymbol{\eta}_{T}^2\right)^{1/2}
\end{equation*}

In what follows, let $(\lambda,(\bu,p))$ be a solution to Problem \ref{prob:continuous}. We assume, for simplicity, that $\lambda$ is a simple eigenvalue. Let us consider  that $\|\bu\|_{0,\O}=1$. Then, for each $\CT_h$, there exists a solution $(\lambda_h,(\bu_h,p_h))$ of problem \eqref{eq:weak_stokes_system_dg} such that $\lambda_h \to \lambda$, $\|\bu_h\|_{0,\O}=1$ and $\|\bu-\bu_h\|_{\boldsymbol{\mathcal{V}}(h)}\to 0$ as $h\to 0$.

\subsection{Reliability}\label{sec_confiabilidad}
Following the approach presented in \cite{Paul2005347}, we decompose the space of discontinuous finite elements by defining $\boldsymbol{\mathcal{V}}_h^c:=\boldsymbol{\mathcal{V}}_h\cap \boldsymbol{\H}_{\Gamma_1}(\Omega)$. The orthogonal complement of 
$\boldsymbol{\mathcal{V}}_h^c$  in $\boldsymbol{\mathcal{V}}_h$ with respect to the norm $\|\cdot\|_{1,h}$   is denoted by $\boldsymbol{\mathcal{V}}_h^r$, where the norm   is defined as:
\begin{equation*}
\|\bv_h\|_{1,h}:=\|\nabla_h\bv_h\|_{0,\O}^2+\|h_{\cF}^{-1/2}\jump{\bv_h}\|_{0,\mathcal{F}_h}^2.
\end{equation*}
Then, we have the decomposition $\boldsymbol{\mathcal{V}}_h=\boldsymbol{\mathcal{V}}_h^c\oplus\boldsymbol{\mathcal{V}}_h^r$, allowing us to uniquely decompose the DG velocity approximation into $\bv_h=\bv_h^c+\bv_h^r$, where $\bv_h^c\in \boldsymbol{\mathcal{V}}_h^c$ and $\bv_h^r\in \boldsymbol{\mathcal{V}}_h^r$.
The following auxiliary result is necessary to prove  that the presented estimator is reliable and has to do with the Scott-Zhang quasi-interpolator operator (see \cite{MR1011446}).
\begin{lemma}\label{lemmaSZ}
Let $\boldsymbol{\mathcal{I}}_h:\boldsymbol{\H}^1(\Omega)\to\boldsymbol{\mathcal{V}}_h^c$ be the  Scott-Zhang quasi-interpolator operator. Then, there exists a constant $C_{SZ}>0$ independent of $h$ such that
\begin{equation*}
\sum_{T\in\CT_h}\left(h_{T}^{-2}\|\bv-\boldsymbol{\mathcal{I}}_h\bv\|_{0,T}^2+\|\nabla(\bv-\boldsymbol{\mathcal{I}}_h\bv)\|_{0,T}^2+h_{T}^{-1}\|\bv-\boldsymbol{\mathcal{I}}_h\bv\|_{0,\partial T}\right)\leq C_{SZ}\|\nabla\bv\|_{0,\O}^2.
\end{equation*}
\end{lemma}

For what follows, it is necessary to obtain an upper bound for $\|\bv_h^r\|_{\boldsymbol{\mathcal{V}(h)}}$. This is achieved using the Poincaré inequality (see \cite[Corollary 5.4]{MR2882148}), the estimate presented in \cite[Proposition 4.1 ]{Paul2005347}, and the definition of the local estimator $\boldsymbol{\eta}_T$ which yield to 
%. In this way, the following estimate is obtained:}The following result, is consequence of Poincar\'e inequality \cite[Corollary 5.4]{MR2882148} and \cite[Proposition 4.1]{Paul2005347} we derive the upper bound for $\|\bv_h^r\|_{\boldsymbol{\mathcal{V}(h)}}$:
\begin{equation}\label{eq_cotasup}
\|\bv_h^r\|_{\boldsymbol{\mathcal{V}(h)}}\leq C_p\nu^{-1}\left(\sum_{T\in\CT_h}\boldsymbol{\eta}_T^2\right)^{1/2},
\end{equation}
where $C_p>0$ is the  Poincar\'e constant.

The following result constitutes the main result on the efficiency of our estimator.
\begin{theorem}\label{eficiencia}
Let $(\lambda,(\bu,p))\in \mathbb{R}\times \boldsymbol{\H}_{\Gamma_1}(\Omega)\times \mathrm{L}^2(\Omega)$ be a solution of Problem \ref{prob:continuous} and $(\lambda_h,(\bu_h,p_h))\in \mathbb{C}\times \boldsymbol{\boldsymbol{\mathcal{V}}}_h\times\mathcal{Q}_h$ its mixed DG solution of problem \eqref{eq:weak_stokes_system_dg} with $\|\bu_h\|_{0,\O}=1$. Then, there exists a constant $C>0$ independent of $h$ such that
\begin{equation*}
\vertiii{(\bu-\bu_h,p-p_h)}_h\leq C\left(\boldsymbol{\eta}+|\lambda|\|\bu-\bu_h\|_{0,\Omega}+|\lambda-\lambda_h|\right),
\end{equation*}  
where $\vertiii{(\bv,q)}_h^2=\|\bv\|_{\boldsymbol{\mathcal{V}}(h)}^2+\|q\|_{0,\O}.$
\end{theorem}
\begin{proof}
Let $(\lambda,(\bu,p))\in \mathbb{R}\times \boldsymbol{\H}_{\Gamma_1}(\Omega)\times \mathrm{L}^2(\Omega)$ be a solution of Problem \ref{prob:continuous} and $(\lambda_h,(\bu_h,p_h))\in \mathbb{C}\times \boldsymbol{\boldsymbol{\mathcal{V}}}_h\times\mathcal{Q}_h$  solution of problem \eqref{eq:weak_stokes_system_dg} with $\|\bu_h\|_{0,\O}=1$, we will decompose $\bu_h$ as $\bu_h=\bu_h^c+\bu_h^r$. Thus using triangular inequality and \eqref{eq_cotasup}, we have that
\begin{multline}\label{eq_1confibalilidad}
\vertiii{(\bu-\bu_h,p-p_h)}_{h}=\vertiii{(\bu-\bu_h^c,p-p_h)}_h+\|\bu_h^r\|_{\boldsymbol{\mathcal{V}}(h)}\\\leq \vertiii{(\bu-\bu_h^c,p-p_h)}_h+C_p\nu^{-1}\left(\sum_{T\in\CT_h}\boldsymbol{\eta}_T^2\right)^{1/2}.
\end{multline}

The next step is to bound the first term of the right hand side of the previous inequality, To this end, we first note that since $\bu-\bu_h^c\in \boldsymbol{\H}_{\Gamma_1}(\Omega)$ and $p-p_h\in \L^2(\O)$, it follows that 
$\vertiii{(\bu-\bu_h^c,p-p_h)}_{h}=\vertiii{(\bu-\bu_h^c,p-p_h)}$. Therefore, we can invoke Lemma \ref{lemma:elliptic} which gives us a function $(\bv,q)\in \boldsymbol{\H}_{\Gamma_1}(\Omega)\times \mathrm{L}^2(\Omega)$ such that $\vertiii{(\bv,q)}_h\leq C_1 \vertiii{(\bu-\bu_h^c,p-p_h)}_h$, where $ C_1> 0$ is the constant involved in Lemma \ref{lemma:elliptic}. Then,  we have 
\begin{multline}\label{eq_2confibalilidad}
C_2^{-1}\vertiii{(\bu-\bu_h^c,p-p_h)}_h^2\leq A((\bu-\bu_h^c,p-p_h),(\bv,q))\\=A((\bu,p),(\bv,q))-A((\bu_h^c,p_h)(\bv,q))\\ 
=\lambda(\bu,\bv)_{0,\O}-\lambda_h(\bu_h,\bv)_{0,\O}+\lambda_h(\bu_h,\bv)_{0,\O}-A((\bu_h,p_h),(\bv,q))+A((\bu_h^r,p_h),(\bv,q)),
\end{multline}
we note that if we define $\bv_h^c:=	\boldsymbol{\mathcal{I}}_h\bv\in \boldsymbol{\mathcal{V}}_h^c$ as the Scott-Zhang quasi-interpolation of $\bv$,  and testing problem \eqref{eq:weak_stokes_system_dg} with $(\bv_h^c,0)$, we obtain that
$$A_h((\bu_h,p_h),(\bv_h^c,0))=\lambda_h(\bu_h,\bv_h^c)_{0,\O}.$$
Therefore, adding and subtracting $\lambda_h(\bu_h,\bv_h^c)_{0,\O}$ in inequality \eqref{eq_2confibalilidad} and using the above equality, we have that
\begin{multline*}
C_2^{-1}\vertiii{(\bu-\bu_h^c,p-p_h)}_h^2\leq  (\lambda\bu-\lambda_h\bu_h,\bv)_{0,\O}+\lambda_h(u_h,\bv-\bv_h^c)_{0,\O}\\
+A_h((\bu_h,p_h),(\bv_h^c,0))-A((\bu_h,p_h),(\bv,q))+A((\bu_h^r,p_h),(\bv,q)).
\end{multline*}

Now, using the fact that $\bv_h^c\in \boldsymbol{\H}_{\Gamma_1}(\Omega)$, applying the definitions of $A(\cdot,\cdot)$ and $A_h(\cdot,\cdot)$, and \cite[Proposition 4.1]{Paul2005347}, the previous inequality can be rewritten as follows
\begin{equation}\label{eq_3confibalilidad}
C_2^{-1}\vertiii{(\bu-\bu_h^c,p-p_h)}_h^2\leq B_1+B_2+B_3+B_4+B_5,
\end{equation}
where 
\begin{equation*}
B_1:=(\lambda\bu-\lambda_h\bu_h,\bv)_{0,\O},\quad B_2:=-\varepsilon \int_{\cF^*_h}\mean{\nu\nabla_h\bv_h}:\jump{\bu_h},
\end{equation*}
%\begin{equation*}
%B_2:=-\varepsilon \int_{\cF^*_h}\mean{\nu\nabla_h\bv_h}:\jump{\bu_h},
%\end{equation*}
\begin{equation*}
B_3:=\int_{\O}\div_h \bu_h q+\int_{\O}\nu\nabla_h\bu_h^r:\nabla\bv+\int_{\Omega} \mathbb{K}^{-1} \bu_h^r \cdot \bv,\quad B_4:=\lambda_h(\bu_h,\bv-\bv_h^c)_{0,\O},
\end{equation*}
and
\begin{align*}
B_5&=-\int_{\Omega} \mathbb{K}^{-1} \bu_h \cdot (\bv-\bv_h^c)-\int_{\O}\nu\nabla_h\bu_h:\nabla(\bv-\bv_h^c)\\
&+\int_{\O}\div_h (\bv-\bv_h^c)p_h+\int_{\O}\div_h \bu_h q.
\end{align*}

The task now is to estimate each of the terms on the right hand side of \eqref{eq_3confibalilidad}. For the term $B_1$ we apply a trace estimate in conjunction with a discrete inverse inequality to an edge or face $F\in\cF^*_h$, where $F=T_1\cap T_2$ if $F\in\cF^0_h$ and $F=T_1\cap \partial \O$ with $T_2=\emptyset$ if $F\in \cF^{\partial}_h$, we obtain
$$\|\nabla_h\bv_h^c\|_{0,F}\leq C h_F^{-1/2}\|\nabla_h\bv_h^c\|_{0,T_1\cup T_2}.$$
Thus, using the stability of the Scott-Zhang quasi-interpolator (cf.  Lemma \ref{lemmaSZ}), we obtain that:
\begin{multline*}
B_2\leq C_{SZ}C_{\nu}C_p|\varepsilon|\left(\sum_{T\in\CT_h}\|\nabla\bv\|_{0,T}^2\right)^{1/2}\left(\sum_{T\in\CT_h}\boldsymbol{\eta}_T^2\right)^{1/2}\\
\leq C_{SZ}C_{\nu}C_p|\varepsilon|\left(\sum_{T\in\CT_h}\boldsymbol{\eta}_T^2\right)^{1/2}\vertiii{(\bu-\bu_h,p-p_h)}_{h}.
\end{multline*}

Let us focus on $B_3$.  Applying the  Cauchy-Schwarz inequality and \eqref{eq_cotasup} we obtain
\begin{multline*}
B_3\leq \max\{\mathbb{K}^*,\nu,1\}\|\bu_h^r\|_{\boldsymbol{\mathcal{V}}(h)}\vertiii{(\bv,q)}\\
\leq C_1\max\{\mathbb{K}^*,\nu,1\} \|\bu_h^r\|_{\boldsymbol{\mathcal{V}}(h)}\vertiii{(\bu-\bu_h,p-p_h)}_{h}\\
\leq C_1\max\{\mathbb{K}^*,\nu,1\}C_p\nu^{-1}\left(\sum_{T\in\CT_h}\boldsymbol{\eta}_T^2\right)^{1/2}\vertiii{(\bu-\bu_h,p-p_h)}_{h}.
\end{multline*}

Finally, we proceed to bound the last two terms $B_4$ and $B_5$. From  integration by parts and the approximation properties of the Scott-Zhang interpolator we obtain the following result
\begin{multline*}
B_4+B_5\\= \sum_{T\in\CT_h}\left(\int_T\left(\lambda_h\bu_h-\mathbb{K}^{-1} \bu_h )\cdot(\bv-\bv_h^c)\right)-\int_T \nu\nabla_h\bu_h:\nabla(\bv-\bv_h^c):(\bv-\bv_h^c)\right.\\
\left.+\int_{T}\div_h (\bv-\bv_h^c)p_h+\int_{T}\div_h \bu_h q\right)\\
=\sum_{T\in\CT_h}\left(\int_T\left[\lambda_h\bu_h-\mathbb{K}^{-1} \bu_h +\nu\Delta_h \bu_h + \nabla_h p_h\cdot(\bv-\bv_h^c)\right]+\int_{T}\div_h \bu_h q\right)\\
+\sum_{T\in\CT_h}\left(\int_{\partial T}(\nu \nabla_h\bu_h- p_h\mathbb{I})\boldsymbol{n}\cdot(\bv-\bv_h^c)\right)
\leq \left(\sum_{T\in\CT_h}\boldsymbol{\eta}_T^2\right)^{1/2}\vertiii{(\bu-\bu_h,p-p_h)}_{h},\\
\end{multline*}
where for the last inequality, we have used that $\bv_h^c=\boldsymbol{\mathcal{I}}_h\bv$ and the approximation properties of the Scott-Zhang quasi-interpolator (see Lemma \ref{lemmaSZ}).

Thus, replacing in \eqref{eq_2confibalilidad}, all the estimates obtained to bound the different $B_i$ with $i=1,2,3,4,5$ and then replacing this result in \eqref{eq_1confibalilidad}, the property is proved.
\end{proof}
We observe that from the proof of Theorem \ref{theorem_ordendoble} and the previous theorem, we obtain the following result
\begin{corollary}
There exists a constant $C>0$ independent of $h$ such that 
$$|\lambda-\lambda_h|\leq C\left(\boldsymbol{\eta}+|\lambda-\lambda_h|+\|\bu-\bu_h\|_{0,\Omega}\right)^2.$$
\end{corollary}

\subsection{Efficiency}\label{eficiencia_s}
In order to perform the efficiency analysis, we adopt standard arguments widely recognized in the literature, which pertain to the behavior of bubble functions (see \cite{Lepe2024,Gedicke2020585,verfuhrt1996}). Given that these approaches have been thoroughly documented and validated in prior works, the detailed proof is omitted here, and only the main result is presented.

\begin{theorem}
Let $(\lambda,(\bu,p))\in \mathbb{R}\times \boldsymbol{\H}_{\Gamma_1}(\Omega)\times \mathrm{L}^2(\Omega)$ be a solution of Problem \ref{prob:continuous} and $(\lambda_h,(\bu_h,p_h))\in \mathbb{C}\times \boldsymbol{\boldsymbol{\mathcal{V}}}_h\times\mathcal{Q}_h$ its mixed DG solution of problem \eqref{eq:weak_stokes_system_dg}. Then, there exists a constant $C>0$ independent of $h$ such that
\begin{equation*}
\boldsymbol{\eta}\leq C\vertiii{(\bu-\bu_h,p-p_h)}_{h}+ h.o.t.
\end{equation*}
where $h.o.t:=\left(\sum_{T\in\CT_h}h_T^2\left(|\lambda-\lambda_h|^2+|\lambda|^2\|\bu-\bu_h\|_{0,\O}^2\right)\right)^{1/2}.$
\end{theorem}

\section{Numerical experiments}
\label{sec:numerics}
This section is dedicated to conducting various numerical experiments to assess the performance of the scheme across different geometries and physical configurations. The implementations are using the DOLFINx software \cite{barrata2023dolfinx,scroggs2022basix}, where the SLEPc eigensolver \cite{hernandez2005slepc} and the MUMPS linear solver are employed to solve the resulting generalized eigenvalue problem. Meshes are generated using GMSH \cite{geuzaine2009gmsh} and the built-in generic meshes provided by DOLFINx. The convergence rates for each eigenvalue are determined using least-squares fitting and highly refined meshes. 

In what follows, we denote the mesh resolution by $N$, which is connected to the mesh-size $h$ through the relation $h\sim N^{-1}$. We also denote the number of degrees of freedom by $\texttt{dof}$. The relation between $\texttt{dof}$ and the mesh size is given by $h\sim\texttt{dof}^{-1/d}$, with $d\in\{2,3\}$. 

Let us define $\err(\lambda_i)$ as the error on the $i$-th eigenvalue, with
$$
\err(\lambda_i):={\vert \lambda_{h,i}-\lambda_{i}\vert},
$$
where $\lambda_i$ is the extrapolated value. Similarly, the effectivity indexes with respect to $\eta$ and the eigenvalue $\lambda_{h,i}$ is defined by
$$
\eff(\lambda_i):=\frac{\err(\lambda_i)}{\boldsymbol{\eta}^2}.$$

In order to apply the adaptive finite element method, we shall generate a sequence of nested conforming triangulations using the loop
\begin{center}
	\textrm{solve $\rightarrow$ estimate $\rightarrow$ mark $\rightarrow$ refine,} 
\end{center}
based on \cite{verfuhrt1996}:
\begin{enumerate}
\item Set an initial mesh $\CT_{h}$.
\item Solve \eqref{eq:weak_stokes_system_dg} in the actual mesh to obtain $(\lambda_{h},(\bu_h,p_h))$. 	
\item Compute $\boldsymbol{\eta}_T$ for each $T\in\CT_{h}$ using the eigenfunctions $(\bu_h,p_h)$. 
\item Use Dörfler \cite{dorfler1996convergent} marking criterion to construct a subset $\mathcal{S}_h\subset\CT_{h}$ such that we refine the elements $T$ that satisfies
$$
\sum_{T\in \mathcal{S}_h}\boldsymbol{\eta}_{T}^2\geq \zeta\sum_{T\in \CT_{h}}\boldsymbol{\eta}_{T}^2,
$$
for some $\zeta\in(0,1)$.
\item Set $\CT_{h}$ as the actual mesh and go to step 2.
\end{enumerate}
For 2D experiments, we choose $\upsilon=0.6$, while $\upsilon=0.8$ is chosen for 3D test.

It is worth noting that, while the permeability tensor is theoretically assumed to be positive definite, in the numerical experiments, it is set close to $\boldsymbol{0}$. As a result, the eigenfunctions in these regions exhibit behavior consistent with the Stokes eigenvalue problem. For all experiments, we set $\nu=1$ and consider various choices for $\mathbb{K}$. The meshing of $\Omega$ into subdomains is such that there is conformity between the regions, i.e, the subregions are delimited exactly by the facets of the domain. 

The choice of the stabilization parameter is an important aspect in the correct prediction of the eigenvalues. Different studies \cite{MR2324460,MR3962898,MR4077220, MR4623018} have shown that taking a sufficiently large $\texttt{a}_S:=\texttt{a}k^2$ guarantees an accurate computation of the spectrum. 

\subsection{Stability analysis on a square domain with a porous subdomain} \label{subsec:square2D-one-obstacle}
Let us consider the domain $\Omega:=(0,1)^2$, $\Omega_D:=(3/8,5/8)^2$ and $\Omega_S:=\Omega\backslash\Omega_D$, which consists of the unit square domain with an internal, possibly porous subdomain $\Omega_D$. For each region, we define the following permeability parameter
$$
\mathbb{K}^{-1}=\left\{
\begin{aligned}
	&\kappa\mathbb{I}&\text{if } (x,y)\in\Omega_D,\\
	&\boldsymbol{0}, & \text{if } (x,y)\in\Omega_S.
\end{aligned}
\right.
$$
This choice determines a region of full permeability on $\Omega_S$, while a variable porosity is considered in $\Omega_D$. The idea of the experiment is to study the convergence of the DG scheme when $\kappa$ is changed. For the tests in this section, we will consider reference values computed using the method proposed in \cite{lepe2025jsc}. 
%10^-8 $   52.3447$ & $     92.1244$ & $    92.1244$ & $   128.2096 $
%10^3 &   $   65.3658$ & $    167.7481$ & $   182.6605$ & $   182.6605 $ \\
%10^8 &   $   75.7836$ & $    219.6318$ & $   226.6955$ & $   226.6964 $ \\

\subsubsection{Dependence on the stabilization parameter}\label{subsec:square2D-alpha-dependence} 
We start by analyzing what happens when we start moving the alpha stabilization parameter. According to Lemma \ref{lmm:disc_ellipticity}, we note that $\texttt{a}_S$ must be large enough to guarantee the stability of the method. A mesh resolution $N=16$, corresponding to $\texttt{dof}=4424$ for $k=1$ is selected.

First, we show the results for the first 40 eigenvalues computed with the three variants of the methods in Figure \ref{fig:square_alpha_dependence}. Here, it is noticeable that the symmetric method presents instability for values of $\texttt{a}\leq 2.5$, while the incomplete and non-symmetric methods can tolerate values closer to zero. The oscillations observed for small values of $\texttt{a}$, including veering and crossing between eigenvalues, are due to the eigensolver detecting spurious eigenvalues, which can be positive or negative, real or imaginary. Also, there is little to no difference in the stabilization of the schemes when choosing different permeability parameters. 

To further study stabilization, given $\texttt{a}>0$, we extract the eigenvalues computed by the solver and compare them with existing methods. In particular, we consider the numerical method given \cite{lepe2025jsc} using Taylor-Hood elements. The results for the values of $\texttt{a}\in\{0.5,3,10\}$ are shown in Figure \ref{fig:square_alfas_10em8}. As expected, for $\texttt{a}=0.5$ the symmetric method shows spurious eigenvalues, while the other methods show an underestimation in the prediction for this stabilization value. On the other hand, for larger values of alpha, the tendency is always towards overestimation, which can be observed in all permeability cases, although for $\texttt{a}=3$ the prediction is quite accurate. This selection, although it gives a small error with respect to the rest, does not guarantee that the convergence is optimal.  In conclusion, a safe parameter for all the cases is $\texttt{a}>3$. Big values of $\alpha$ will produce overprediction of eigenvalues, but it may help with convergence. In particular, considering \cite{MR4623018} as a reference, we note that a safe parameter for the method is $\texttt{a}\geq10$. We select this parameter for all experiments in the rest of the numerical section.

\subsubsection{Convergence of the DG schemes}\label{subsec:square2D_convergence}
This section analyzes the computational convergence of the proposed DG schemes when a safe stabilization parameter is given. All the cases consider $\mathrm{a}=10$. The reference values computed from \cite{lepe2025jsc} are shown in Table \ref{table:square2D-one-obstacle-reference-values} for the different permeability cases under study.

\begin{table}[t!]
	\setlength{\tabcolsep}{4.5pt}
	\centering 
	\caption{Example \ref{subsec:square2D-one-obstacle}. Lowest four reference eigenvalues for different permeability parameters, computed using the method from~\cite{lepe2025jsc}.}
	\label{table:square2D-one-obstacle-reference-values}
	{\small\begin{tabular}{|r|c|c|c|}
			\hline\hline
			&   $\mathbb{K}^{-1}\vert_{\Omega_D}$ =  $10^{-8}$ &   $\mathbb{K}^{-1}\vert_{\Omega_D}$ = $10^{3}$ &   $\mathbb{K}^{-1}\vert_{\Omega_D}$ = $10^{5}$\\%  &   $\mathbb{K}^{-1}\vert_{\Omega_D}$ = $10^{8}$  \\
			\hline
			$\lambda_1$ & 52.3447 & 65.3658 &74.4455\\%& 75.7836 \\
			$\lambda_2$ &  92.1244 & 167.7481 &214.1789\\%& 219.6318 \\
			$\lambda_3$ & 92.1244& 182.6605 &222.0403\\%& 226.6955 \\
			$\lambda_4$ & 128.2096& 182.6605 & 222.0352\\%& 226.6964 \\
			\hline
			\hline
	\end{tabular}}
	\smallskip
\end{table}

The absolute error and convergence behavior for the different schemes and permeability parameters are presented in Figures~\ref{fig:square_error_1em8}--\ref{fig:square_error_1e3}. The error history in Figure~\ref{fig:square_error_1em8}, which corresponds to the case $\mathbb{K}^{-1}|_{\Omega_D} = 10^{-8} \mathbb{I}$, shows optimal convergence for all values of $k$ when the symmetric scheme is used. A small perturbation is observed for $k=3$. However, for the non-symmetric methods with $k=2,3$, a convergence rate of order $\mathcal{O}(h^{2(k-1)})$ is observed, which reflects the suboptimal behavior predicted in Theorem~\ref{theorem_ordendoble}.

We also analyze convergence in the case of a semi-permeable zone by considering $\mathbb{K}^{-1}|_{\Omega_D} = 10^3 \mathbb{I}$ to study the maximum achievable convergence rate. From the results in Figure~\ref{fig:square_error_1e3}, we observe that for $k=1,2$, the methods behave similarly to the previous case. For $k=3$, however, the symmetric method yields only $\mathcal{O}(\texttt{dof}^{-2}) \approx \mathcal{O}(h^4)$, suggesting that the geometric regularity induces $\min\{r,k\}=r=2$. The non-symmetric methods exhibit behavior consistent with theoretical expectations.

Finally, although not shown here, we also studied the case $\mathbb{K}^{-1}|_{\Omega_D} = 10^5 \mathbb{I}$. In that case, convergence of order $\mathcal{O}(h^m)$ was observed, with $1.2 < m < 1.8$ for all values of $k$, which aligns with the expected behavior due to the obstacle effect induced by the subdomain $\Omega_D$ and the presence of reentrant corners within the domain $\Omega$. 
%We conclude the experiment by presenting the velocity and pressure eigenmodes corresponding to $\lambda_{h,1}$ in Figure~\ref{fig:square_modes}. The effect of the permeability tensor $\mathbb{K}$ on the subdomain $\Omega_D$ is clearly visible, with high pressure gradients observed for $\mathbb{K}^{-1}|_{\Omega_D} = 10^5 \mathbb{I}$.

\begin{figure}[!hbt]\centering
	\begin{minipage}{0.325\linewidth}\centering
		{\footnotesize $\varepsilon=1, \mathbb{K}^{-1}\vert_{\Omega_D}=10^{-8}$}\\
		\includegraphics[scale=0.17,trim=0cm 0cm 1.5cm 2.5cm, clip]{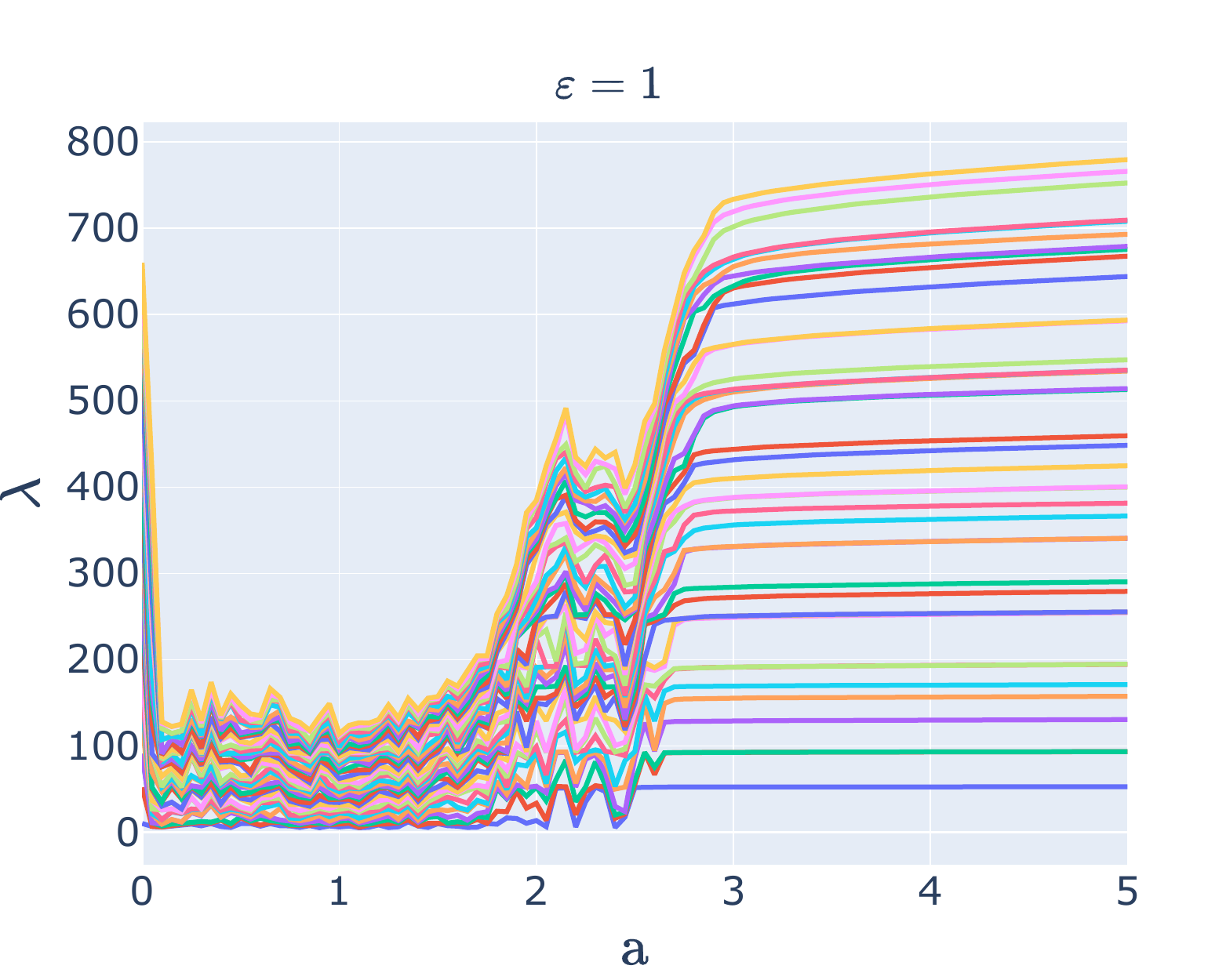}
	\end{minipage}
	\begin{minipage}{0.325\linewidth}\centering
		{\footnotesize $\varepsilon=1, \mathbb{K}^{-1}\vert_{\Omega_D}=10^3$}\\
		\includegraphics[scale=0.17,trim=0cm 0cm 1.5cm 2.5cm, clip]{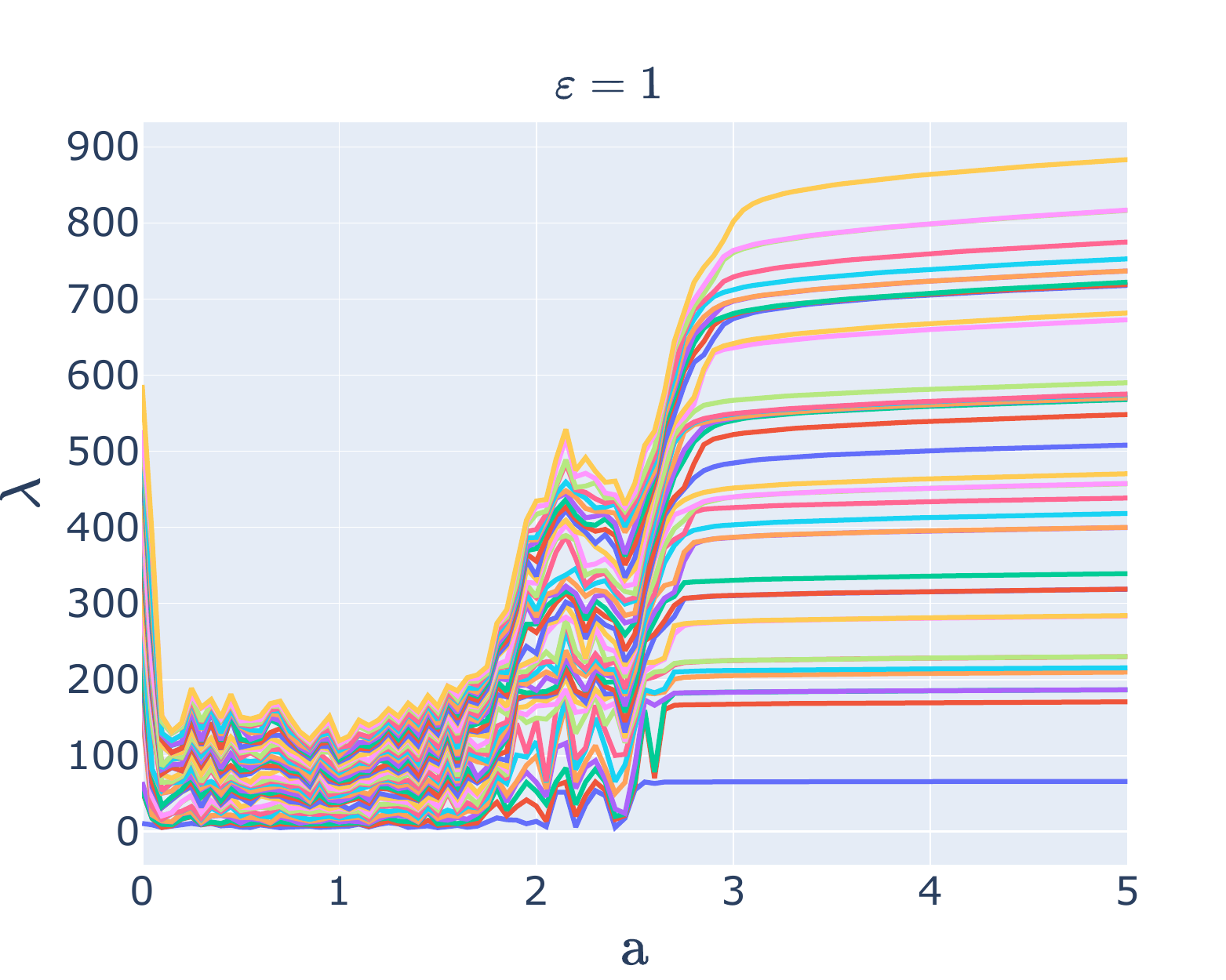}
	\end{minipage}
	\begin{minipage}{0.325\linewidth}\centering
		{\footnotesize $\varepsilon=1, \mathbb{K}^{-1}\vert_{\Omega_D}=10^5$}\\
		\includegraphics[scale=0.17,trim=0cm 0cm 1.5cm 2.5cm, clip]{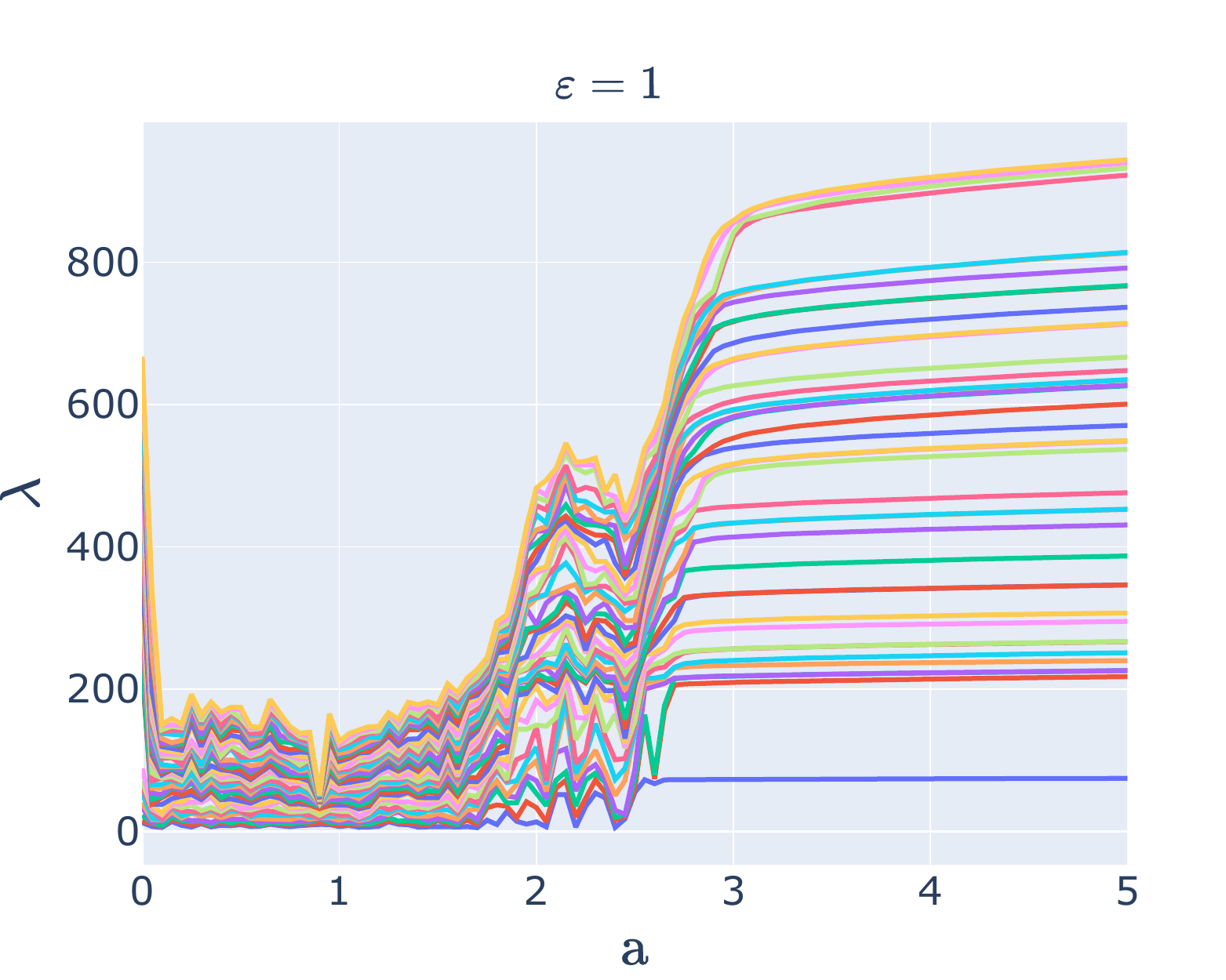}
	\end{minipage}\\
	\begin{minipage}{0.325\linewidth}\centering
		{\footnotesize $\varepsilon=0, \mathbb{K}^{-1}\vert_{\Omega_D}=10^{-8}$}\\
		\includegraphics[scale=0.17,trim=0cm 0cm 1.5cm 2.5cm, clip]{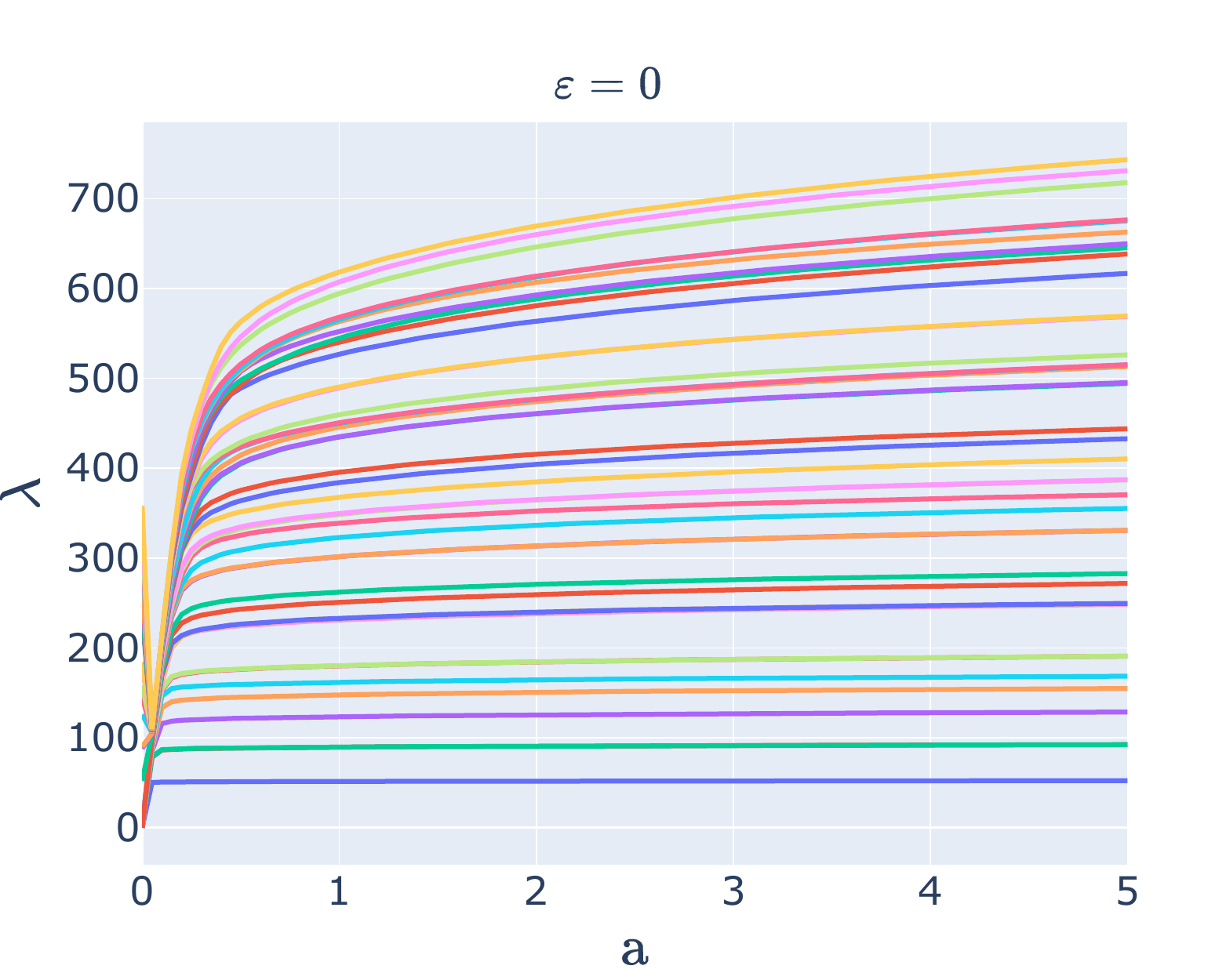}
	\end{minipage}
	\begin{minipage}{0.325\linewidth}\centering
		{\footnotesize $\varepsilon=0, \mathbb{K}^{-1}\vert_{\Omega_D}=10^3$}\\
		\includegraphics[scale=0.17,trim=0cm 0cm 1.5cm 2.5cm, clip]{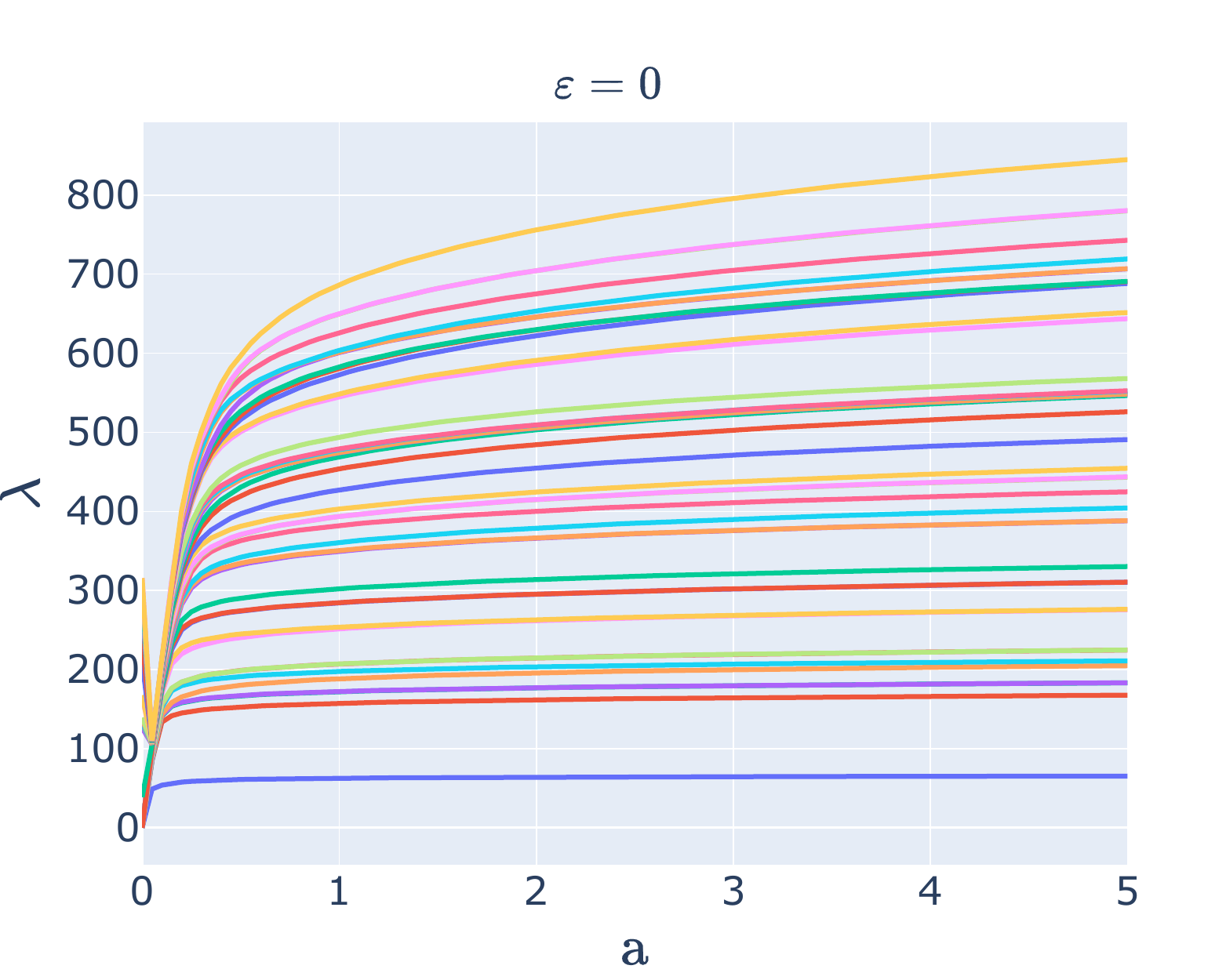}
	\end{minipage}
	\begin{minipage}{0.325\linewidth}\centering
		{\footnotesize $\varepsilon=0, \mathbb{K}^{-1}\vert_{\Omega_D}=10^5$}\\
		\includegraphics[scale=0.17,trim=0cm 0cm 1.5cm 2.5cm, clip]{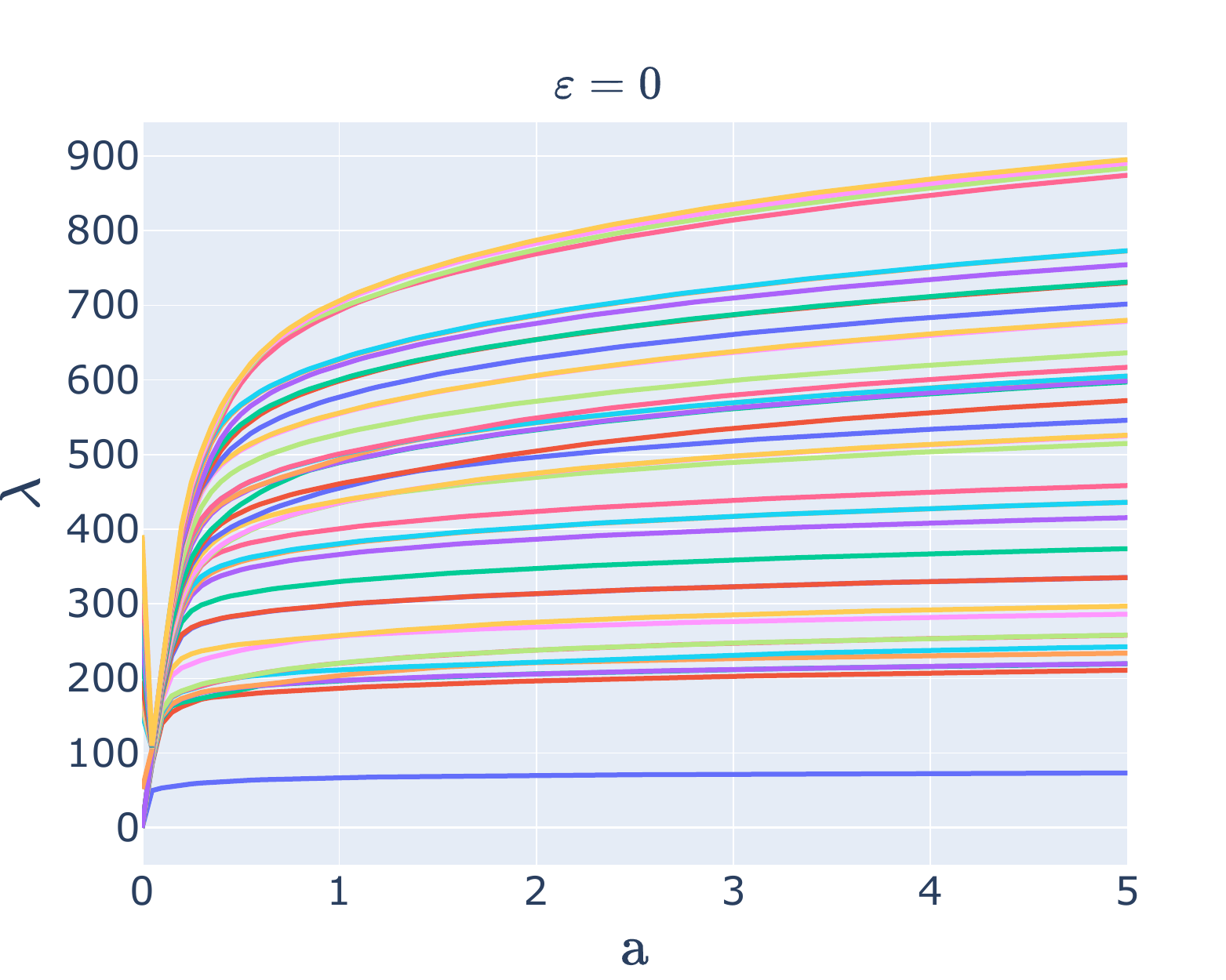}
	\end{minipage}\\
	\begin{minipage}{0.325\linewidth}\centering
		{\footnotesize $\varepsilon=-1, \mathbb{K}^{-1}\vert_{\Omega_D}=10^{-8}$}\\
		\includegraphics[scale=0.17,trim=0cm 0cm 1.5cm 2.5cm, clip]{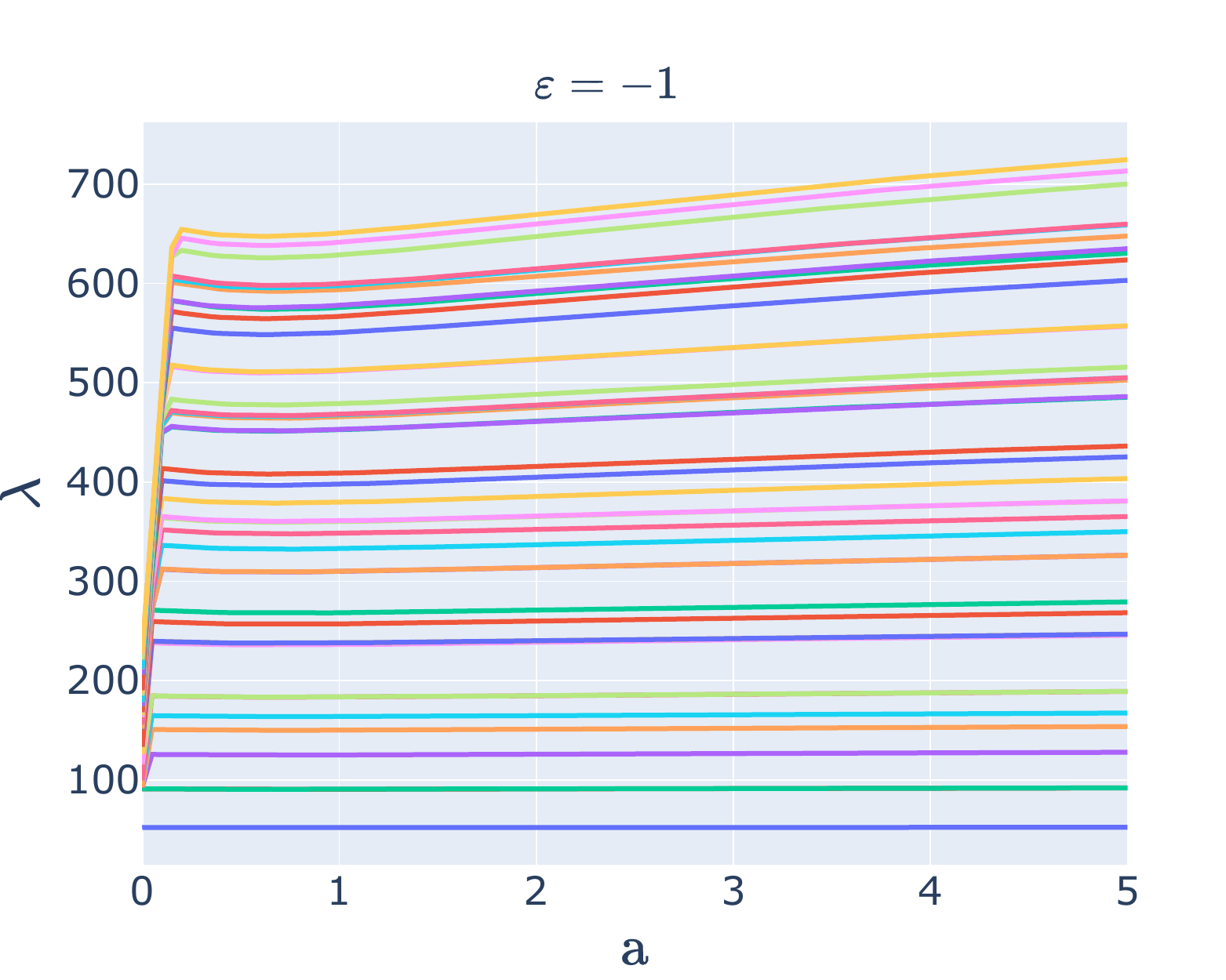}
	\end{minipage}
	\begin{minipage}{0.325\linewidth}\centering
		{\footnotesize $\varepsilon=-1, \mathbb{K}^{-1}\vert_{\Omega_D}=10^3$}\\
		\includegraphics[scale=0.17,trim=0cm 0cm 1.5cm 2.5cm, clip]{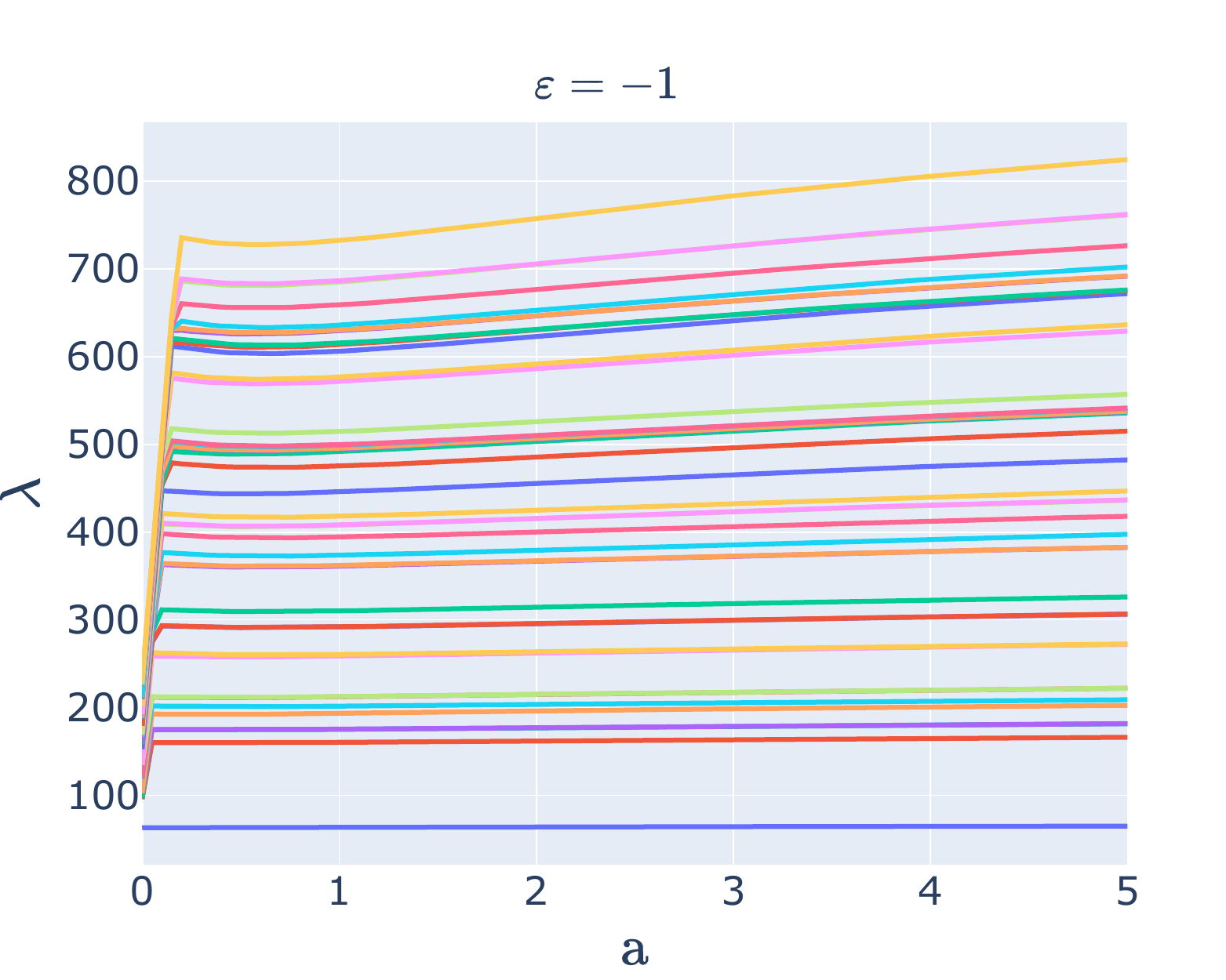}
	\end{minipage}
	\begin{minipage}{0.325\linewidth}\centering
		{\footnotesize $\varepsilon=-1, \mathbb{K}^{-1}\vert_{\Omega_D}=10^5$}\\
		\includegraphics[scale=0.17,trim=0cm 0cm 1.5cm 2.5cm, clip]{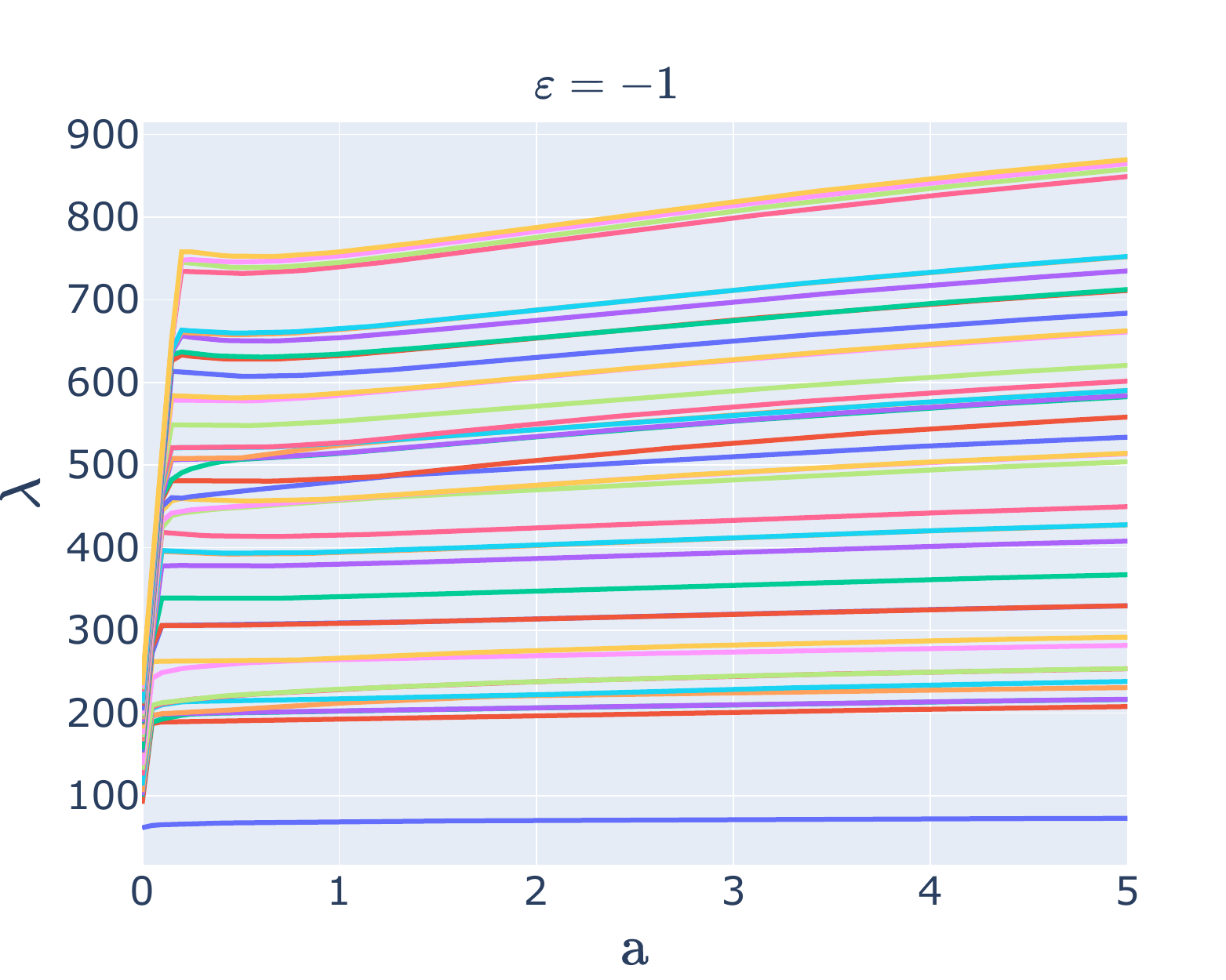}
	\end{minipage}
	\caption{Test \ref{subsec:square2D-alpha-dependence}. Dependence of the spectrum on the stabilization parameter $\mathrm{a}>0$ in the three IPDG variants ($\varepsilon \in \{1, 0, -1\}$), showing the first 40 lowest computed eigenvalues for different values of $\mathbb{K}$.}
	\label{fig:square_alpha_dependence}
\end{figure}

\begin{figure}[!hbt]\centering
\begin{minipage}{0.325\linewidth}\centering
	{\footnotesize $\mathbb{K}^{-1}\vert_{\Omega_D}=10^{-8}\mathbb{I}$}
	\includegraphics[scale=0.172, trim= 0cm 0cm 2cm 2cm,clip]{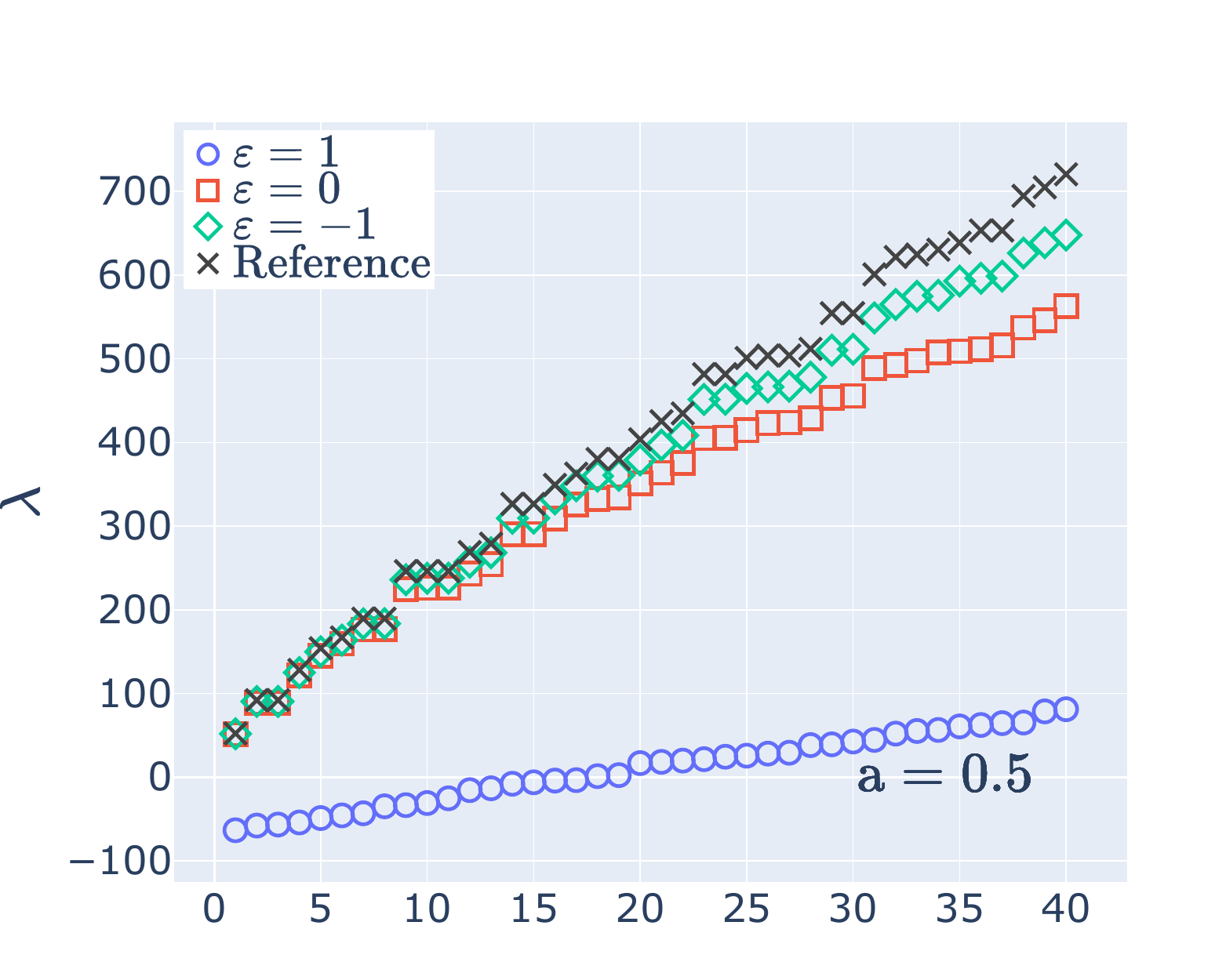}
\end{minipage}
\begin{minipage}{0.325\linewidth}\centering
	{\footnotesize $\mathbb{K}^{-1}\vert_{\Omega_D}=10^{-8}\mathbb{I}$}
	\includegraphics[scale=0.172, trim= 0cm 0cm 2cm 2cm,clip]{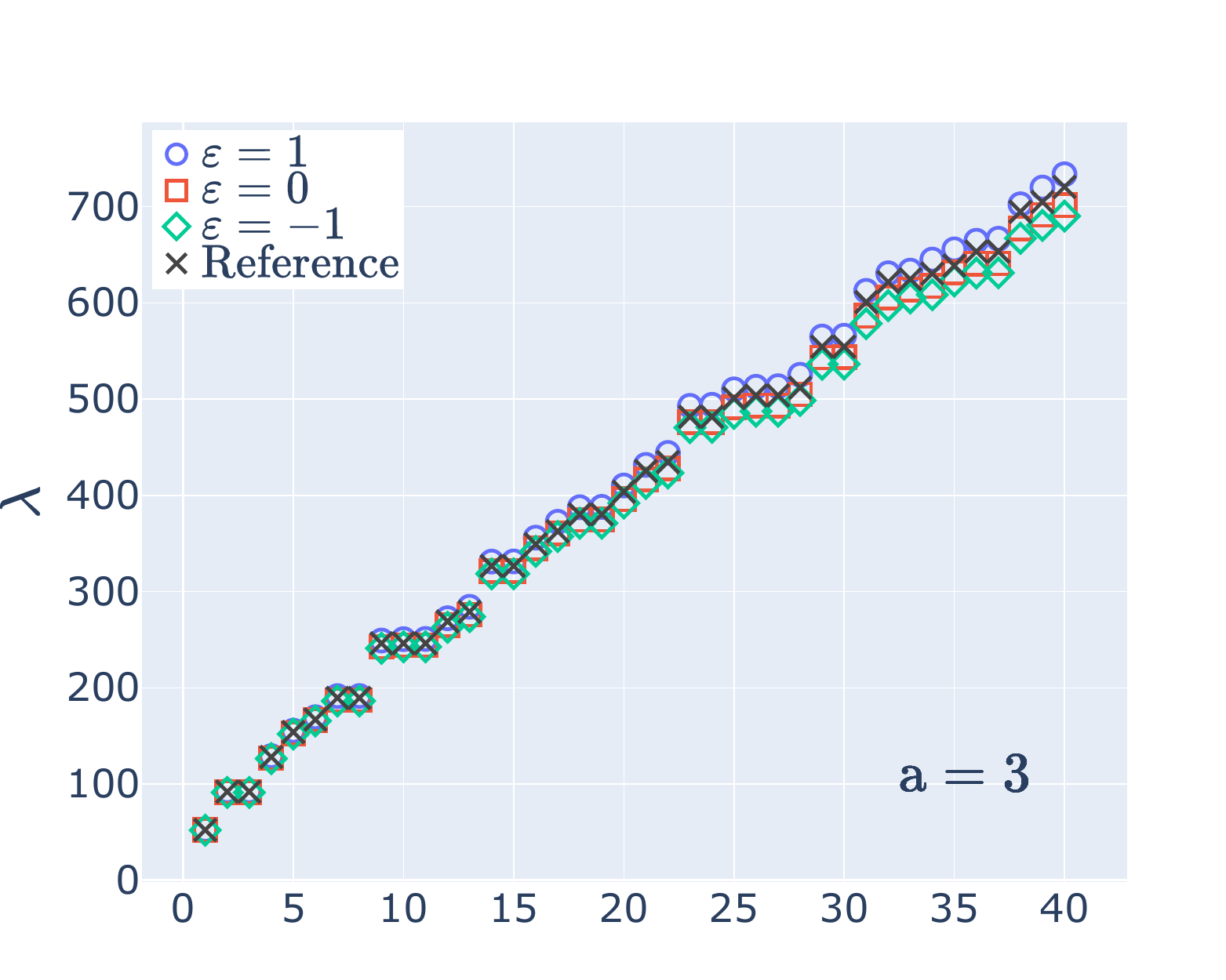}
\end{minipage}
\begin{minipage}{0.325\linewidth}\centering
	{\footnotesize $\mathbb{K}^{-1}\vert_{\Omega_D}=10^{-8}\mathbb{I}$}
	\includegraphics[scale=0.172, trim= 0cm 0cm 2cm 2cm,clip]{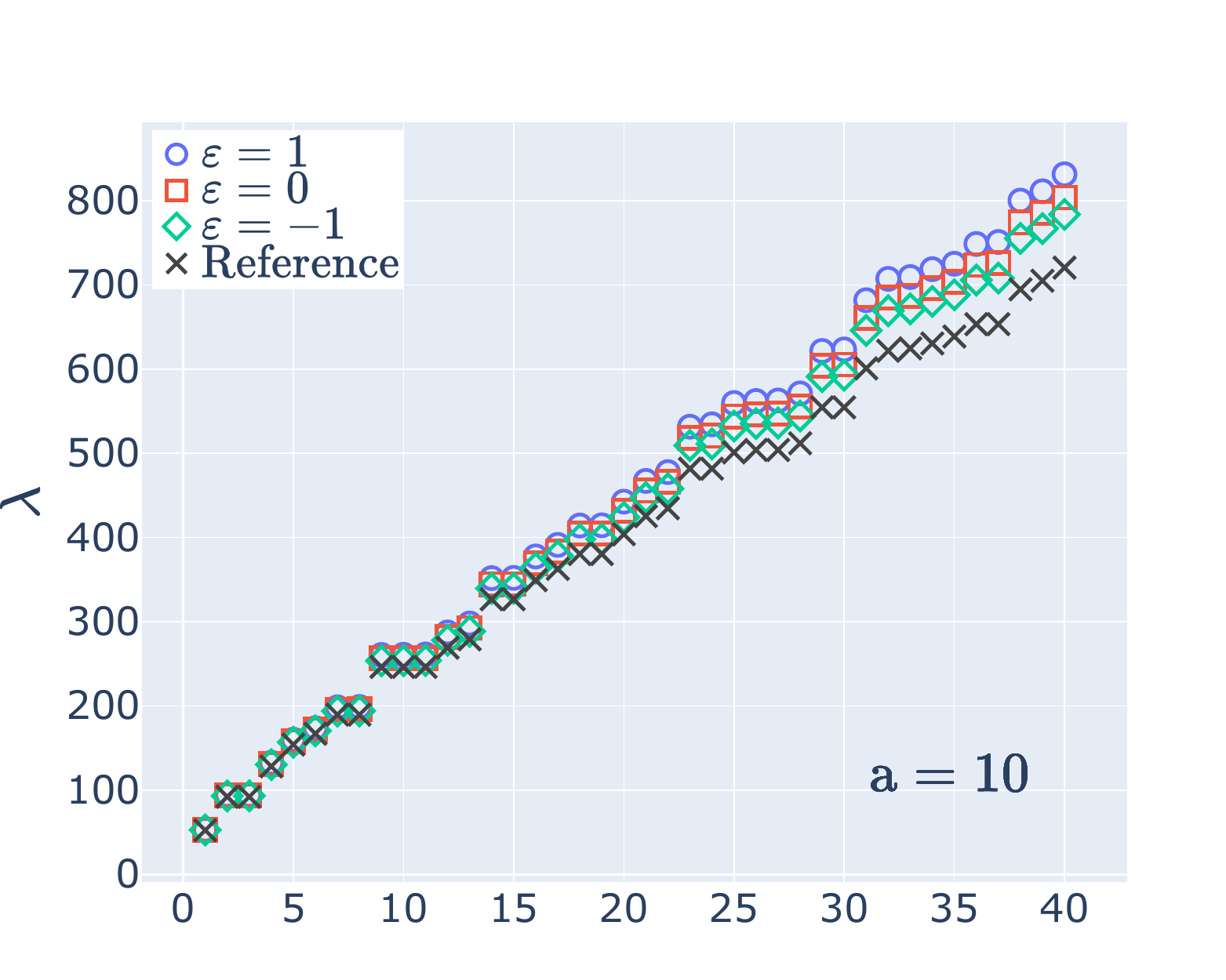}
\end{minipage}\\
\begin{minipage}{0.325\linewidth}\centering
	{\footnotesize $\mathbb{K}^{-1}\vert_{\Omega_D}=10^{3}\mathbb{I}$}
	\includegraphics[scale=0.172, trim= 0cm 0cm 2cm 2cm,clip]{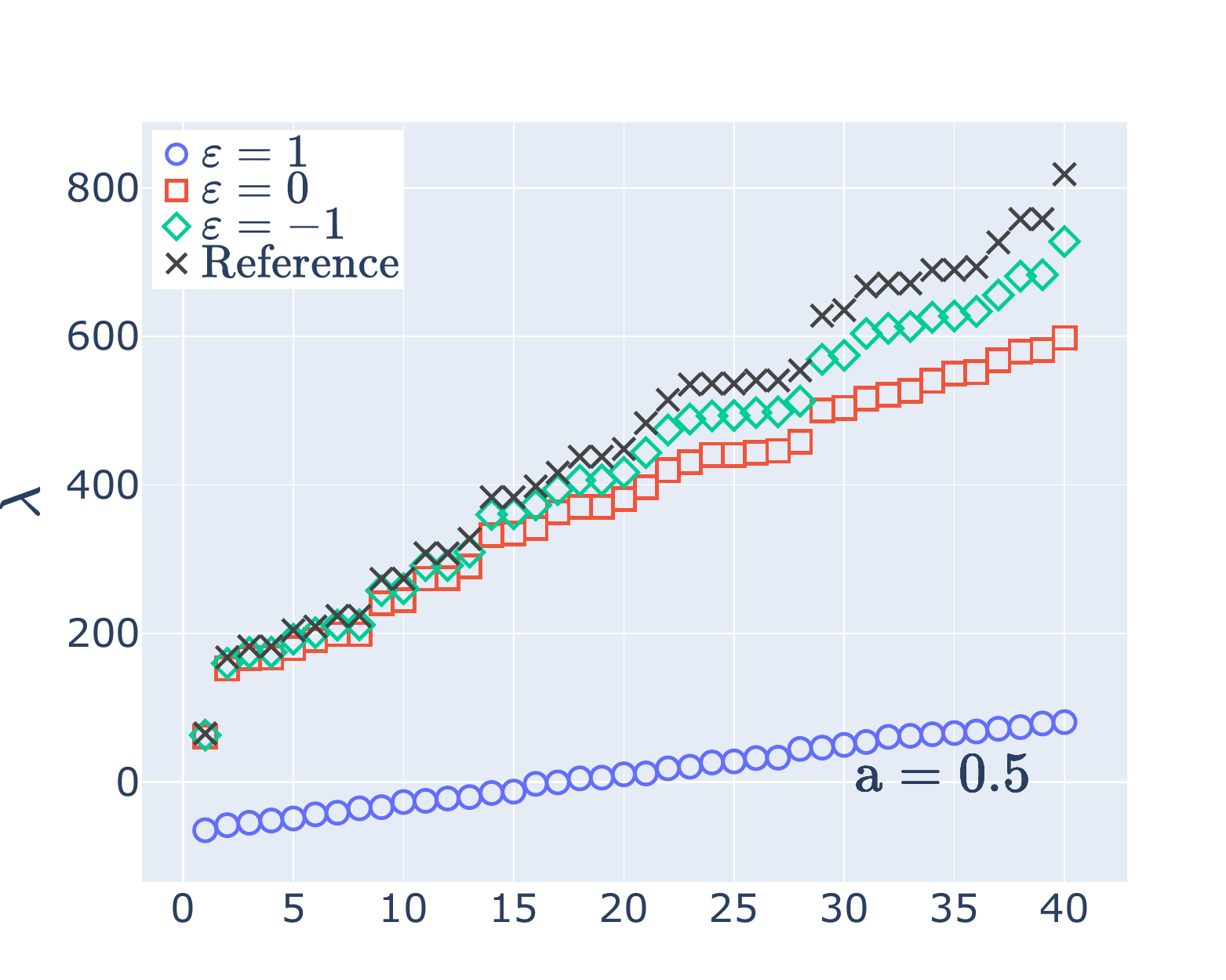}
\end{minipage}
\begin{minipage}{0.325\linewidth}\centering
	{\footnotesize $\mathbb{K}^{-1}\vert_{\Omega_D}=10^{3}\mathbb{I}$}
	\includegraphics[scale=0.172, trim= 0cm 0cm 2cm 2cm,clip]{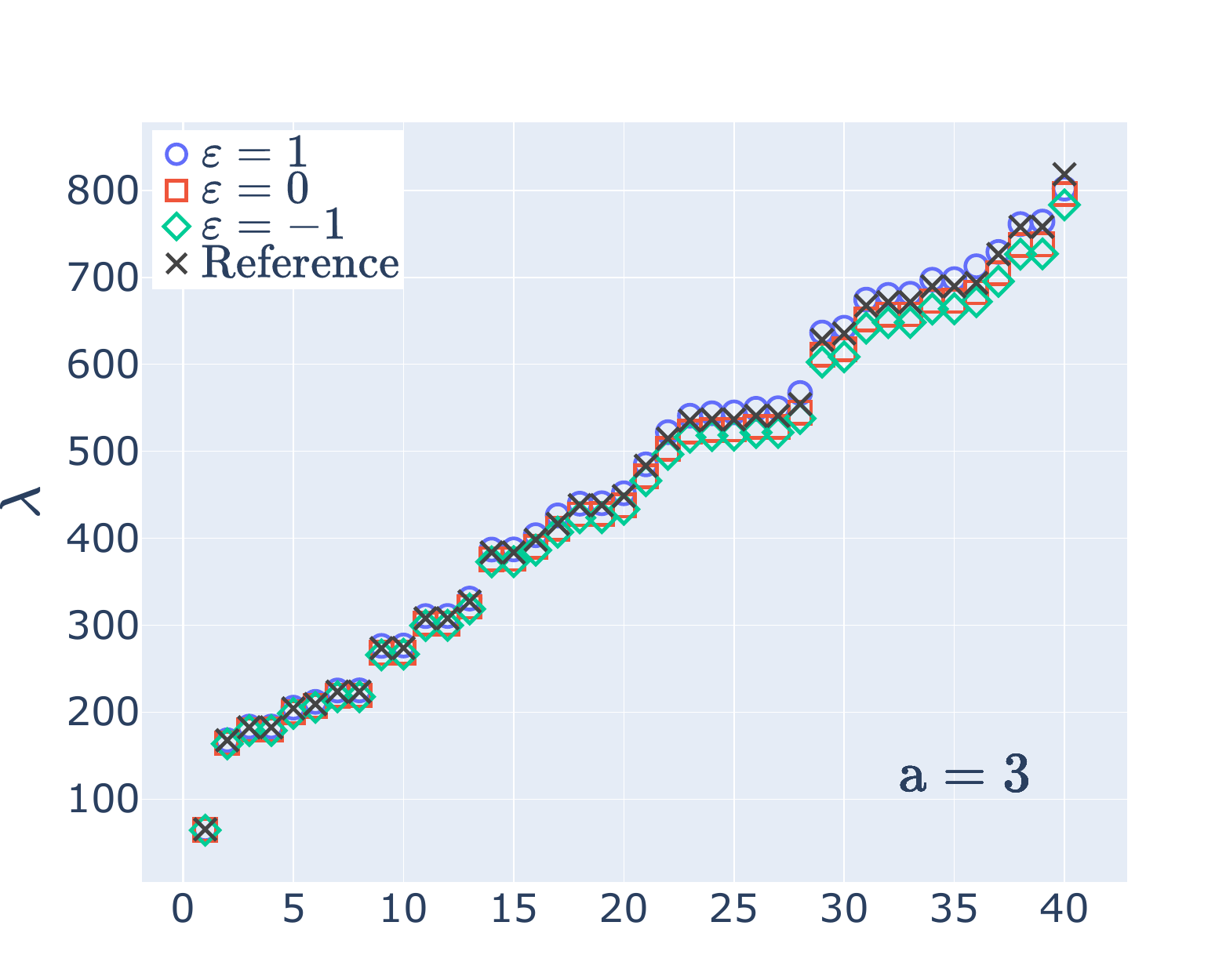}
\end{minipage}
\begin{minipage}{0.325\linewidth}\centering
	{\footnotesize $\mathbb{K}^{-1}\vert_{\Omega_D}=10^{3}\mathbb{I}$}
	\includegraphics[scale=0.172, trim= 0cm 0cm 2cm 2cm,clip]{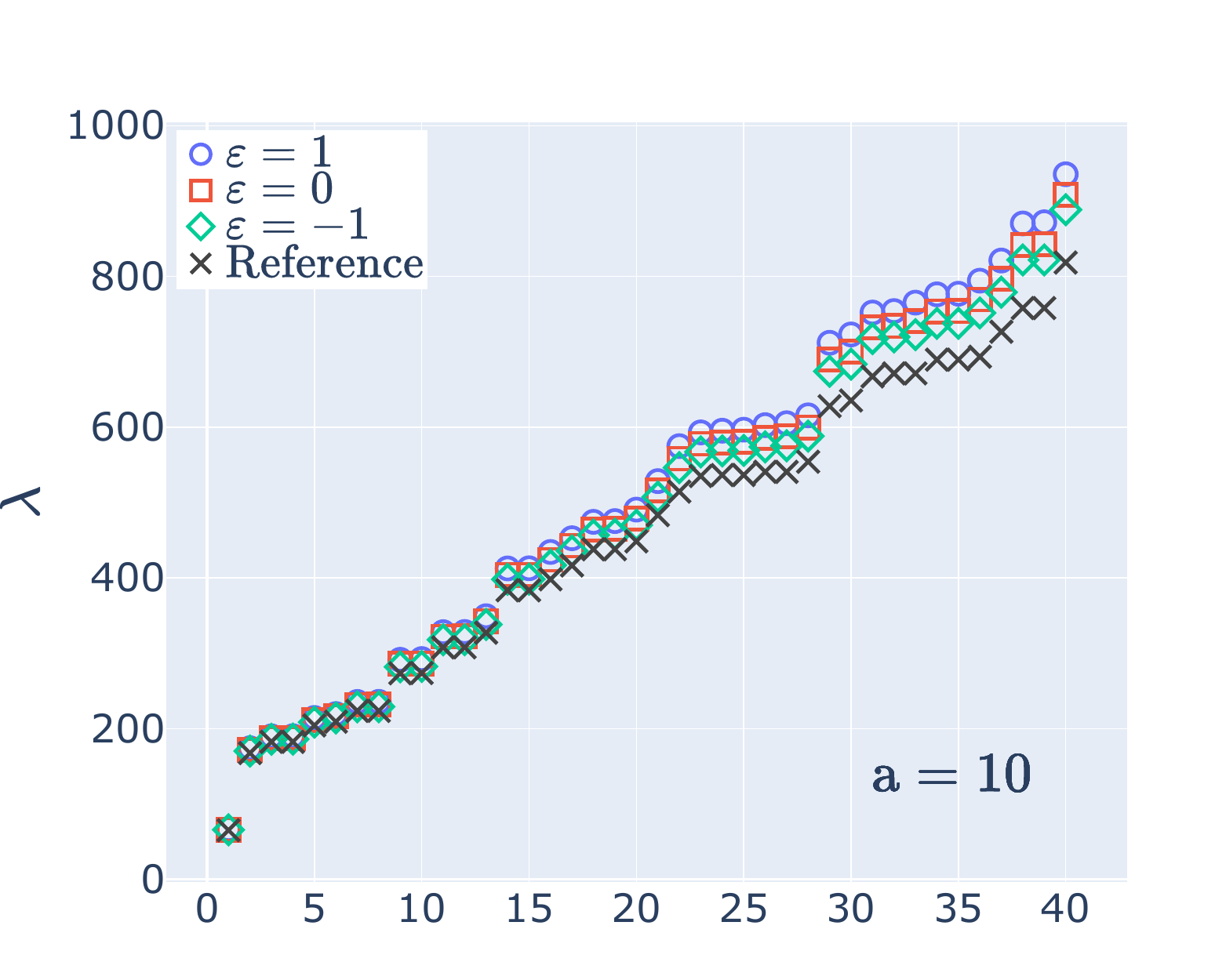}
\end{minipage}\\
\begin{minipage}{0.325\linewidth}\centering
	{\footnotesize $\mathbb{K}^{-1}\vert_{\Omega_D}=10^{5}\mathbb{I}$}
	\includegraphics[scale=0.172, trim= 0cm 0cm 2cm 2cm,clip]{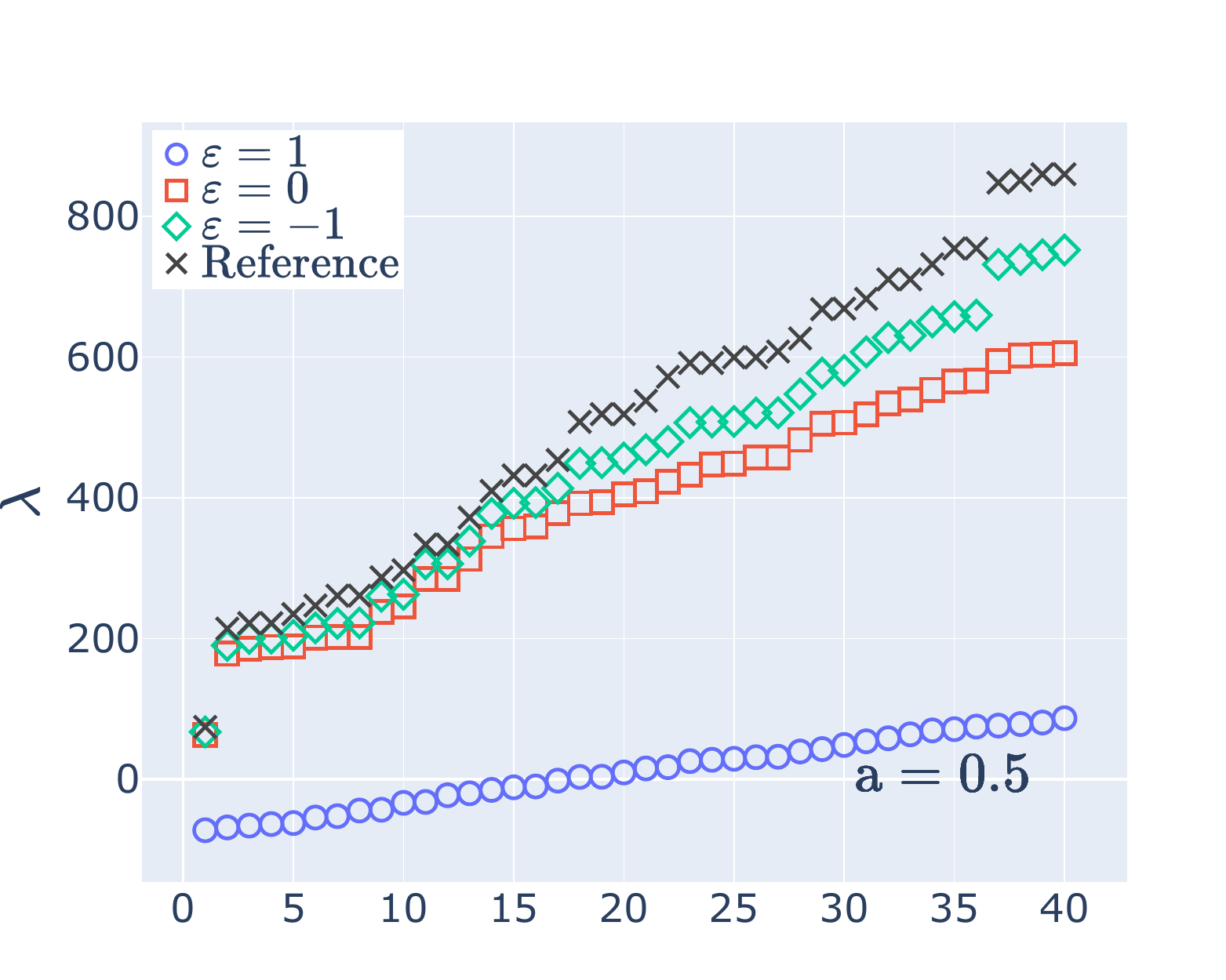}
\end{minipage}
\begin{minipage}{0.325\linewidth}\centering
	{\footnotesize $\mathbb{K}^{-1}\vert_{\Omega_D}=10^{5}\mathbb{I}$}
	\includegraphics[scale=0.172, trim= 0cm 0cm 2cm 2cm,clip]{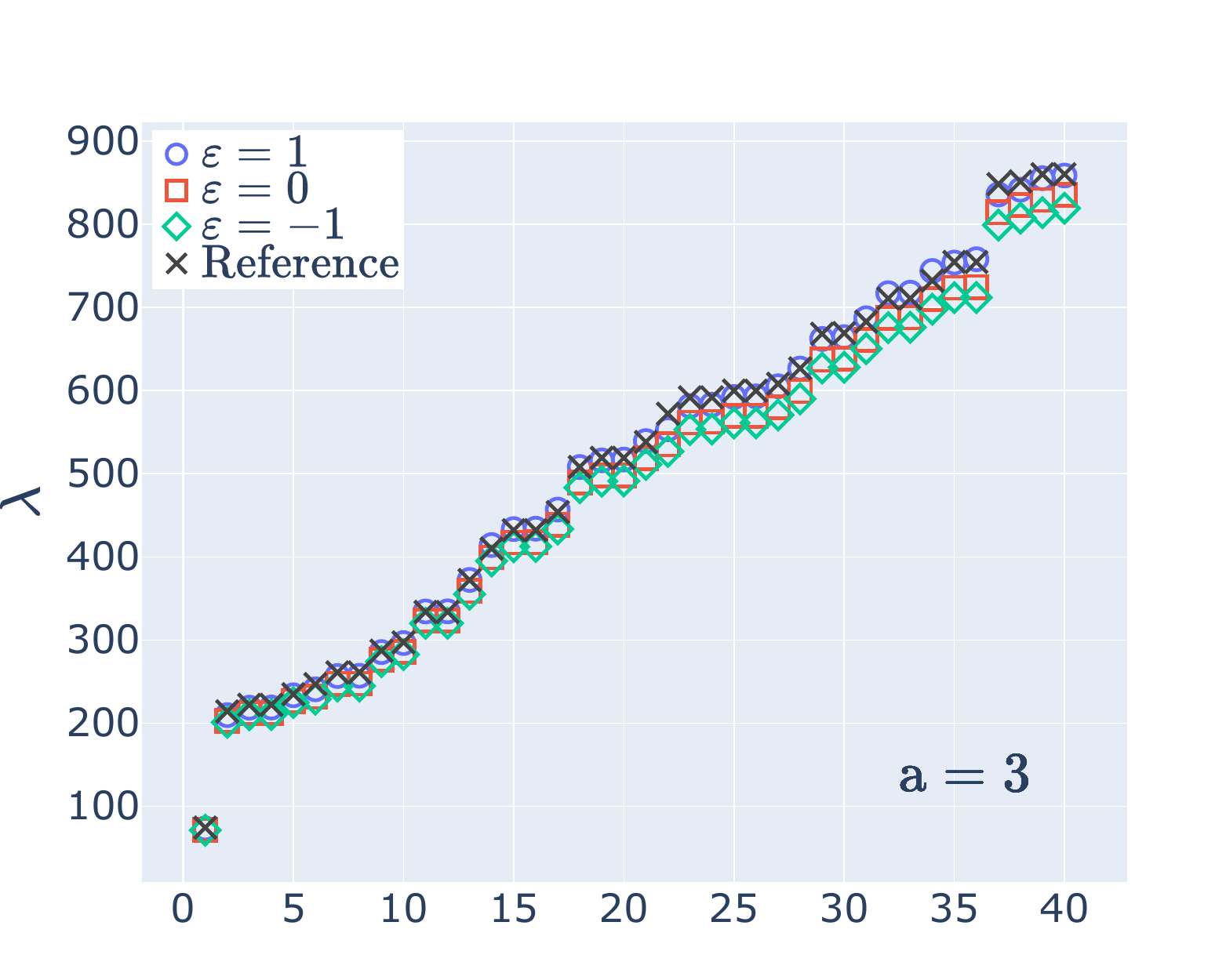}
\end{minipage}
\begin{minipage}{0.325\linewidth}\centering
	{\footnotesize $\mathbb{K}^{-1}\vert_{\Omega_D}=10^{5}\mathbb{I}$}
	\includegraphics[scale=0.172, trim= 0cm 0cm 2cm 2cm,clip]{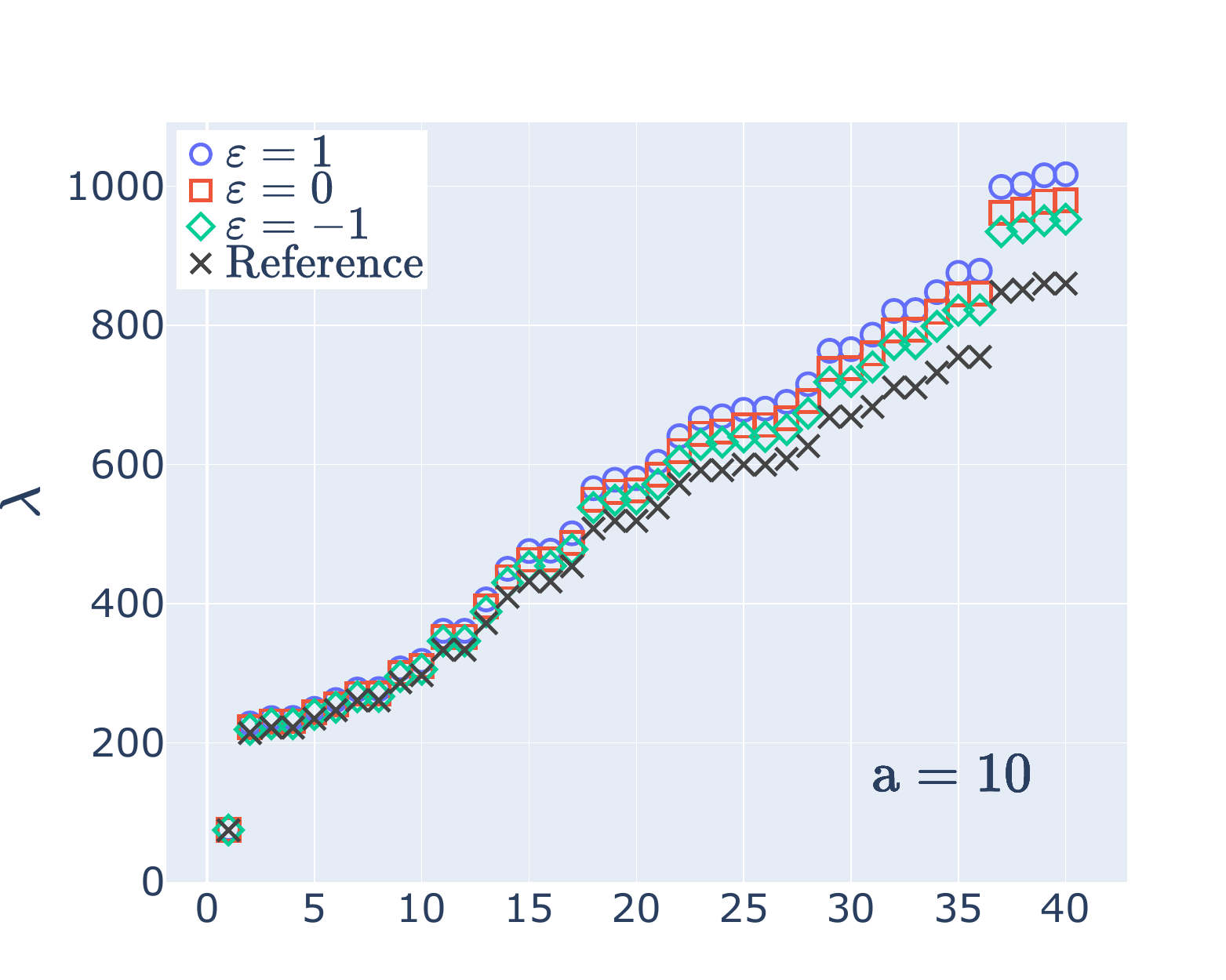}
\end{minipage}
\caption{Test \ref{subsec:square2D-alpha-dependence}. Comparison of the first 40 lowest eigenvalues calculated with the three DG variants against the reference values from \cite{lepe2025jsc}}. 
\label{fig:square_alfas_10em8}
\end{figure}

\begin{figure}[!hbt]\centering
	\begin{minipage}{0.32\linewidth}\centering
		{\footnotesize $\varepsilon=1, k=1$}\\
		\includegraphics[scale=0.217, trim= 0cm 0cm 2cm 2cm,clip]{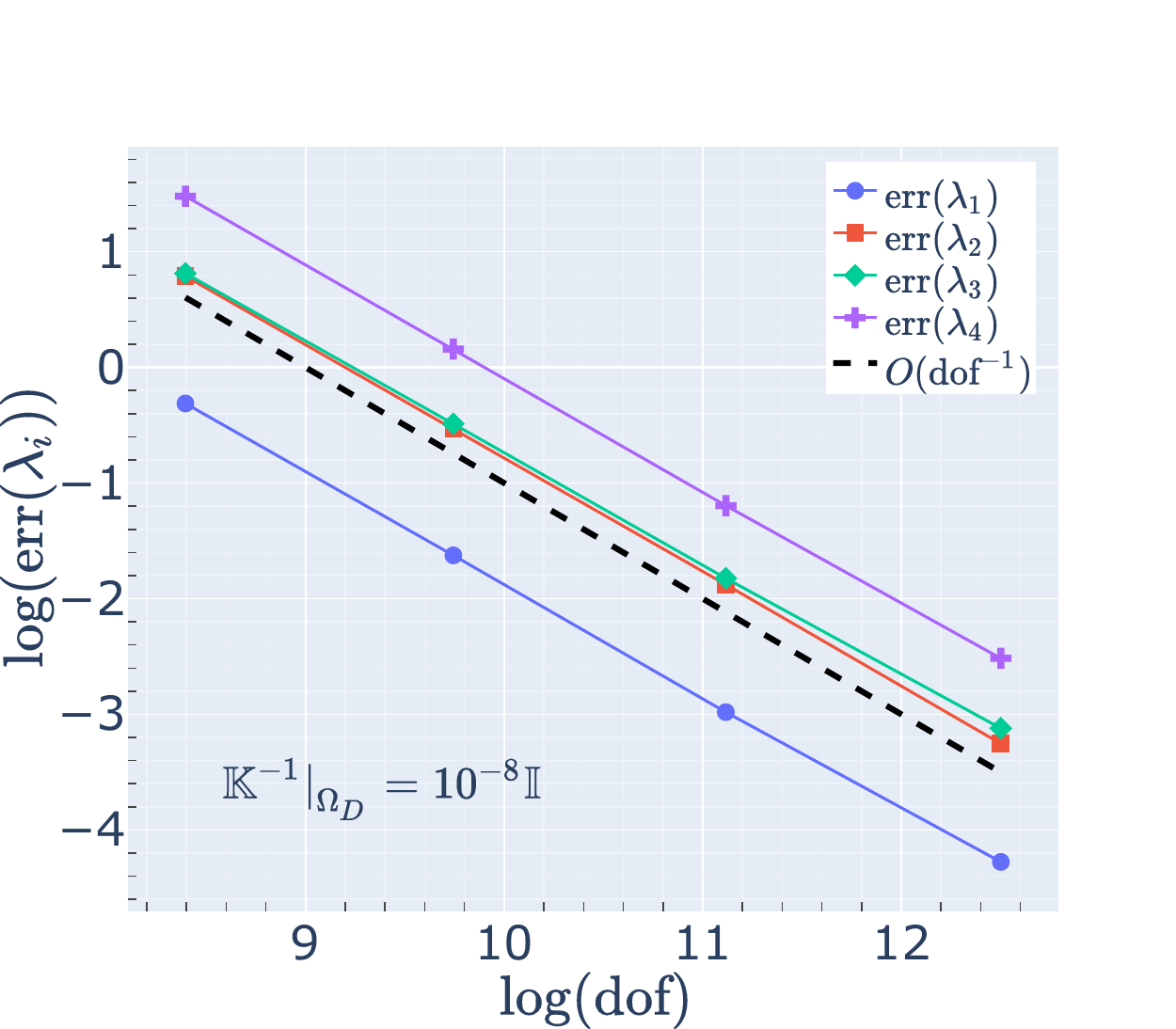}
	\end{minipage}
	\begin{minipage}{0.32\linewidth}\centering
		{\footnotesize $\varepsilon=0, k=1$}\\
		\includegraphics[scale=0.217, trim= 0cm 0cm 2cm 2cm,clip]{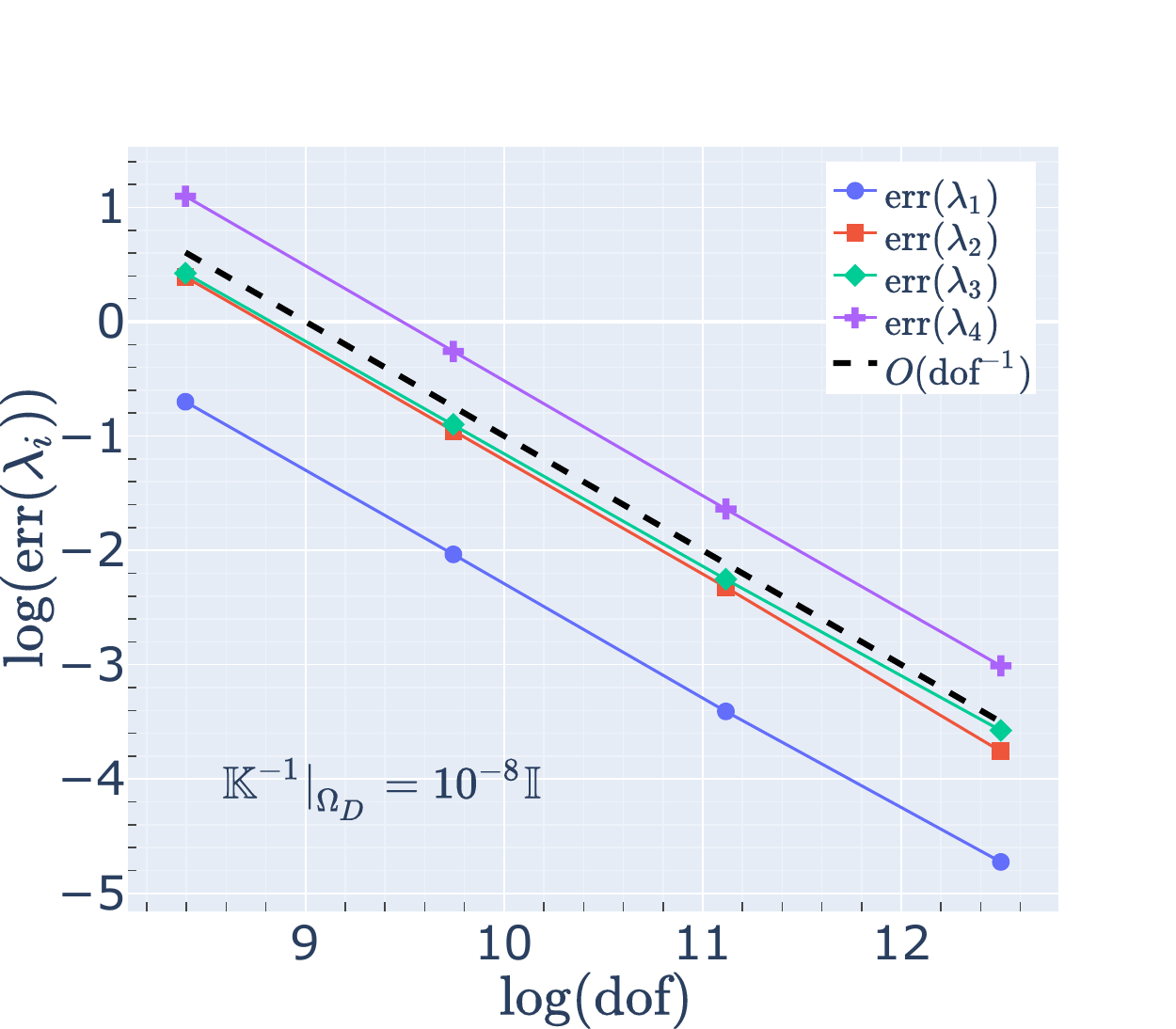}
	\end{minipage}
	\begin{minipage}{0.32\linewidth}\centering
		{\footnotesize $\varepsilon=-1, k=1$}\\
		\includegraphics[scale=0.217, trim= 0cm 0cm 2cm 2cm,clip]{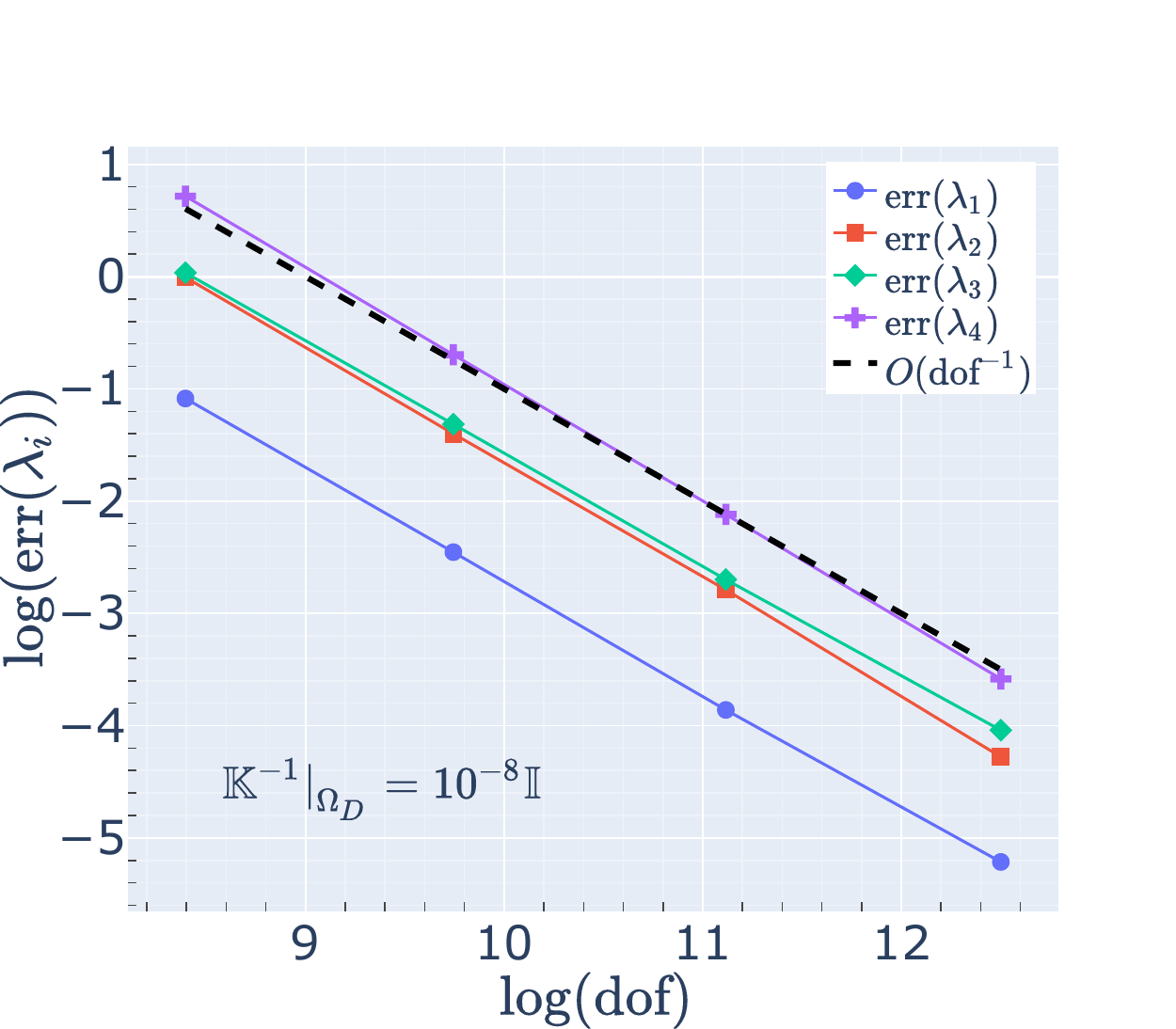}
	\end{minipage}\\
\begin{minipage}{0.32\linewidth}\centering
{\footnotesize $\varepsilon=1, k=2$}\\
\includegraphics[scale=0.217, trim= 0cm 0cm 2cm 2cm,clip]{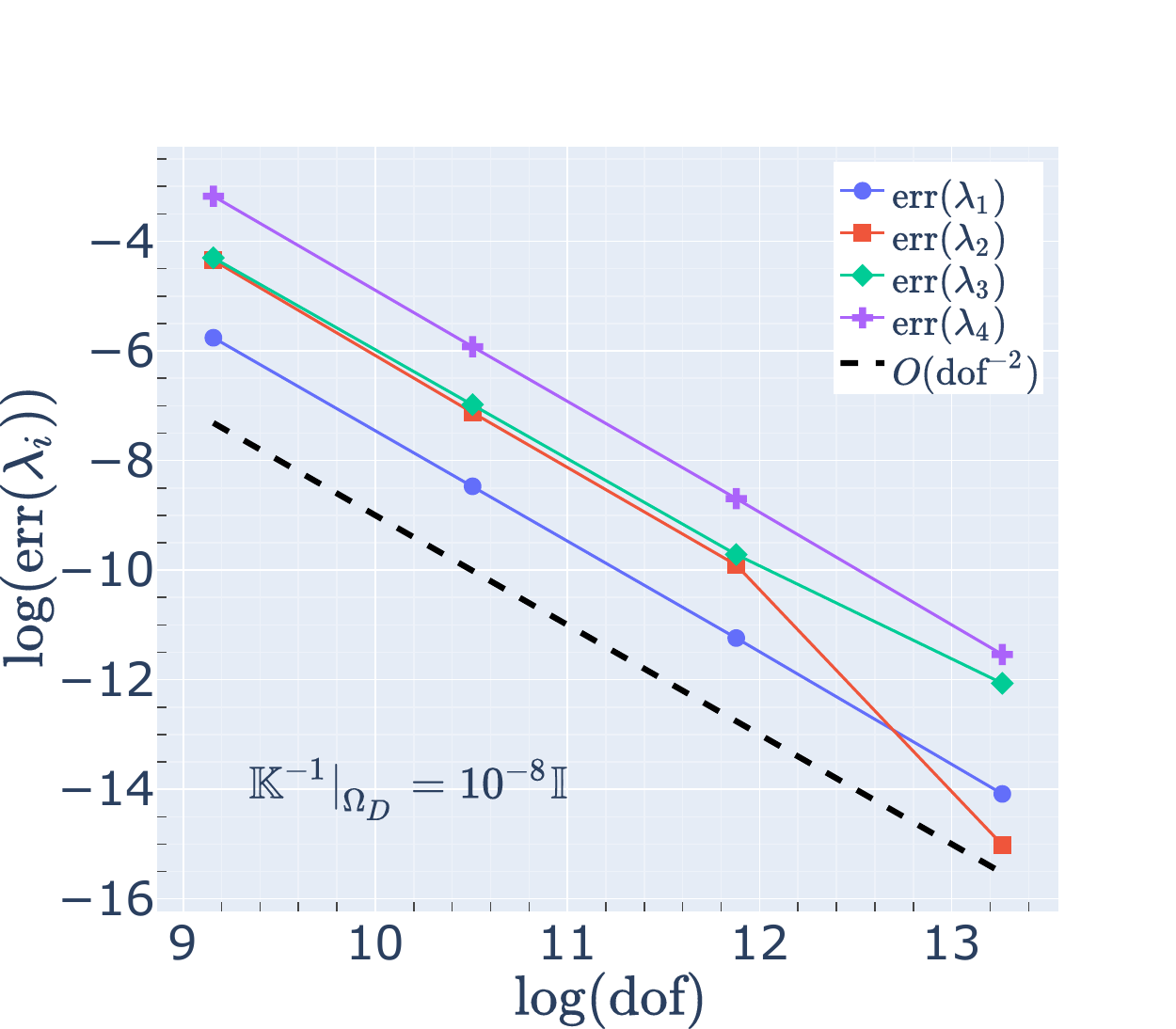}
\end{minipage}
\begin{minipage}{0.32\linewidth}\centering
	{\footnotesize $\varepsilon=0, k=2$}\\
	\includegraphics[scale=0.217, trim= 0cm 0cm 2cm 2cm,clip]{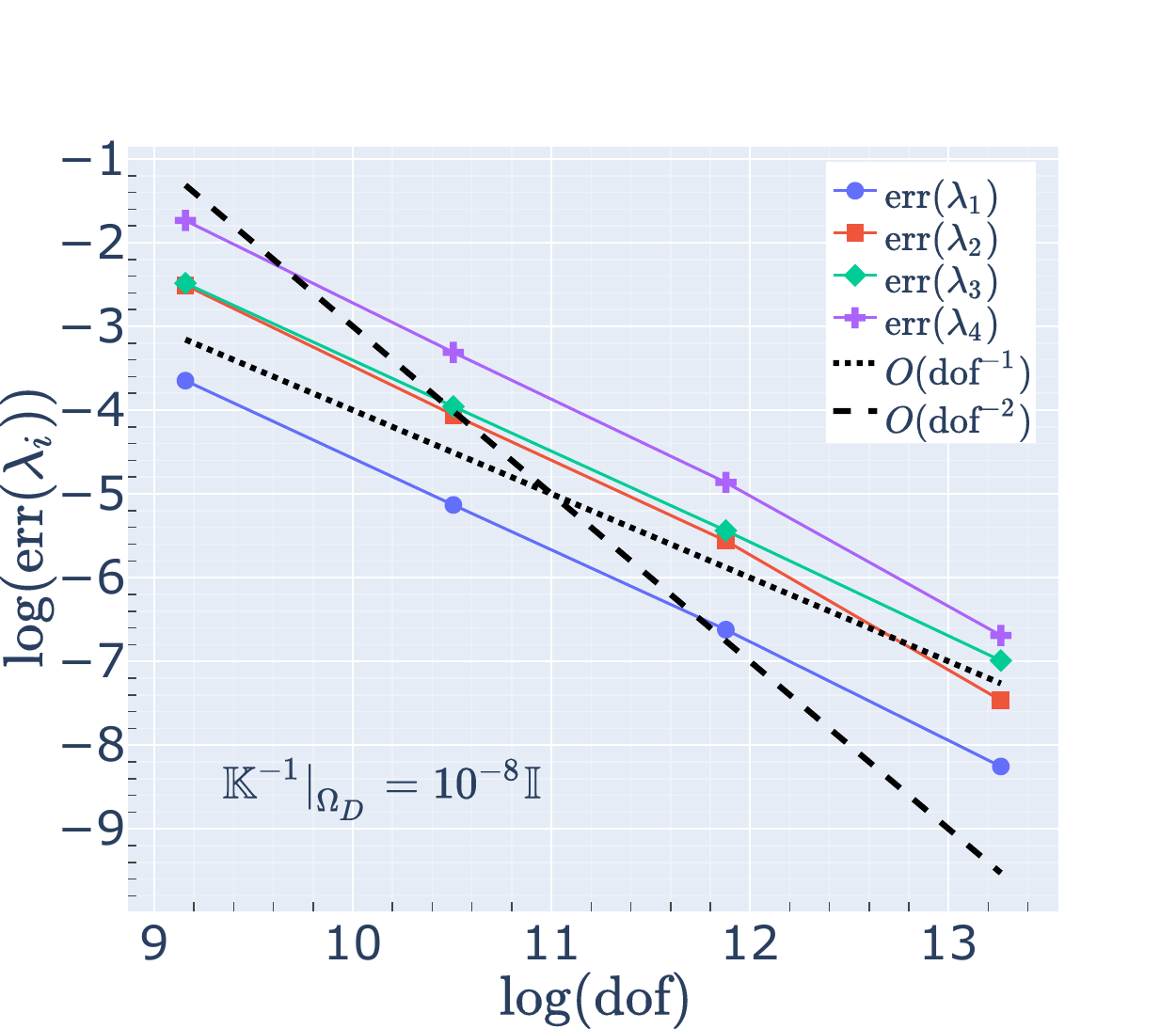}
\end{minipage}
\begin{minipage}{0.32\linewidth}\centering
	{\footnotesize $\varepsilon=-1, k=2$}\\
	\includegraphics[scale=0.217, trim= 0cm 0cm 2cm 2cm,clip]{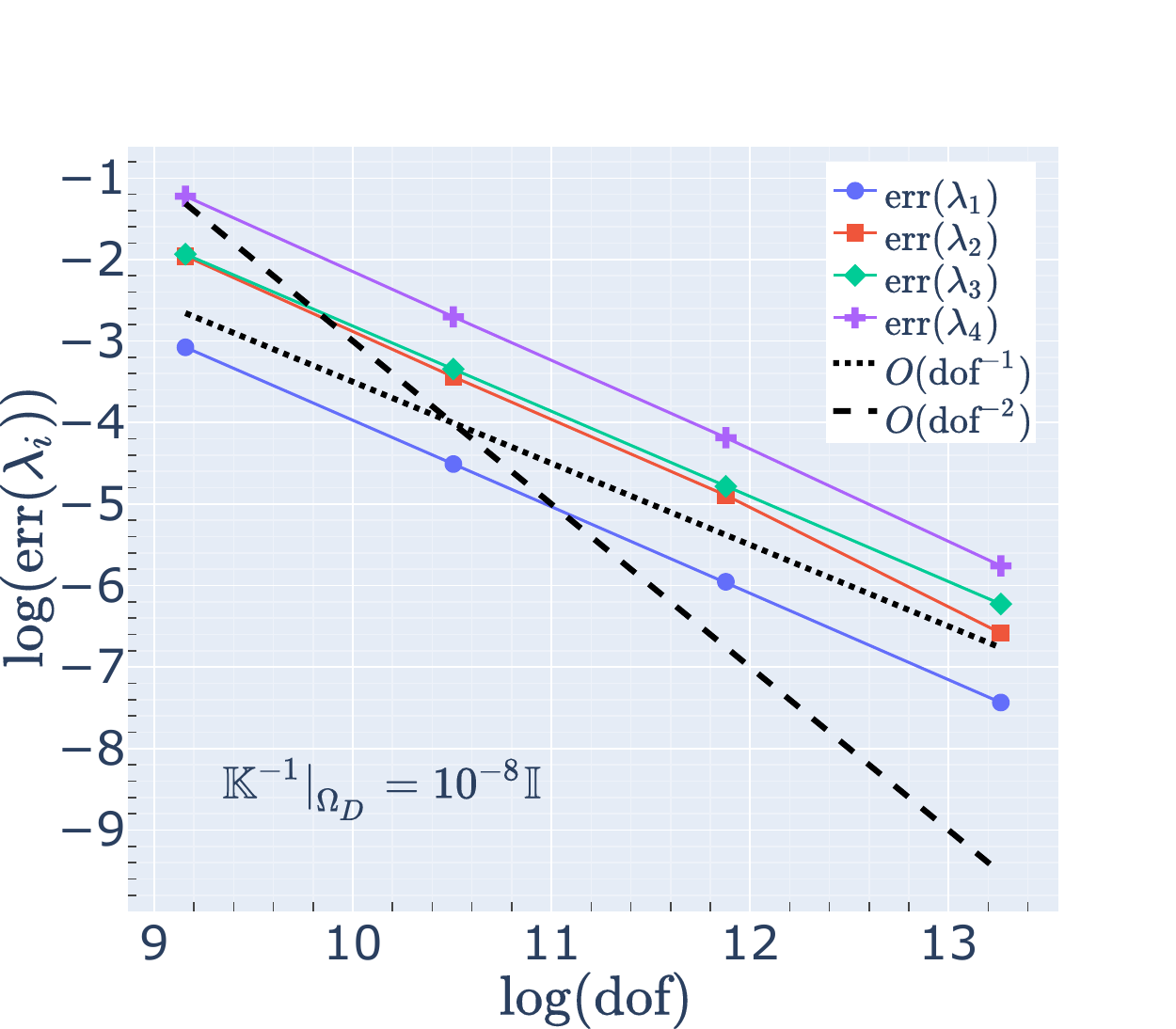}
\end{minipage}\\
\begin{minipage}{0.32\linewidth}\centering
{\footnotesize $\varepsilon=1, k=3$}\\
\includegraphics[scale=0.217, trim= 0cm 0cm 2cm 2cm,clip]{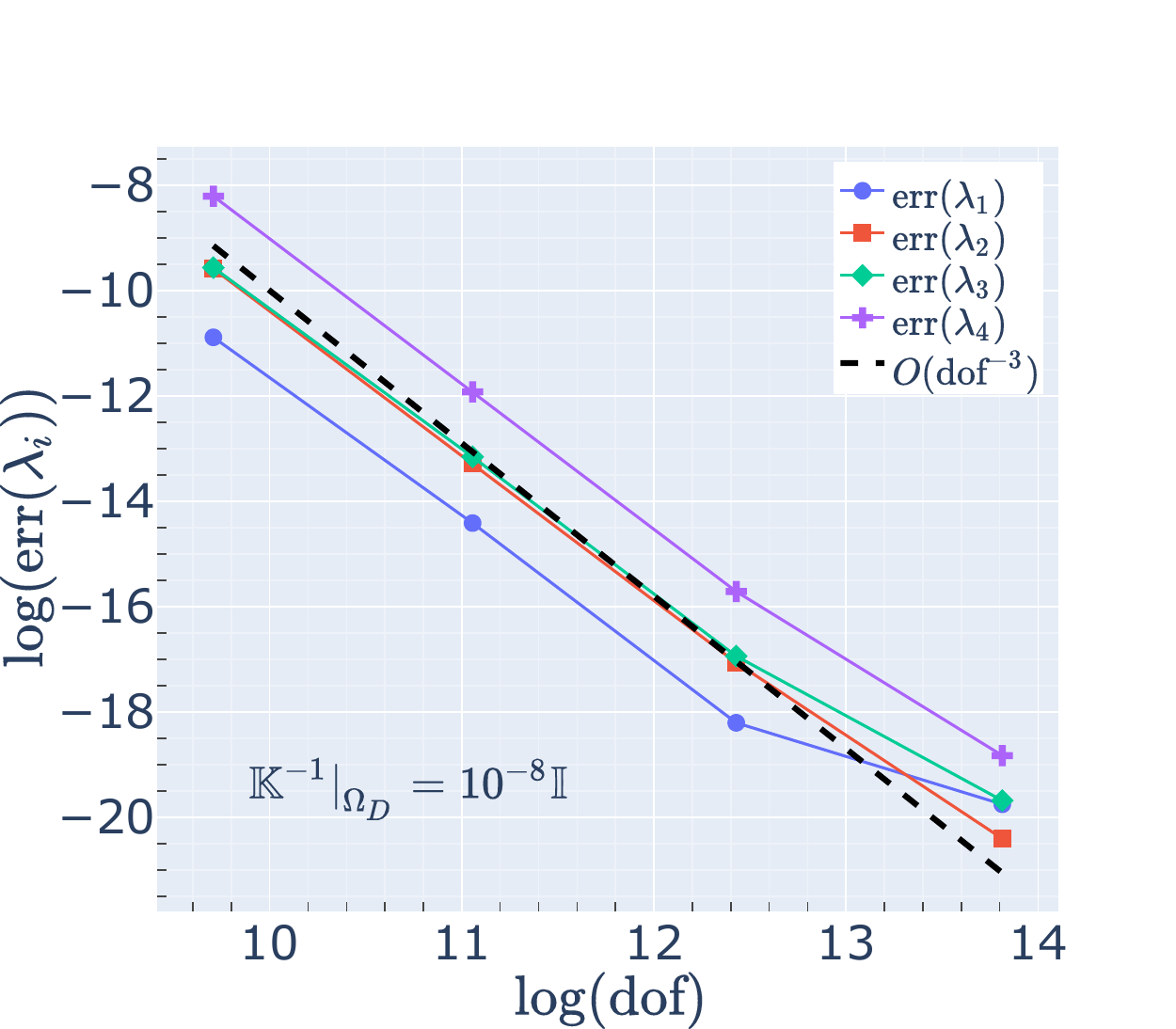}
\end{minipage}
\begin{minipage}{0.32\linewidth}\centering
{\footnotesize $\varepsilon=0, k=3$}\\
\includegraphics[scale=0.217, trim= 0cm 0cm 2cm 2cm,clip]{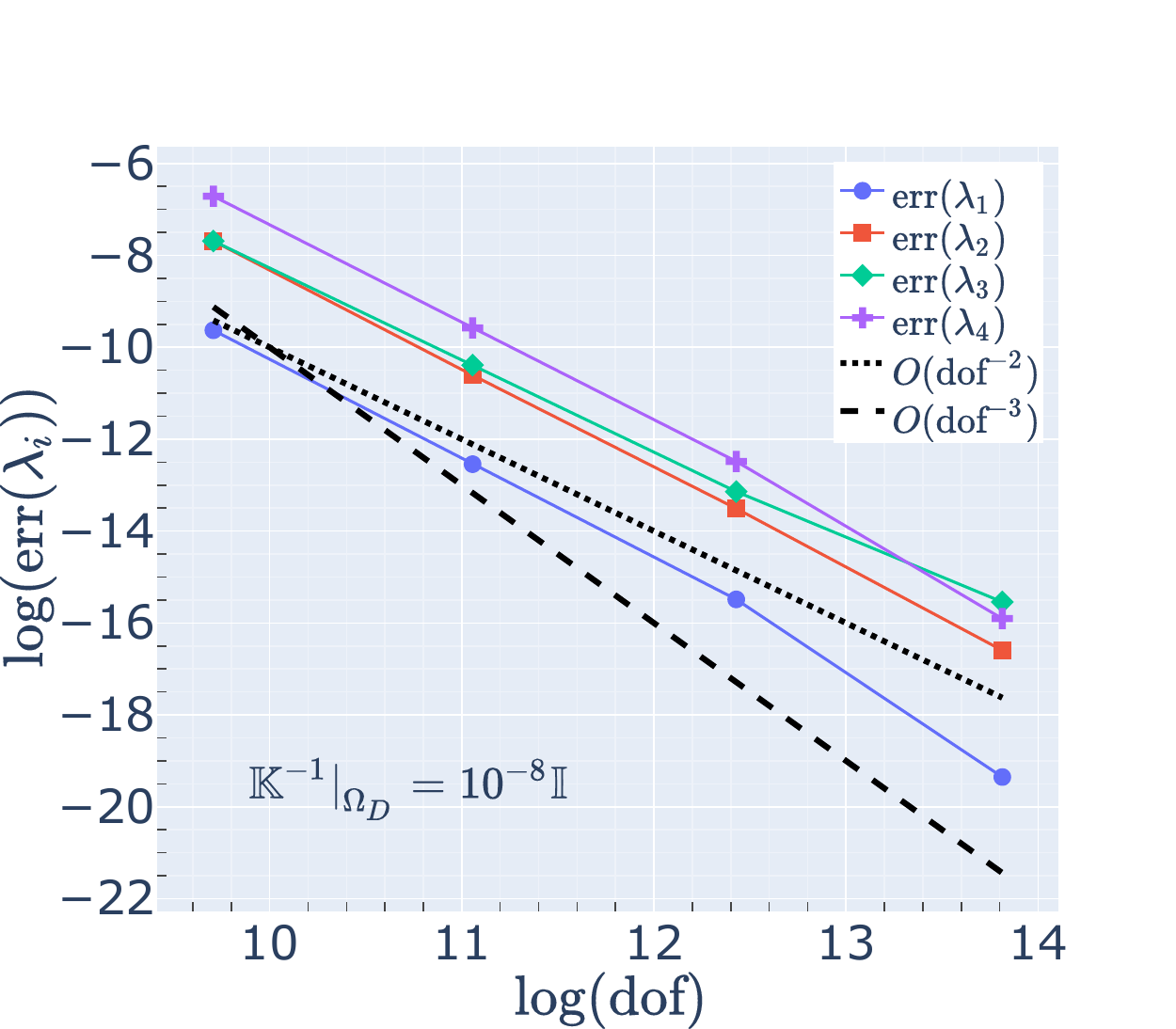}
\end{minipage}
\begin{minipage}{0.32\linewidth}\centering
{\footnotesize $\varepsilon=-1, k=3$}\\
\includegraphics[scale=0.217, trim= 0cm 0cm 2cm 2cm,clip]{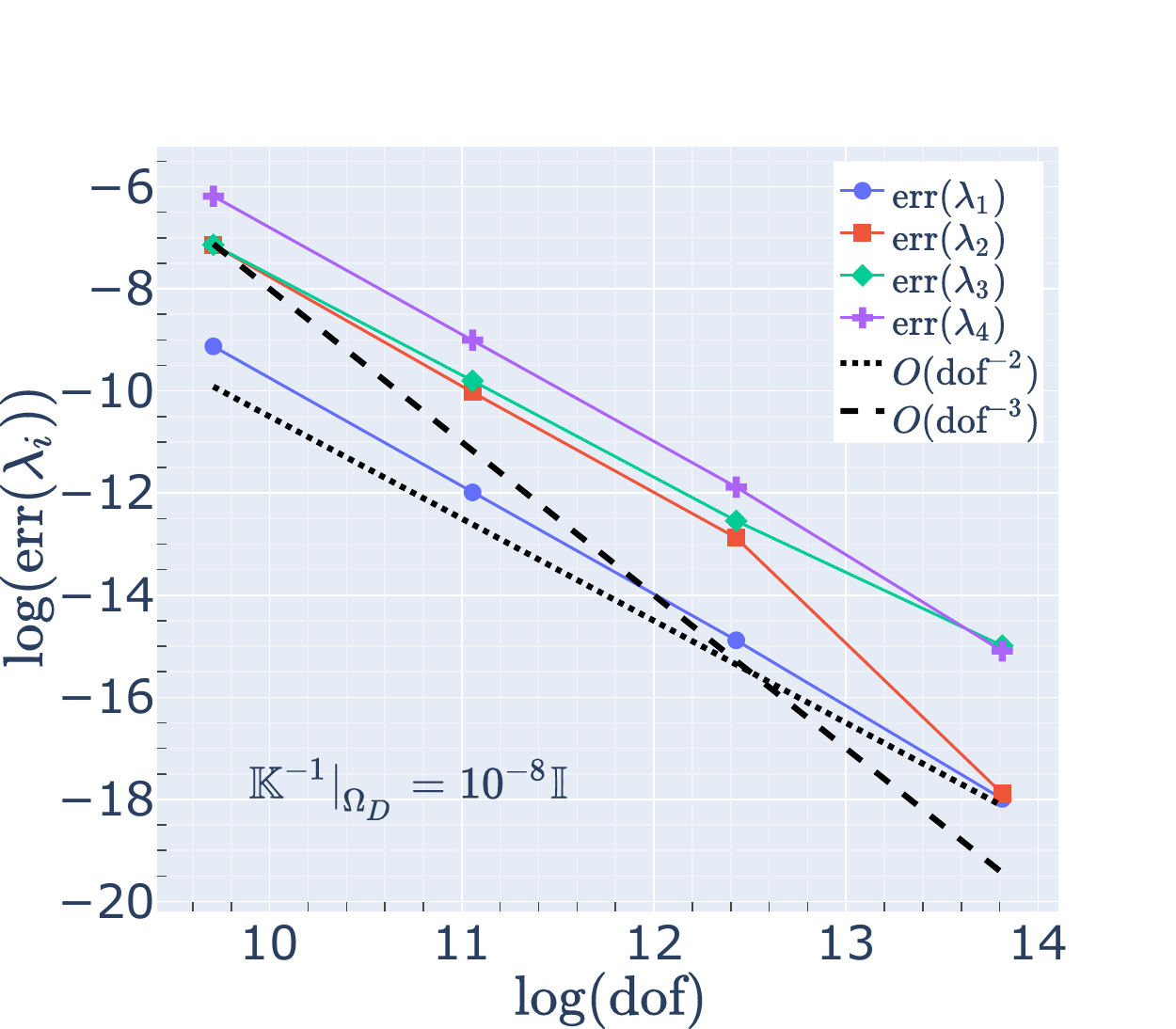}
\end{minipage}
\caption{Test \ref{subsec:square2D_convergence}. Convergence history for the first four eigenvalues on each IPDG scheme for a selected stability parameter $\texttt{a}=10k^2$ and $\mathbb{K}^{-1}\vert_{\Omega_D}=10^{-8}$.}
\label{fig:square_error_1em8}
\end{figure}

\begin{figure}[!hbt]\centering
	\begin{minipage}{0.32\linewidth}\centering
		{\footnotesize $\varepsilon=1, k=1$}\\
		\includegraphics[scale=0.217, trim= 0cm 0cm 2cm 2cm,clip]{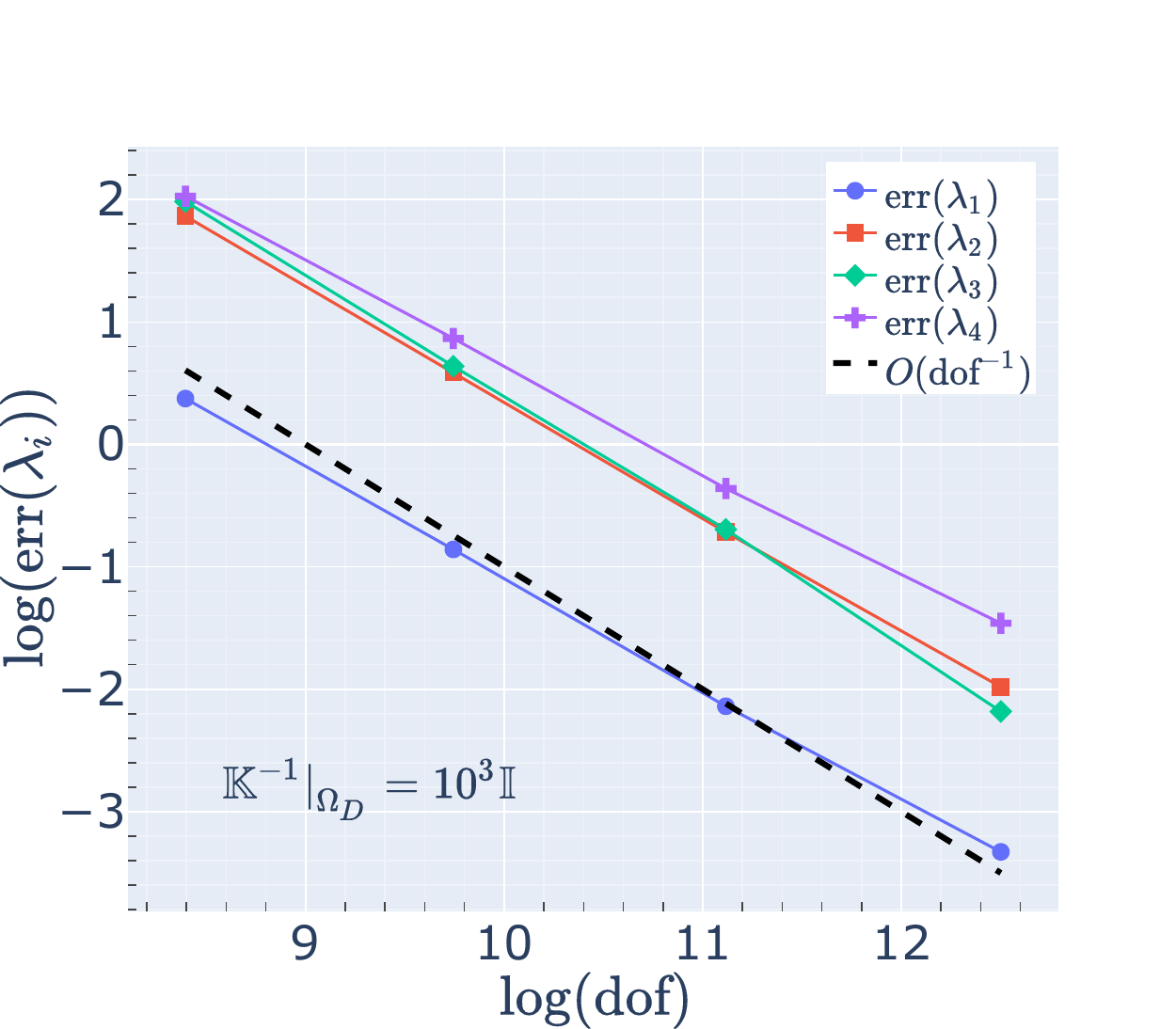}
	\end{minipage}
	\begin{minipage}{0.32\linewidth}\centering
		{\footnotesize $\varepsilon=0, k=1$}\\
		\includegraphics[scale=0.217, trim= 0cm 0cm 2cm 2cm,clip]{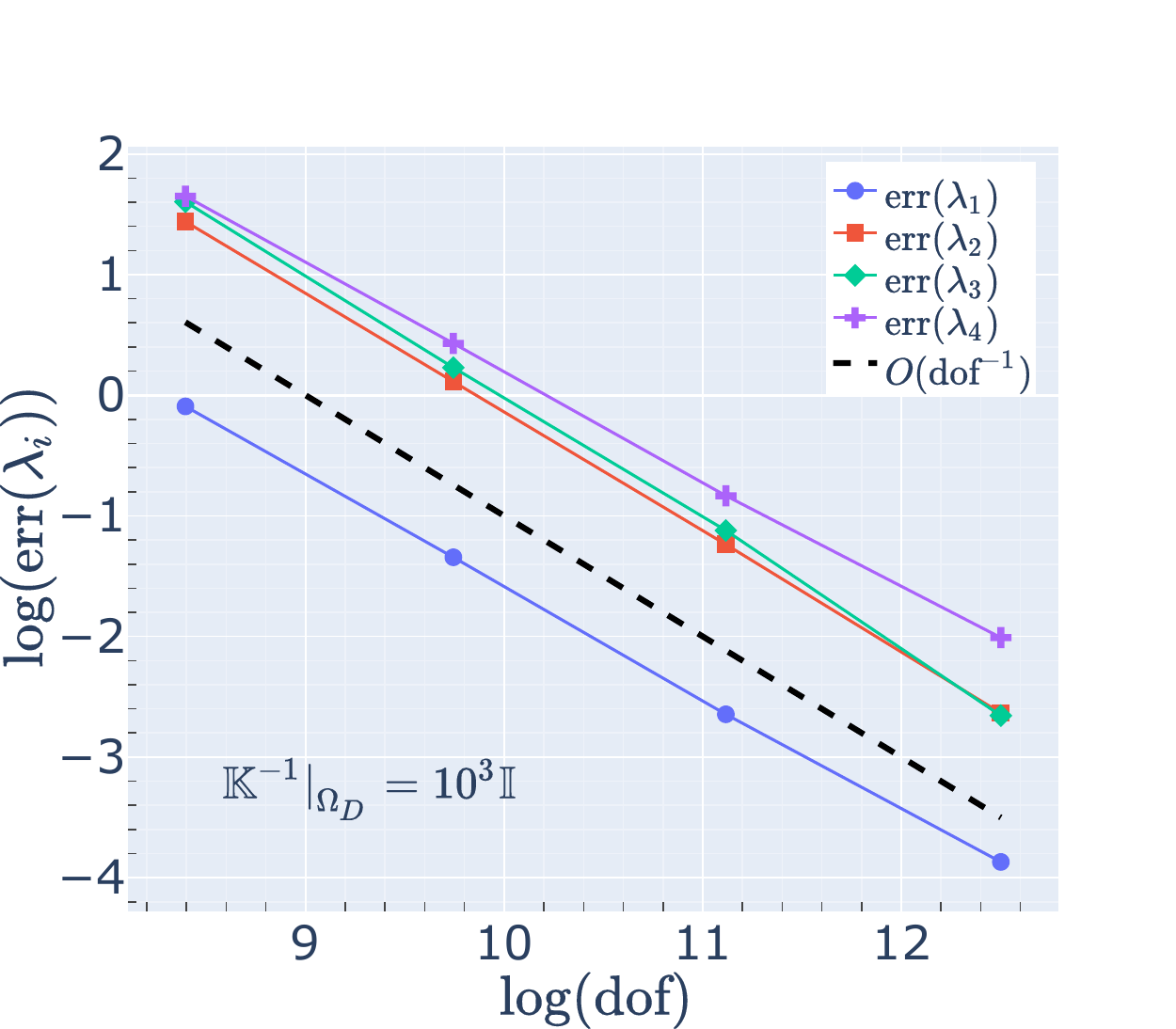}
	\end{minipage}
	\begin{minipage}{0.32\linewidth}\centering
		{\footnotesize $\varepsilon=-1, k=1$}\\
		\includegraphics[scale=0.217, trim= 0cm 0cm 2cm 2cm,clip]{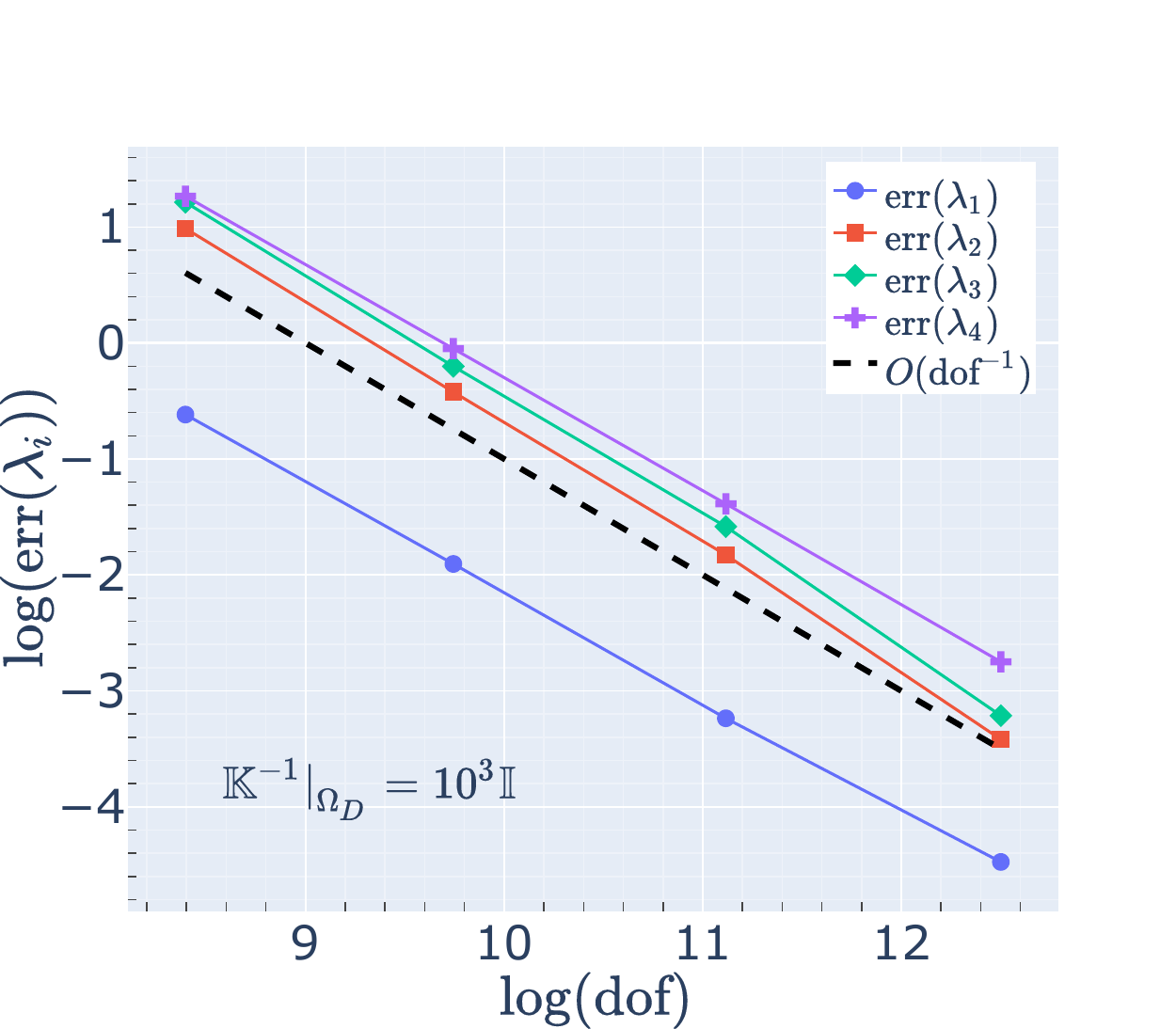}
	\end{minipage}\\
	\begin{minipage}{0.32\linewidth}\centering
		{\footnotesize $\varepsilon=1, k=2$}\\
		\includegraphics[scale=0.217, trim= 0cm 0cm 2cm 2cm,clip]{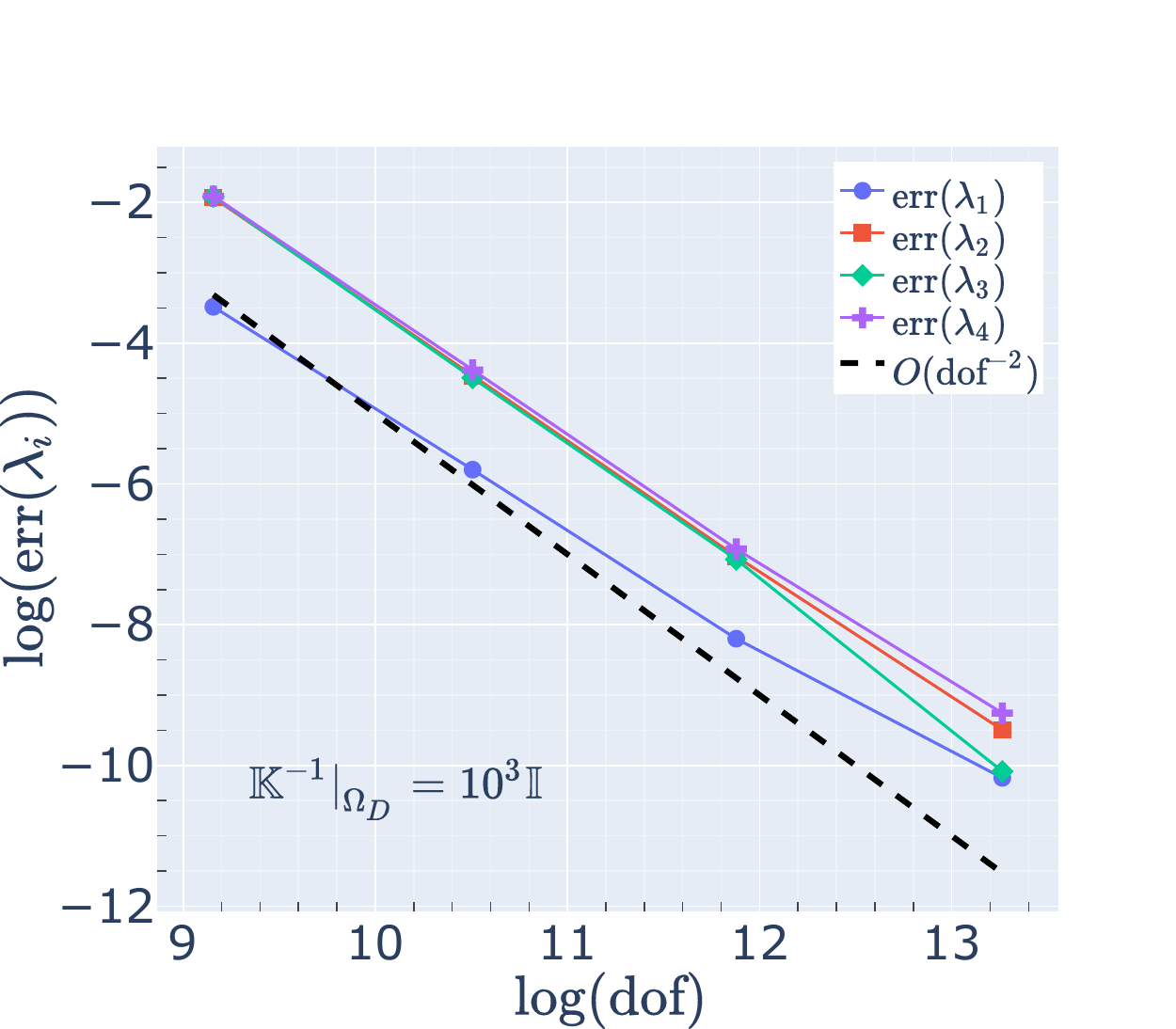}
	\end{minipage}
	\begin{minipage}{0.32\linewidth}\centering
		{\footnotesize $\varepsilon=0, k=2$}\\
		\includegraphics[scale=0.217, trim= 0cm 0cm 2cm 2cm,clip]{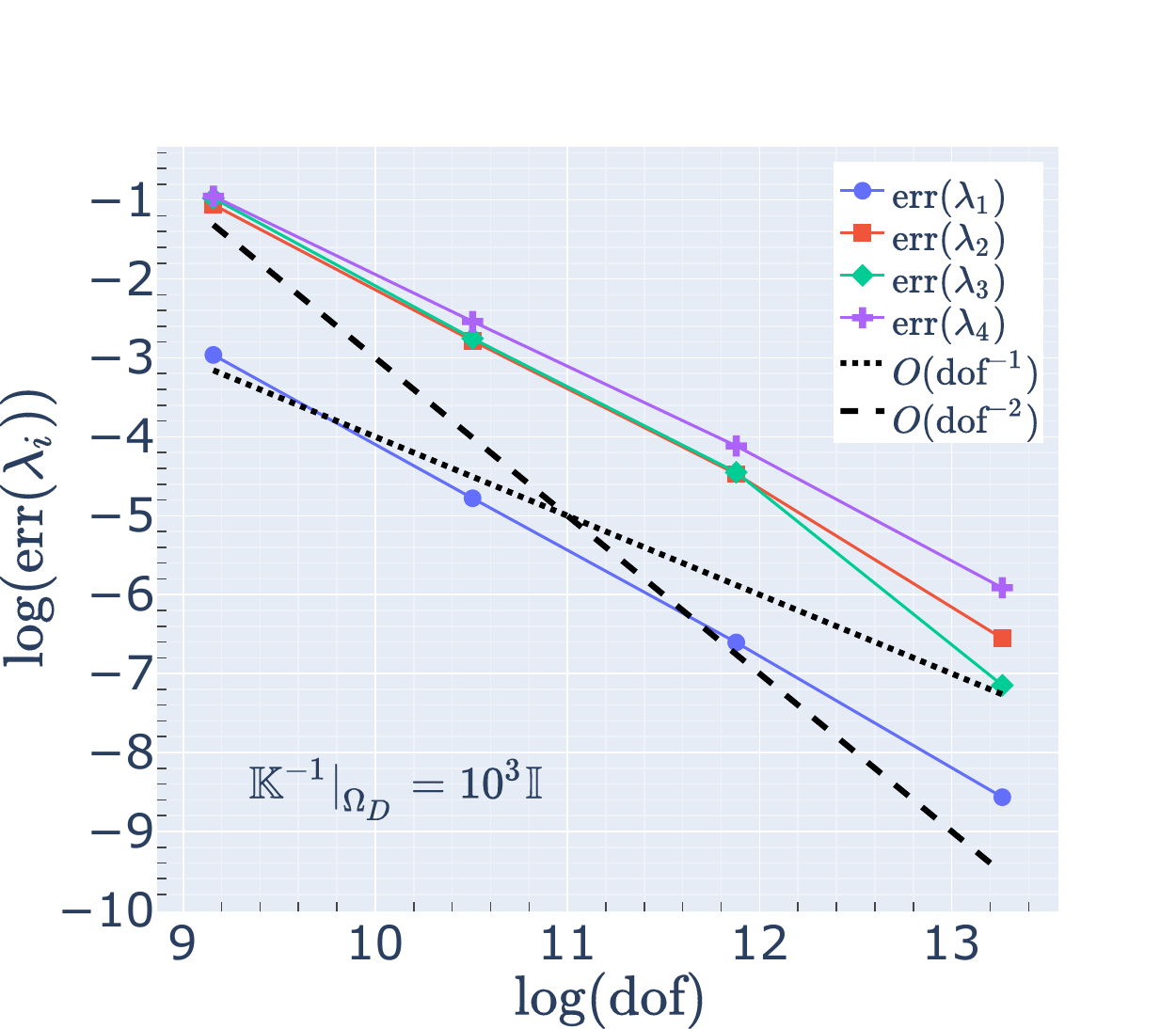}
	\end{minipage}
	\begin{minipage}{0.32\linewidth}\centering
		{\footnotesize $\varepsilon=-1, k=2$}\\
		\includegraphics[scale=0.217, trim= 0cm 0cm 2cm 2cm,clip]{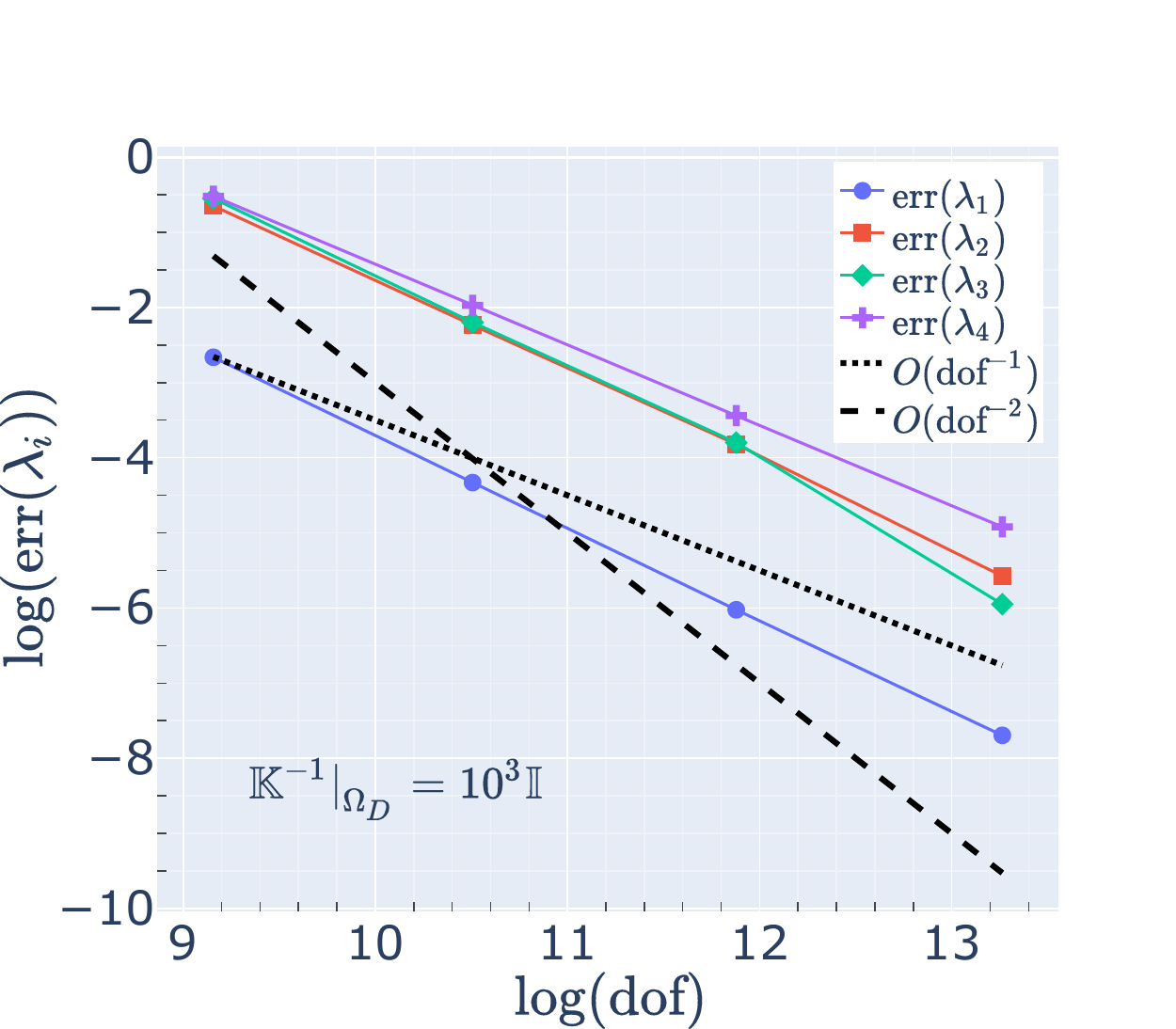}
	\end{minipage}\\
	\begin{minipage}{0.32\linewidth}\centering
		{\footnotesize $\varepsilon=1, k=3$}\\
		\includegraphics[scale=0.217, trim= 0cm 0cm 2cm 2cm,clip]{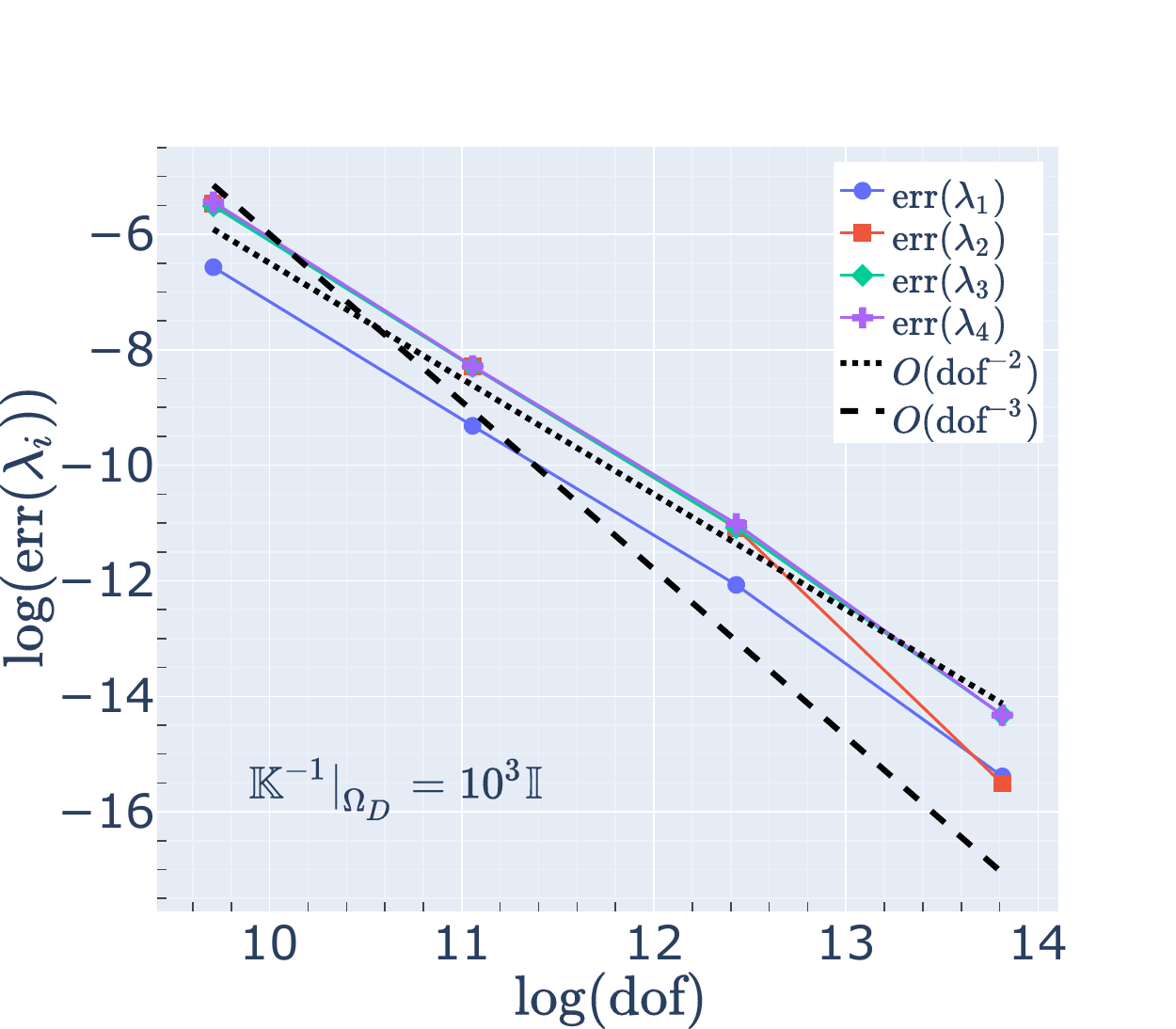}
	\end{minipage}
	\begin{minipage}{0.32\linewidth}\centering
		{\footnotesize $\varepsilon=0, k=3$}\\
		\includegraphics[scale=0.217, trim= 0cm 0cm 2cm 2cm,clip]{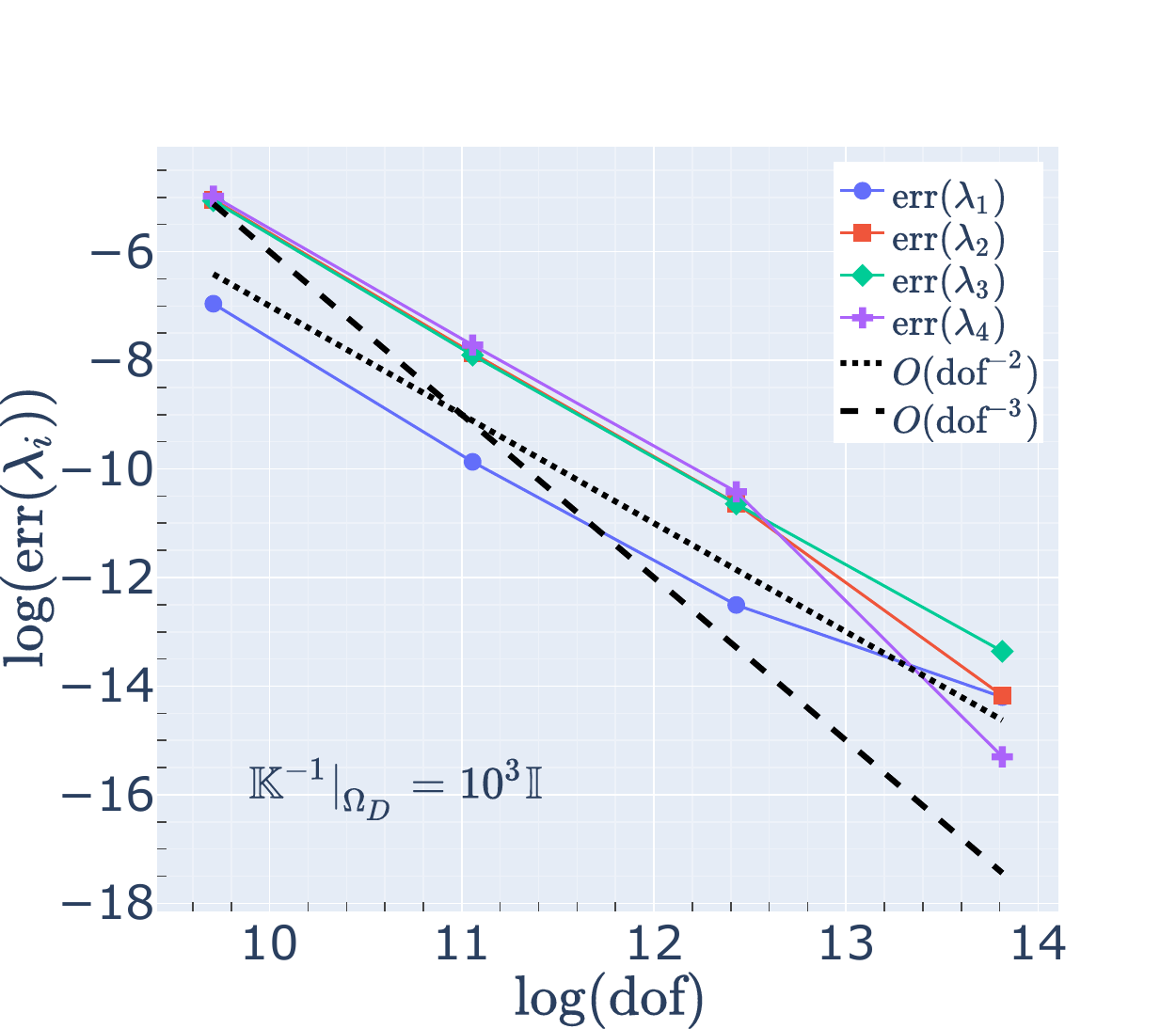}
	\end{minipage}
	\begin{minipage}{0.32\linewidth}\centering
		{\footnotesize $\varepsilon=-1, k=3$}\\
		\includegraphics[scale=0.217, trim= 0cm 0cm 2cm 2cm,clip]{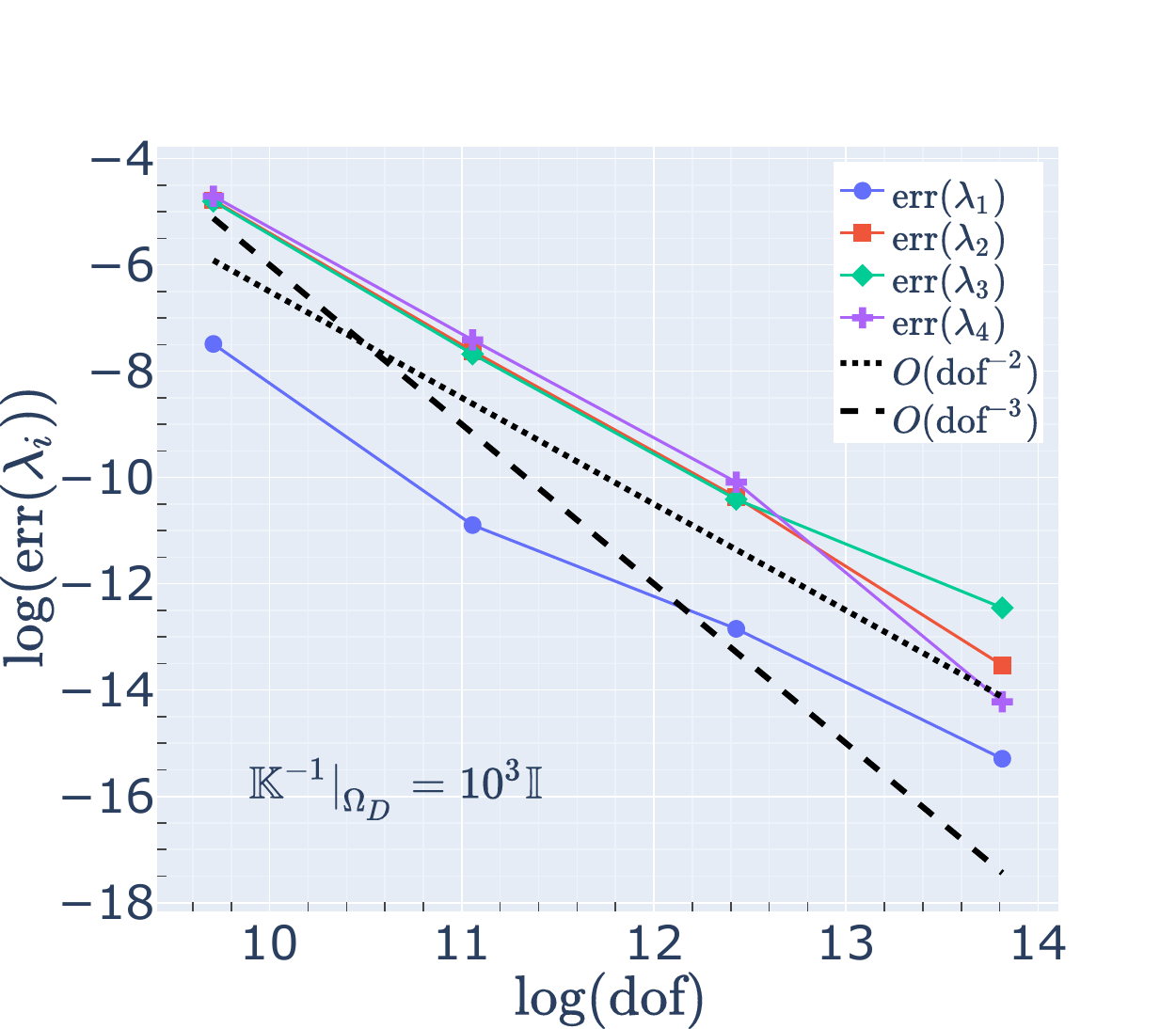}
	\end{minipage}
	\caption{Test \ref{subsec:square2D_convergence}. Convergence history for the first four eigenvalues on each IPDG scheme for a selected stability parameter $\texttt{a}=10k^2$ and $\mathbb{K}^{-1}\vert_{\Omega_D}=10^{3}$.}
	\label{fig:square_error_1e3}
\end{figure}

\subsection{Convergence on a 2D Lshaped porous domain with mixed boundary conditions}\label{subsec:lshape2D}
In this experment we put to the test the proposed scheme in the three variants of the method for a domain there there are singularities and mixed boundary conditions. The domain is the two-dimensional Lshape, defined as $\Omega:=(0,1)^2\backslash\left((0.5,0)\times(1,0.5)\right)$. We split the interior of $\Omega$ in such a way that there is an arrangement of different zones with given permeability parameters.  A sample of the meshed geometry is depicted in Figure \ref{fig:lshape-sample-domain}. Non-slip and \textit{do-nothing} boundary conditions are assumed on $\Gamma_1$ and  $\Gamma_2$, respectively. We take $\mathbb{K}$ such that $\mathbb{K}^{-1}=\boldsymbol{0}$ on $\Omega_S$, while $\mathbb{K}^{-1}=10^3\mathbb{I}$ on $\Omega_D$. On this test, the stabilization parameter is set to be $\texttt{a}=10$.

Taking as reference the suboptimal behavior observed in \cite{lepe2025jsc}, we test the a posteriori estimator for $k=1$ in the three variants of the proposed DG scheme, while, due to the suboptimality predicted in Theorem~\ref{theorem_ordendoble} and therefore the lost efficiency of the estimator, we only consider the symmetric case for the higher order $k=2$. 

%In Figure~\ref{fig:lshape2d:eigenmodes}, we present the computed eigenmodes, followed by the adaptive meshes obtained by the proposed method for $k=1$, shown in Figure~\ref{fig:lshape2d-meshes}. 
In Figure~\ref{fig:lshape2d-meshes} we present the adaptive meshes obtained by the DG variants for $k=1$. It is evident that our adaptive algorithm concentrates most of the refinements around the reentrant corner, as well as in regions with high pressure gradients. It is also worth noting that the number of elements marked inside the domain by the skew-symmetric method is higher than in the other schemes.

We conclude this test by presenting the error history and estimator efficiency for all IPDG schemes in Figures~\ref{fig:error-eff-lshape2d-nonsymmetric}. In all cases, convergence rates of double order are observed, and the estimator remains both reliable and efficient, remaining bounded away from zero. For the symmetric case, a convergence rate of $\mathcal{O}(h^{2k})$ is clearly achieved. For the non-symmetric methods, the error curves and estimator efficiency exhibit very similar behavior and are optimal for $k=1$.

\begin{figure}[!hpbt]\centering
	\includegraphics[scale=0.15]{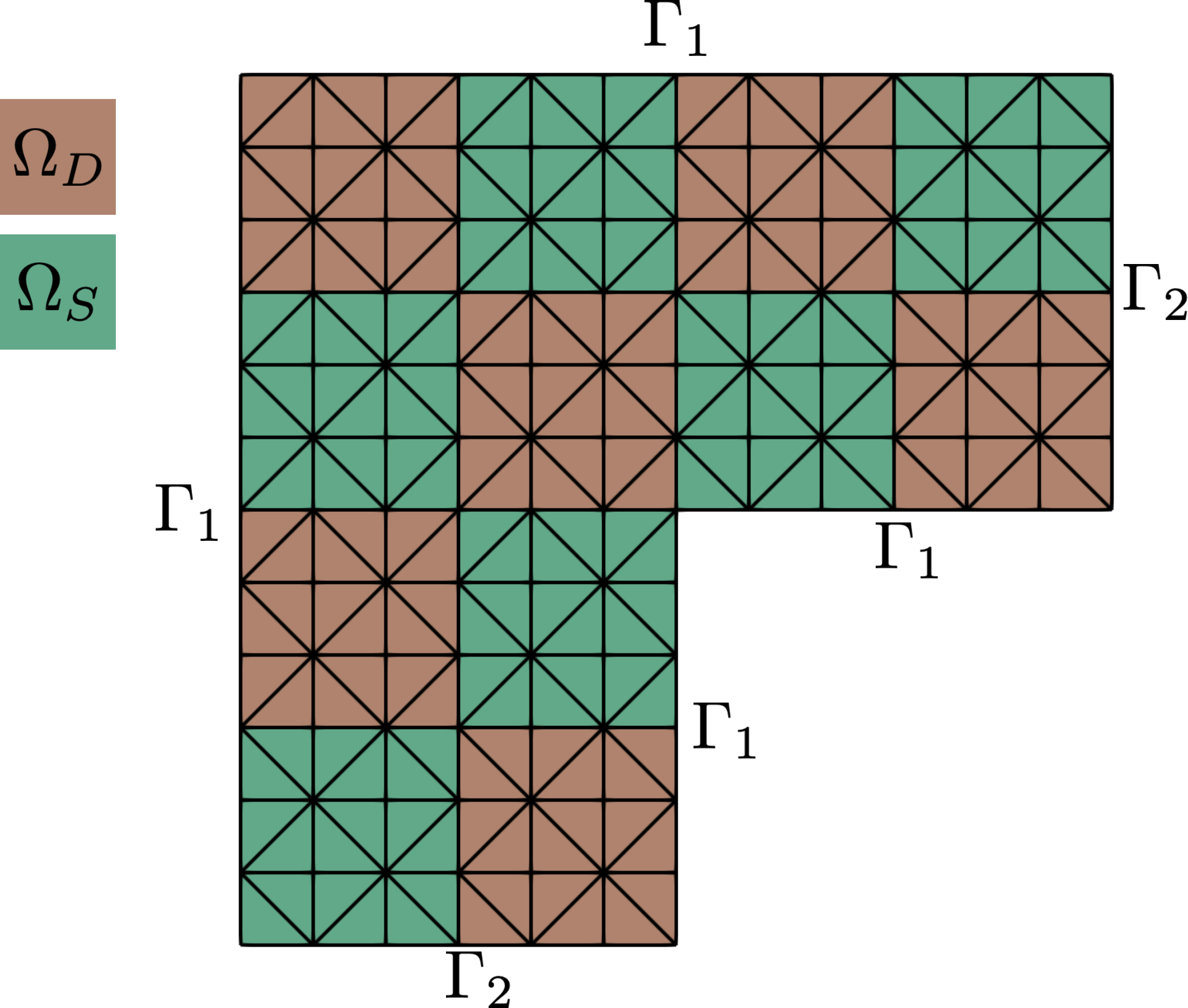}
	\caption{Test \ref{subsec:lshape2D}. Sample geometry of an Lshaped domain with a chessboard-like distribution of permeability regions and $N=10$.}
	\label{fig:lshape-sample-domain}
\end{figure}

\begin{figure}[!hpbt]\centering
	\begin{minipage}{0.32\linewidth}\centering
		{\footnotesize $\lambda_{h,1}, \varepsilon=1, \texttt{dof}=174678$}\\
		\includegraphics[scale=0.09,trim=28.2cm 8cm 28.2cm 8cm,clip]{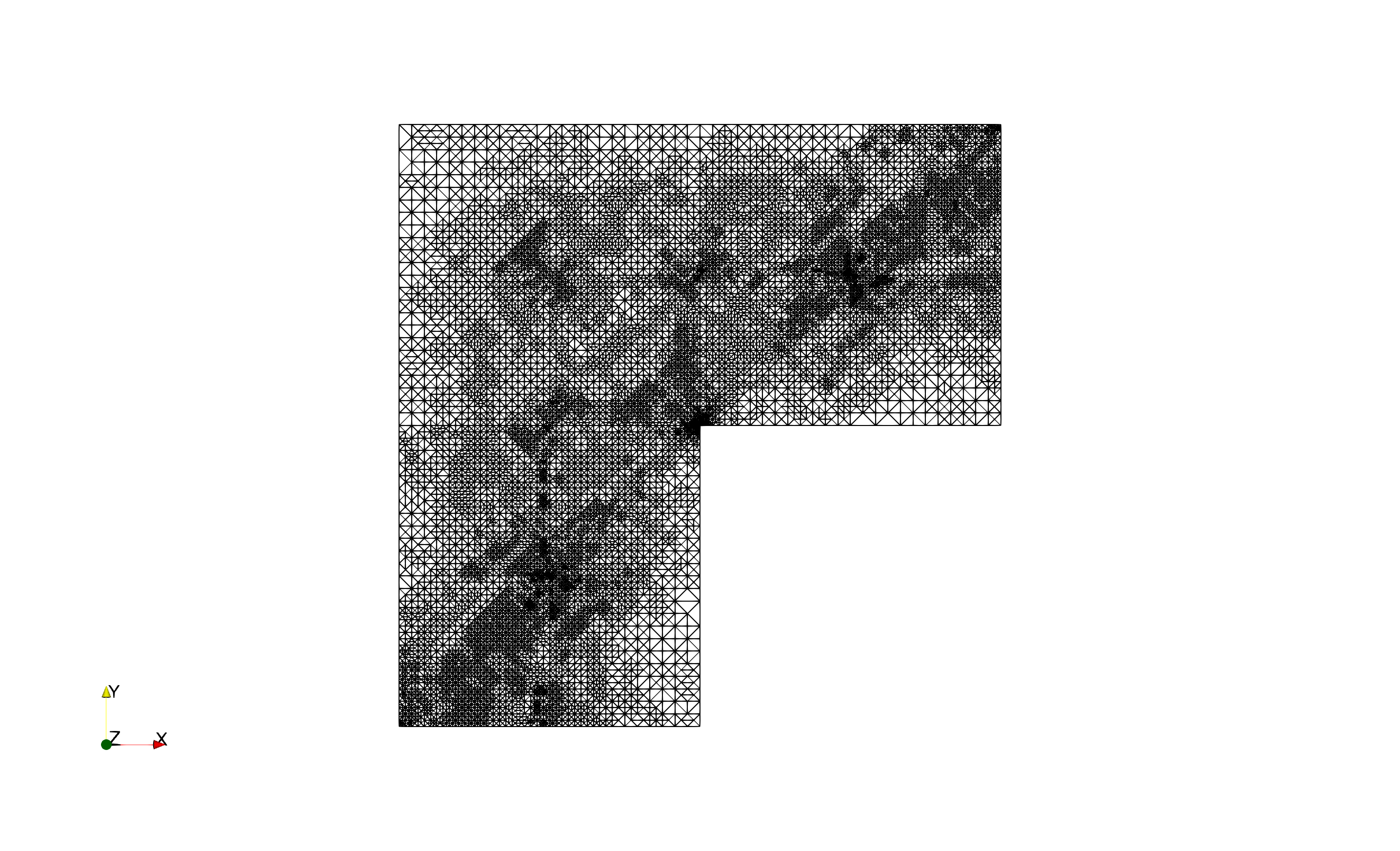}
	\end{minipage}
	\begin{minipage}{0.32\linewidth}\centering
		{\footnotesize $\lambda_{h,1}, \varepsilon=0, \texttt{dof}=204862$}\\
		\includegraphics[scale=0.09,trim=28.2cm 8cm 28.2cm 8cm,clip]{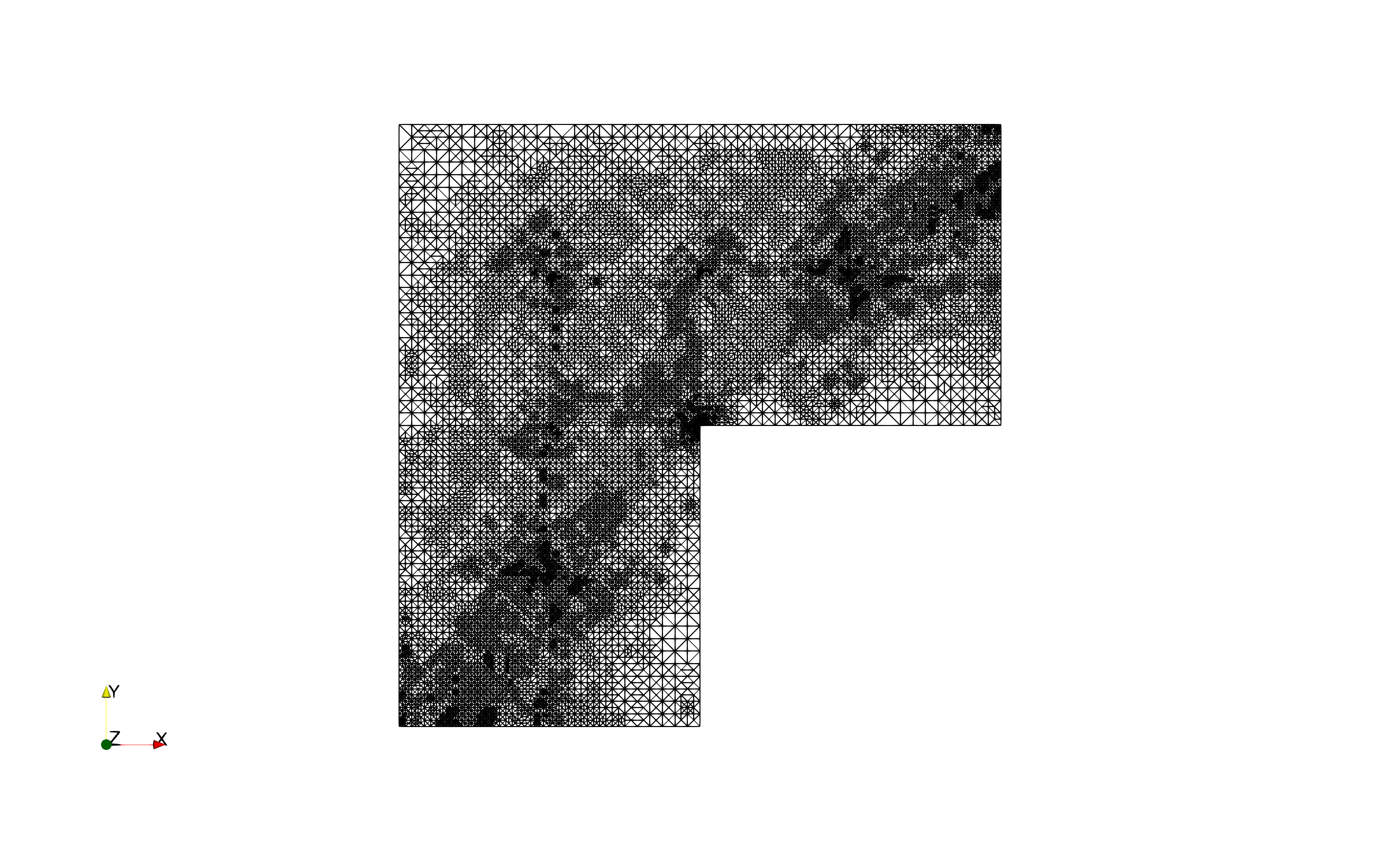}
	\end{minipage}
	\begin{minipage}{0.32\linewidth}\centering
		{\footnotesize $\lambda_{h,1},\varepsilon=-1, \texttt{dof}=229579$}\\
		\includegraphics[scale=0.09,trim=28.2cm 8cm 28.2cm 8cm,clip]{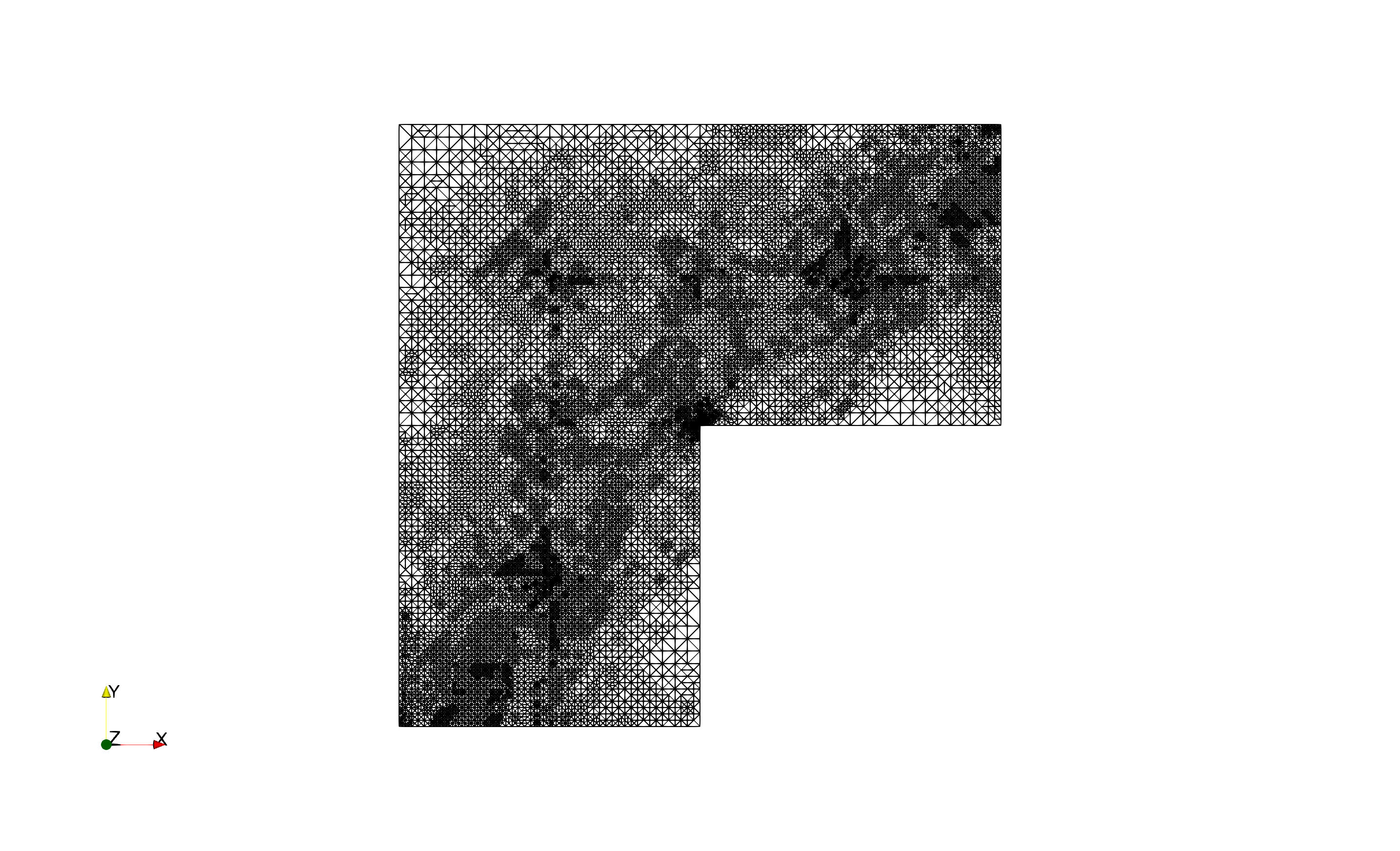}
	\end{minipage}\\
	\begin{minipage}{0.32\linewidth}\centering
		{\footnotesize $\lambda_{h,4}, \varepsilon=1, \texttt{dof}=166467$}\\
		\includegraphics[scale=0.09,trim=28.2cm 8cm 28.2cm 8cm,clip]{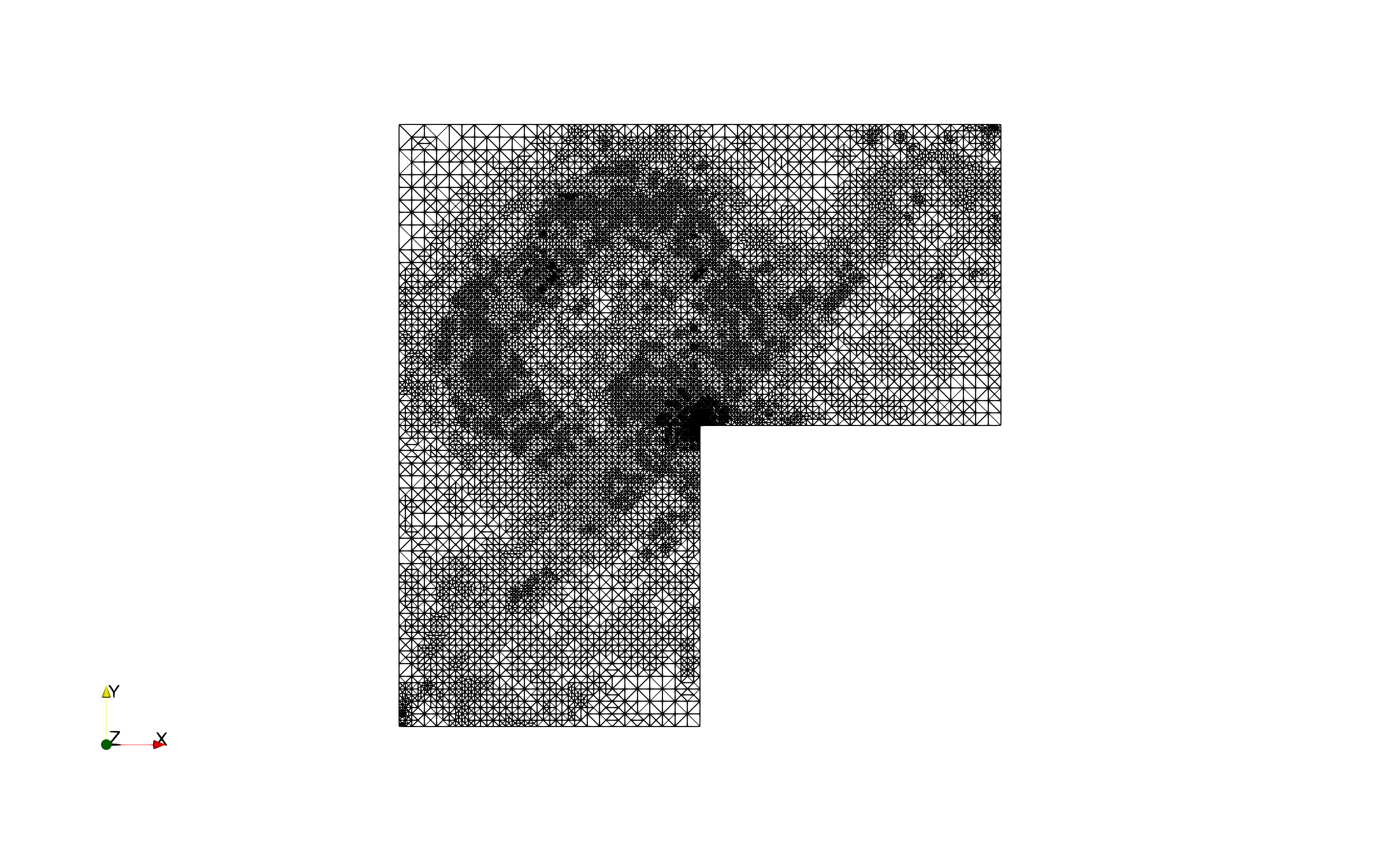}
	\end{minipage}
	\begin{minipage}{0.32\linewidth}\centering
		{\footnotesize $\lambda_{h,4}, \varepsilon=0, \texttt{dof}=167160$}\\
		\includegraphics[scale=0.09,trim=28.2cm 8cm 28.2cm 8cm,clip]{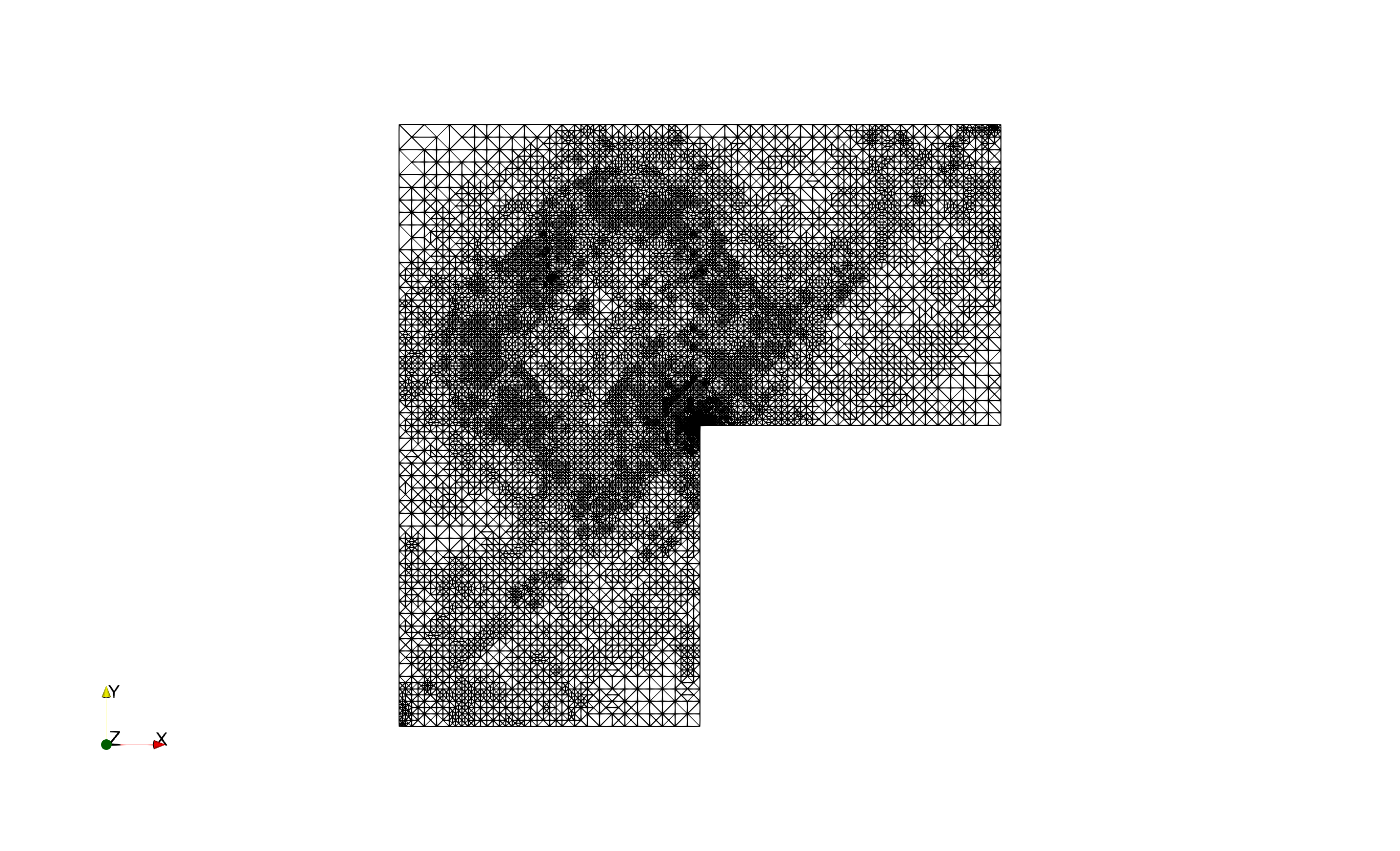}
	\end{minipage}
	\begin{minipage}{0.32\linewidth}\centering
		{\footnotesize $\lambda_{h,4}, \varepsilon=-1, \texttt{dof}=178409$}\\
		\includegraphics[scale=0.09,trim=28.2cm 8cm 28.2cm 8cm,clip]{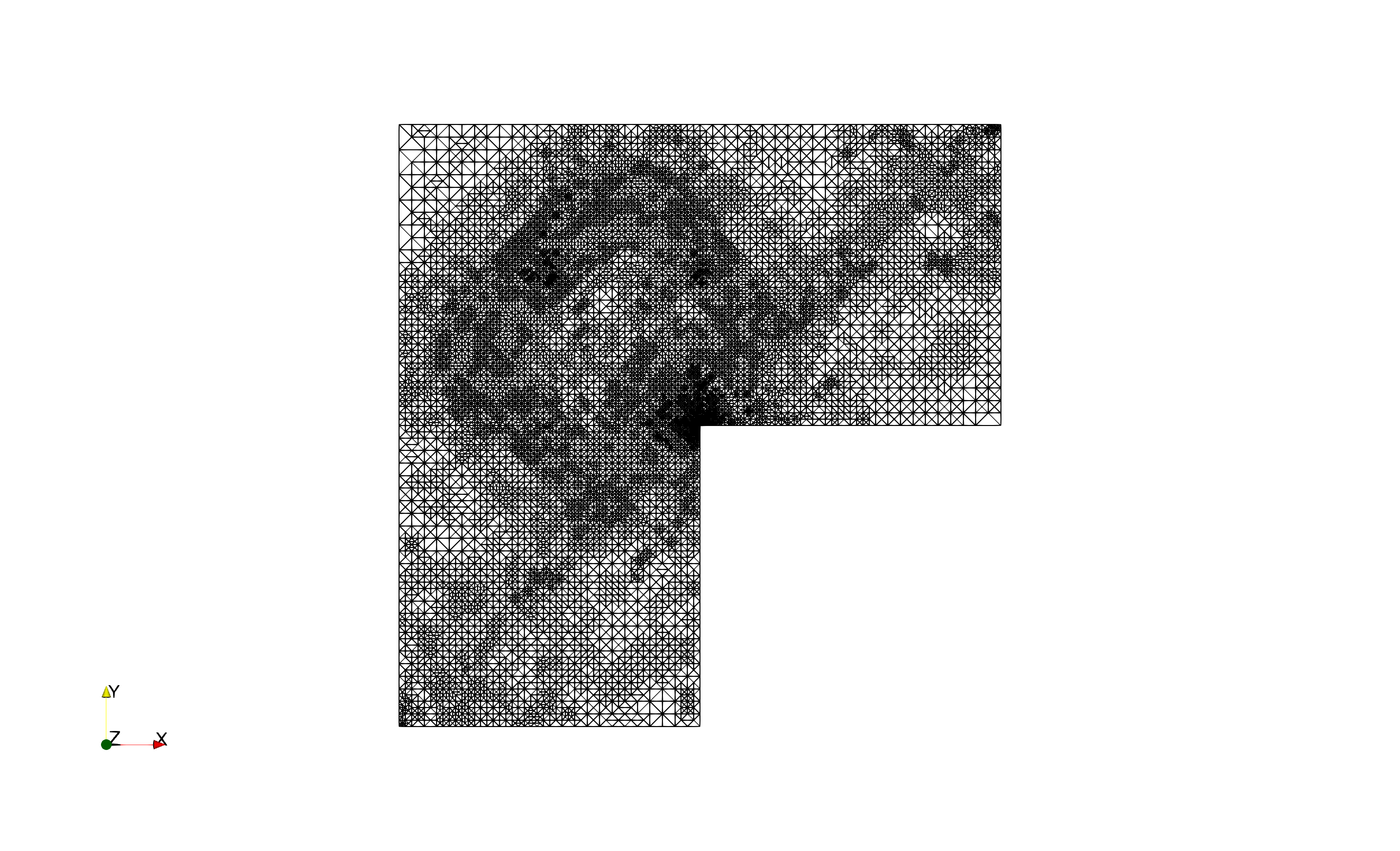}
	\end{minipage}\\
	\caption{Test \ref{subsec:lshape2D}. Last adaptive meshes for the first and fourth computed eigenvalue with for all the variants of the IPDG method in the Lshaped geometry with mixed boundary conditions and $\mathbb{K}^{-1}:=10^{3}\mathbb{I}$ on $\Omega_D$.}
	\label{fig:lshape2d-meshes}
\end{figure}

\begin{figure}[!hpbt]\centering
	\begin{minipage}{0.49\linewidth}\centering
		\includegraphics[scale=0.32,trim=0cm 0cm 2cm 2cm,clip]{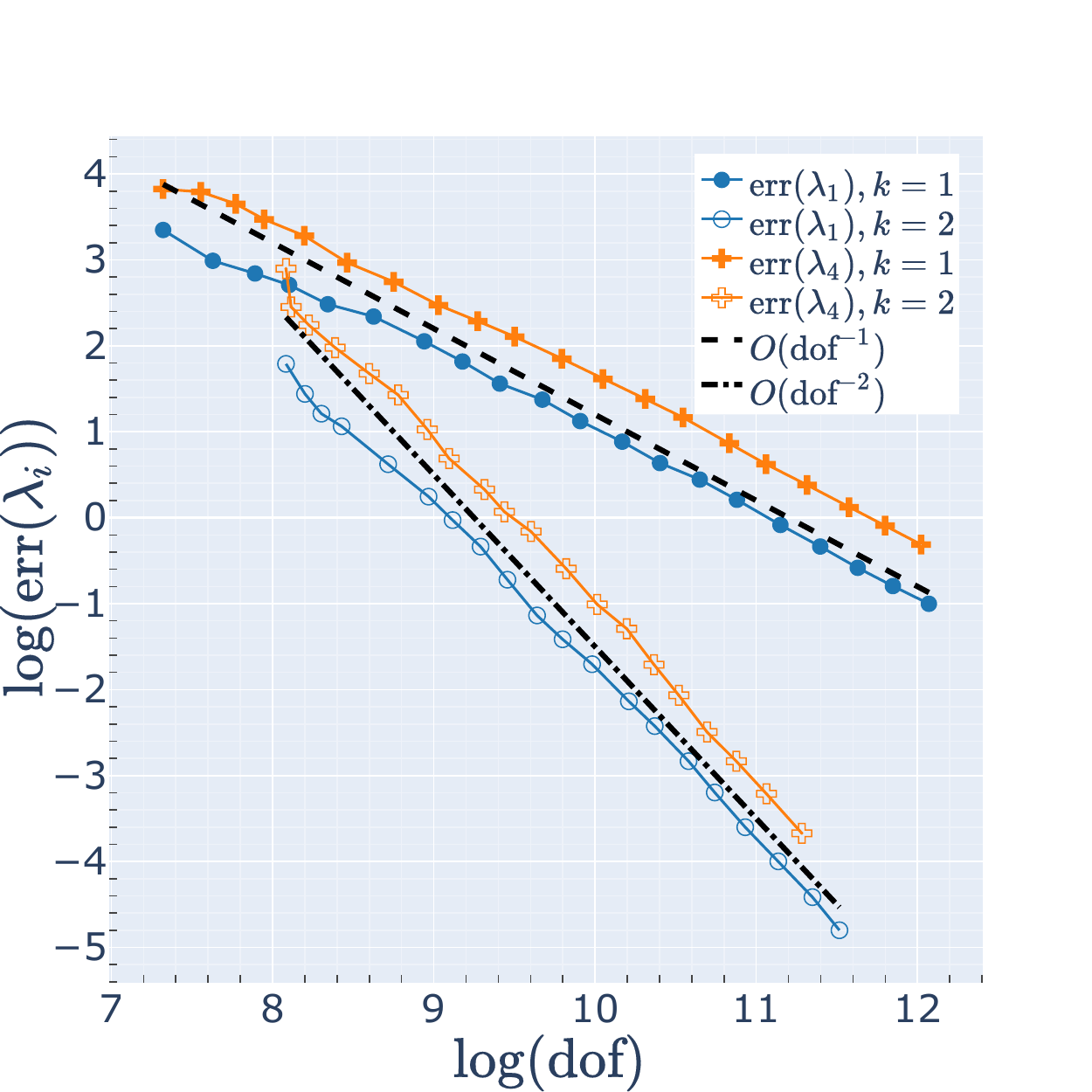}
	\end{minipage}
	\begin{minipage}{0.49\linewidth}\centering
		\includegraphics[scale=0.32,trim=0cm 0cm 2cm 2cm,clip]{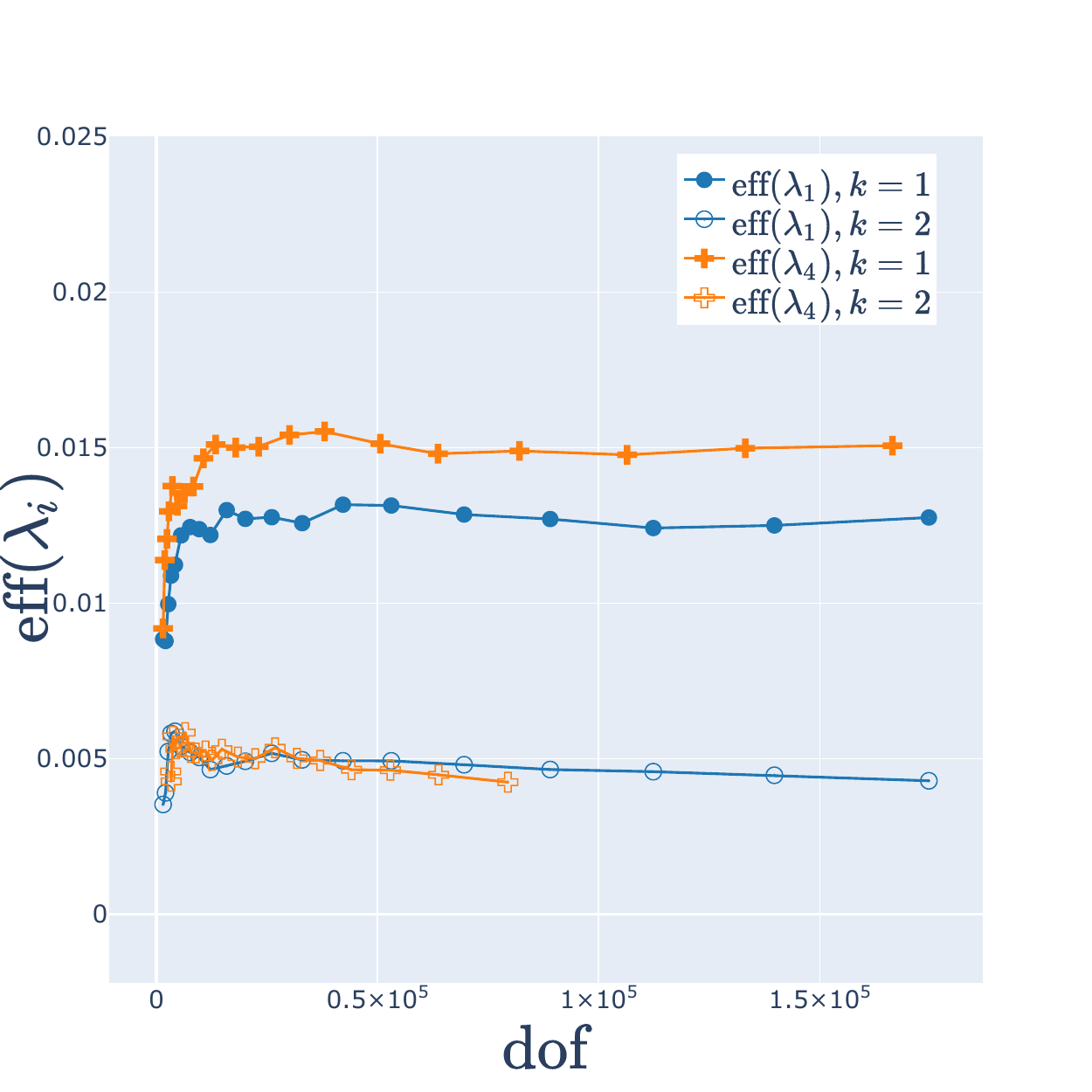}
	\end{minipage}
	\caption{Test \ref{subsec:lshape2D}. Error history for the adaptive refinements when computing the first and fourth eigenvalue (left) together with their corresponding effectivity indexes (right) in the symmetric case ($\varepsilon=1$).}
	\label{fig:error-eff-lshape2d-symmetric}
\end{figure}

\begin{figure}[!hpbt]\centering
	\begin{minipage}{0.49\linewidth}\centering
		\includegraphics[scale=0.32,trim=0cm 0cm 2cm 2cm,clip]{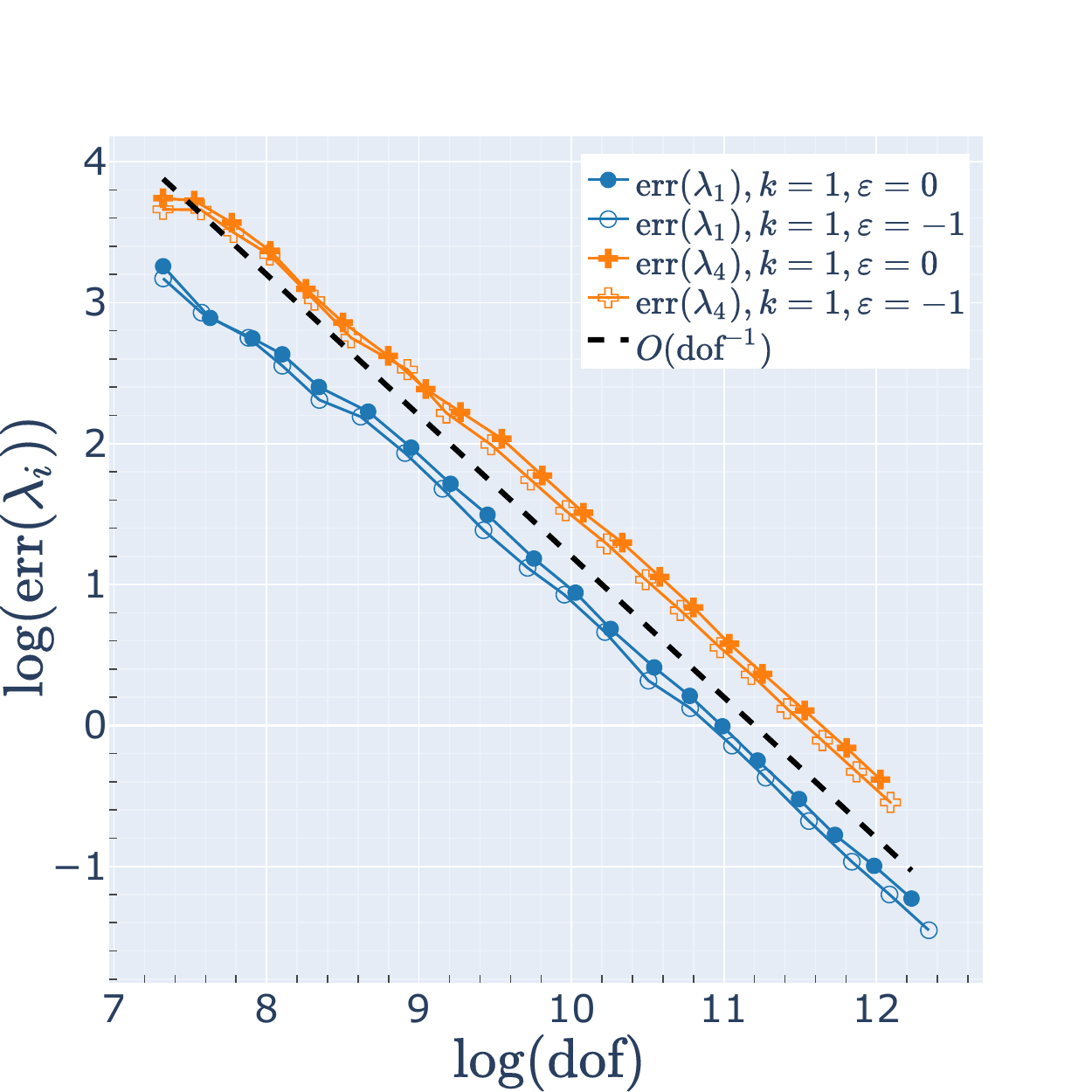}
	\end{minipage}
	\begin{minipage}{0.49\linewidth}\centering
		\includegraphics[scale=0.32,trim=0cm 0cm 2cm 2cm,clip]{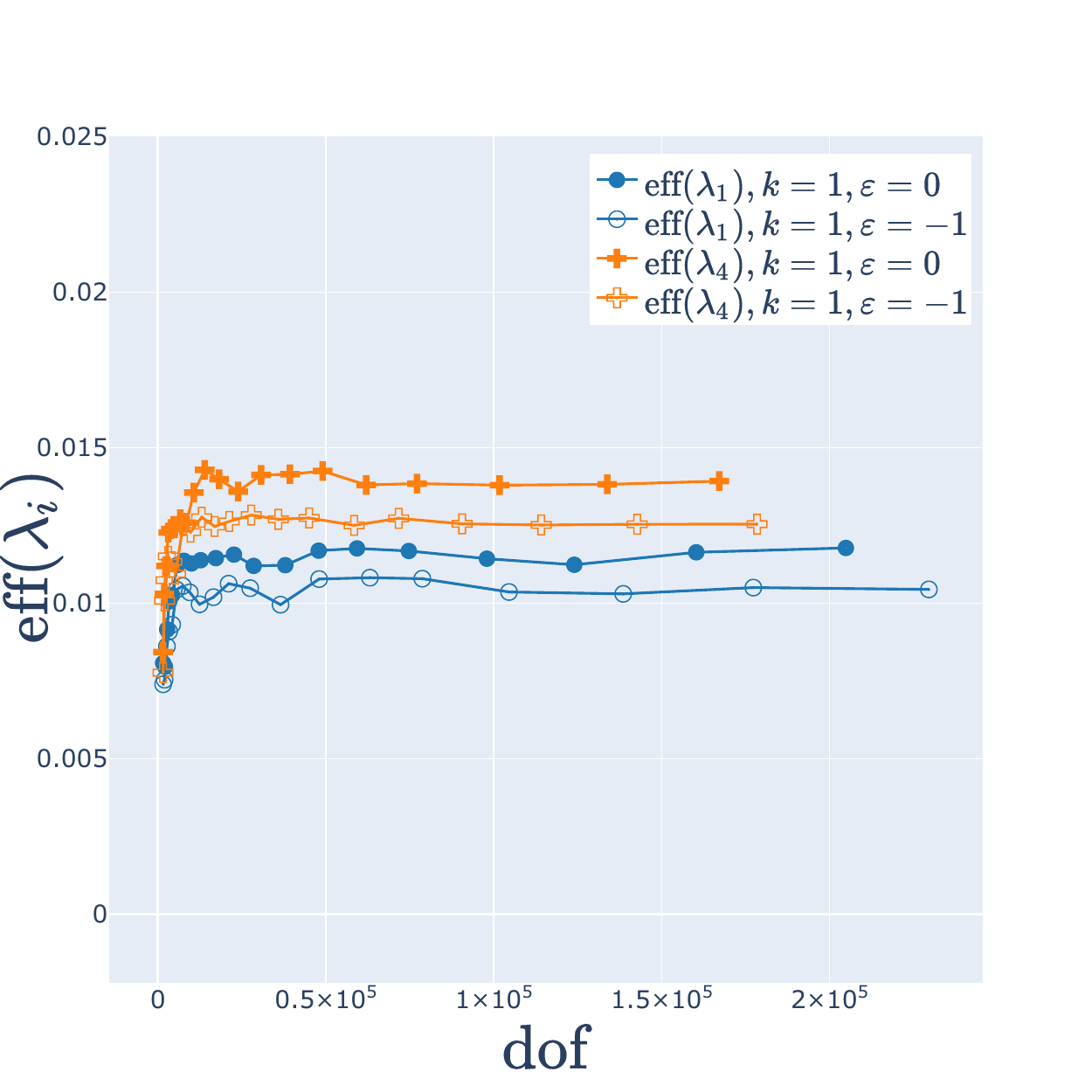}
	\end{minipage}
	\caption{Test \ref{subsec:lshape2D}. Error history for the adaptive refinements when computing the first and fourth eigenvalue (left) together with their corresponding effectivity indexes (right) in the incomplete and skew-symmetric  cases.}
	\label{fig:error-eff-lshape2d-nonsymmetric}
\end{figure}

\subsection{3D channel with a porous obstacle}\label{subsec:3d-channel}
We end the numerical section by presenting some results of the DG method on three-dimensional domains. For simplicity, we only consider the symmetric case $\varepsilon=1$. The domain under study is a box defined by $\Omega:=(0,1)\times(0,1)\times(0,3)$. Within this domain, we define the permeability parameter $\mathbb{K}$ as
$$
\mathbb{K}^{-1}=\left\{
\begin{aligned}
	&\kappa\mathbb{I}&\text{if } (x,y)\in\Omega_D,\\
	&\boldsymbol{0}, & \text{if } (x,y)\in\Omega_S,
\end{aligned}
\right.
$$
where $\Omega_D:=(0,1)\times(1/3,2/3)\times(4/3,5/3)$ and $\Omega_S=\Omega\backslash\Omega_D$. We choose $\mathbb{K}^{-1}=10^{3}\mathbb{I}$. This choice allows to have a membrane-like behavior with partial permeability across $\Omega_D$. A graphical description of the domain is portrayed in Figure \ref{fig:3d-channel-init-mesh}.

We solve the eigenvalue problem with $k=1,2$ and obtain the extrapolated discrete eigenvalue $\lambda_{1}=33.70064$, which is considered as the exact solution. Then, we perform 10 adaptive iterations for $k=1$ and 9 iterations for $k=2$ in order to observe the convergence rates and the reliability/efficiency of the estimator.  The stabilization parameter is set to be $\texttt{a}=10$.

We present the corresponding lowest eigenmodes in Figure~\ref{fig:3d-channel-uh-ph}. Here, we observe the velocity field across the domain, entering and exiting through $\Gamma_2$, and we also note that some of the fluid, although with low magnitude, passes through the porous subdomain. This mild porosity causes high pressure gradients, represented by a concentrated cloud of points around $\Omega_D$. The eigenmode behavior is detected by the estimator and the adaptive algorithm, which marks the elements near the boundary of $\Omega_D$ for refinement. Some samples of the adaptive meshes for $k=1,2$ are presented in Figure~\ref{fig:3d-channel-adaptive-meshes}, where critical singular zones are refined as expected. Similar to the 2D case, fewer elements are marked to achieve optimal rates for $k=2$.

The error and effectivity indices for the symmetric DG scheme are depicted in Figure~\ref{fig:error-eff-channel3d-symmetric}. A rate $\mathcal{O}(\texttt{dof}^{-2k/d})$ is observed for $k\geq 1$ and $d=3$, implying an $h$-convergence rate of $\mathcal{O}(h^{2k})$. Moreover, the estimator effectivity remains properly bounded, demonstrating the reliability and efficiency of the estimator in the three-dimensional case with mixed boundary conditions.

\begin{figure}[!hpbt]\centering
\includegraphics[scale=0.14, trim=0cm 0cm 0cm 0cm,clip]{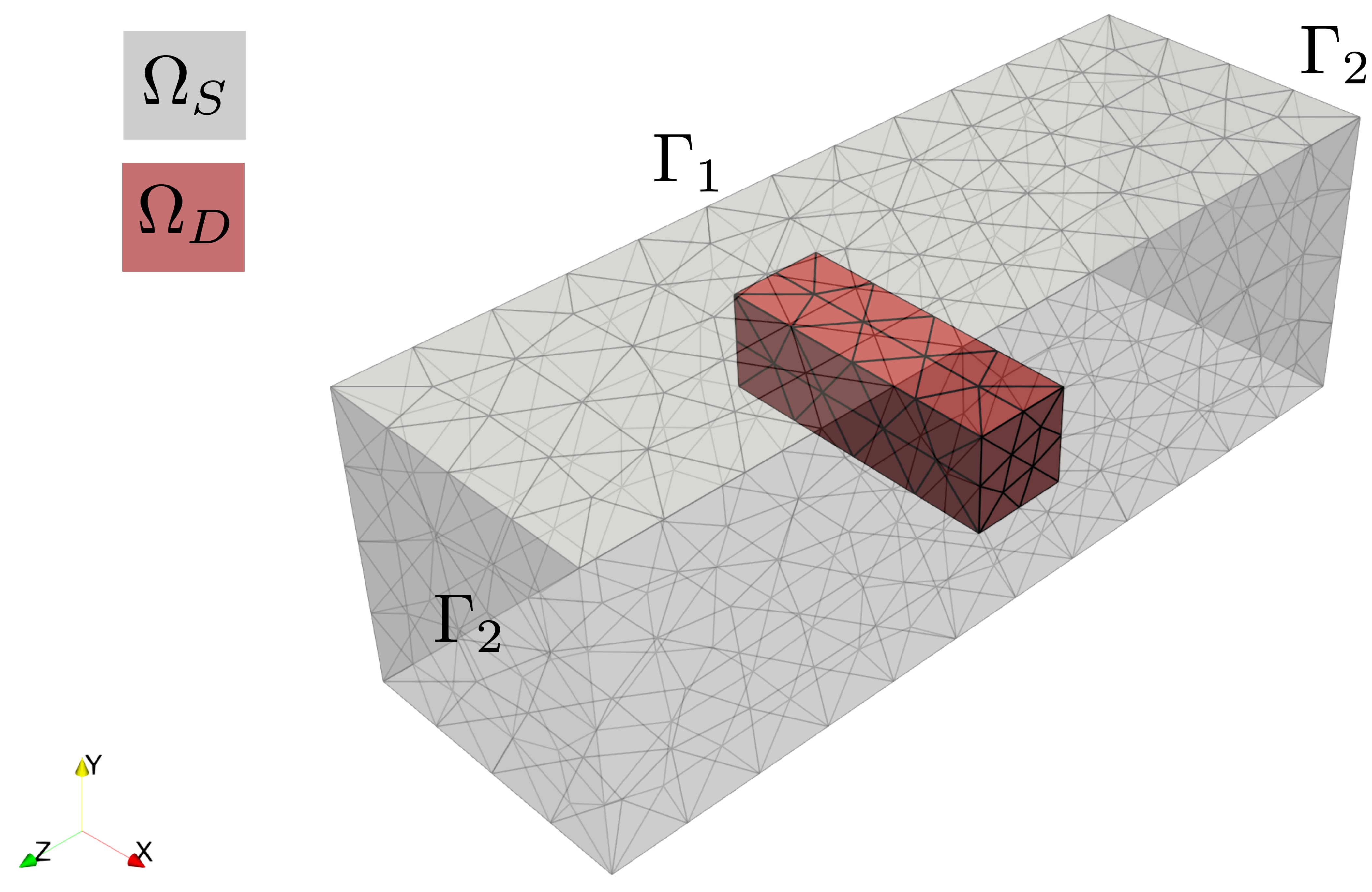}
\caption{Test \ref{subsec:3d-channel}. The channel domain with the initial mesh configuration consisting of 2616 elements. }
\label{fig:3d-channel-init-mesh}
\end{figure}

\begin{figure}[!hpbt]\centering
\begin{minipage}{\linewidth}\centering
	{\footnotesize $\bu_{h,1}$}\\
\includegraphics[scale=0.09,trim=4cm 5cm 10cm 14cm,clip]{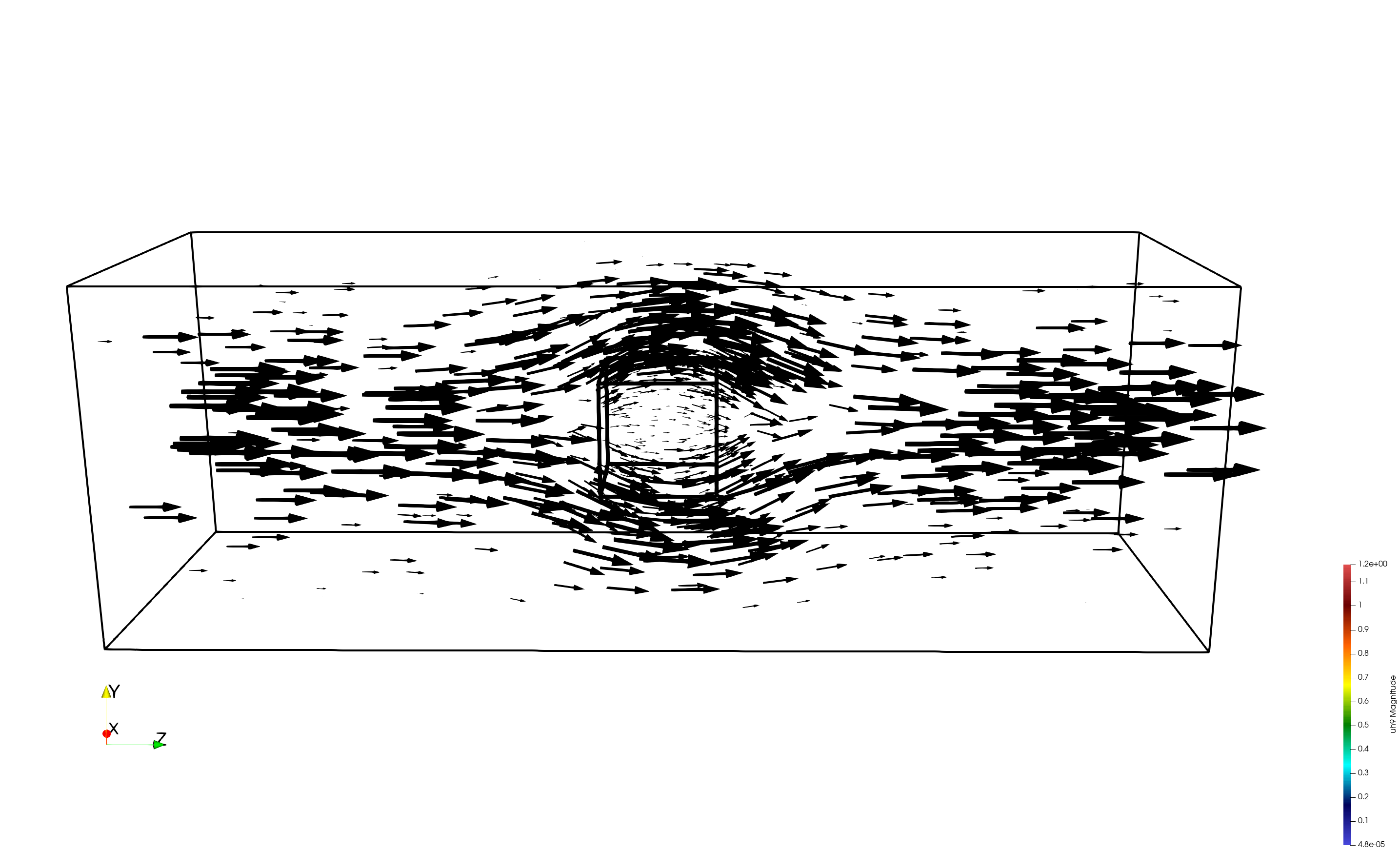}
\end{minipage}\\
\begin{minipage}{\linewidth}\centering
	{\footnotesize $p_{h,1}$}\\
	\includegraphics[scale=0.09,trim=4cm 5cm 10cm 14cm,clip]{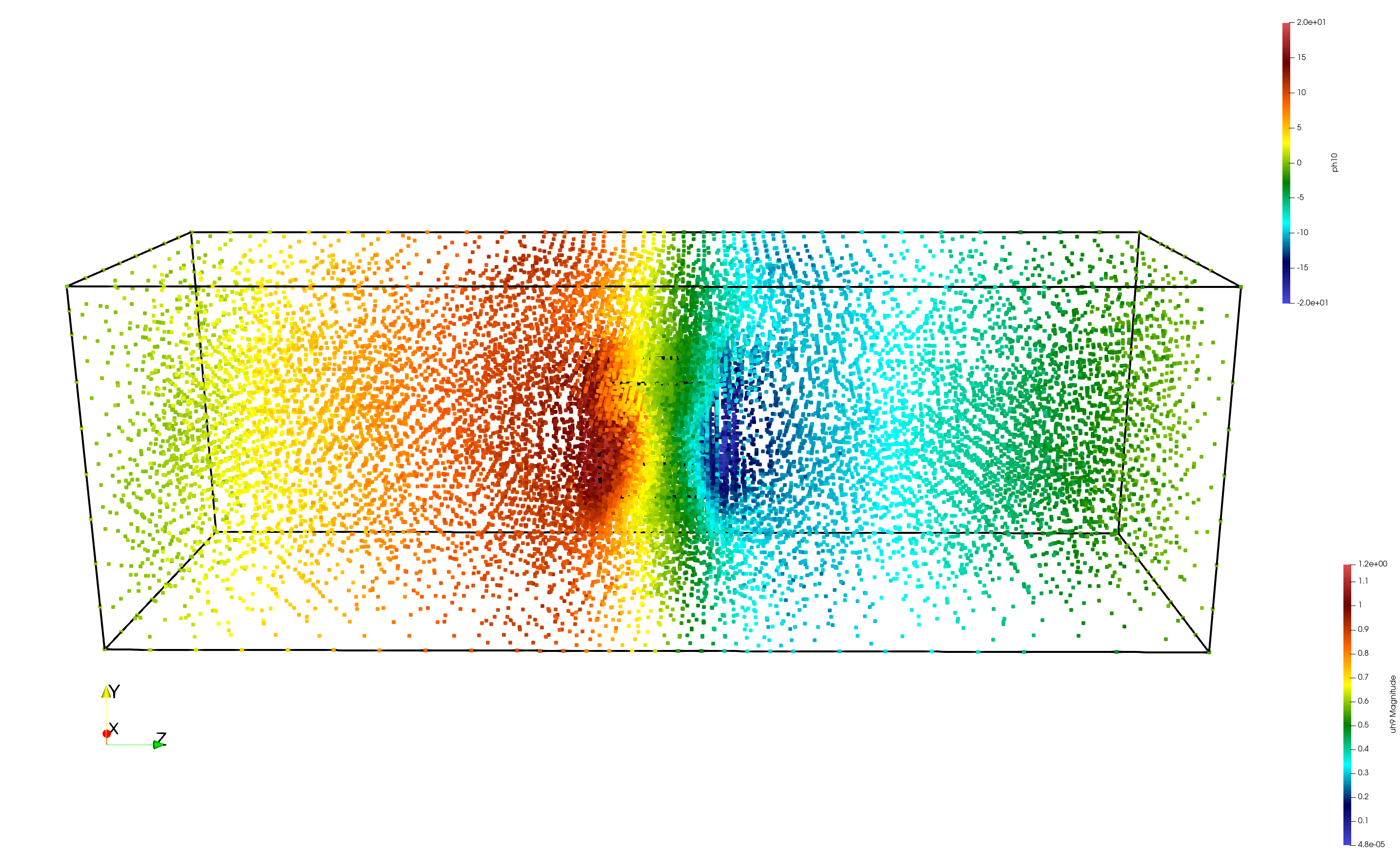}
\end{minipage}
\caption{Test \ref{subsec:3d-channel}. First lowest computed eigenmodes represented as the velocity field (top) and pressure dots cloud (bottom) in the last adaptive iteration.}
\label{fig:3d-channel-uh-ph}
\end{figure}

\begin{figure}[!hpbt]\centering
	\begin{minipage}{0.32\linewidth}\centering
		{\footnotesize $\texttt{dof}=76635$}\\
		\includegraphics[scale=0.1,trim=18cm 3cm 22cm 7cm,clip]{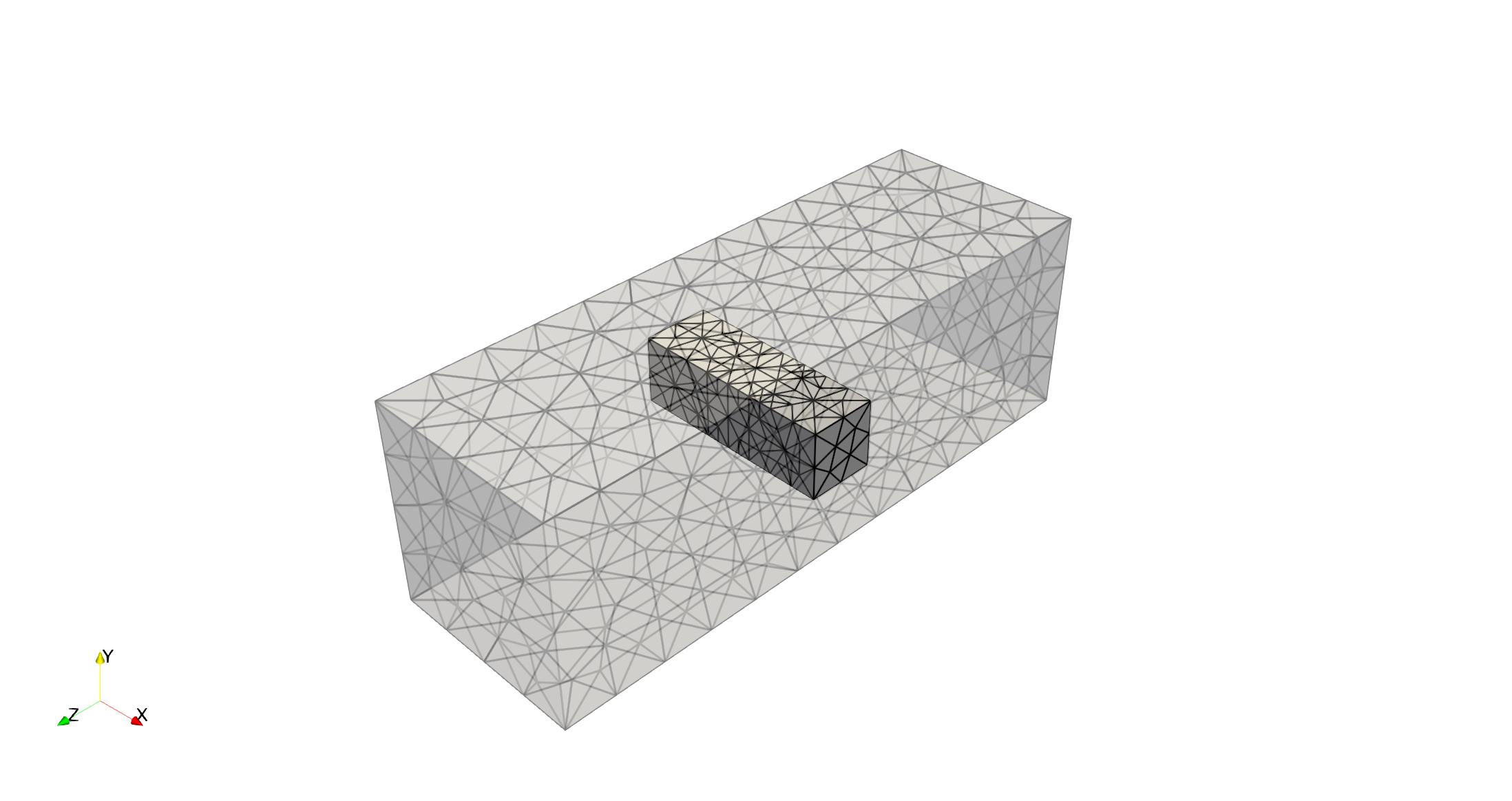}
	\end{minipage}
	\begin{minipage}{0.32\linewidth}\centering
		{\footnotesize $\texttt{dof}=324402$}\\
		\includegraphics[scale=0.1,trim=18cm 3cm 22cm 7cm,clip]{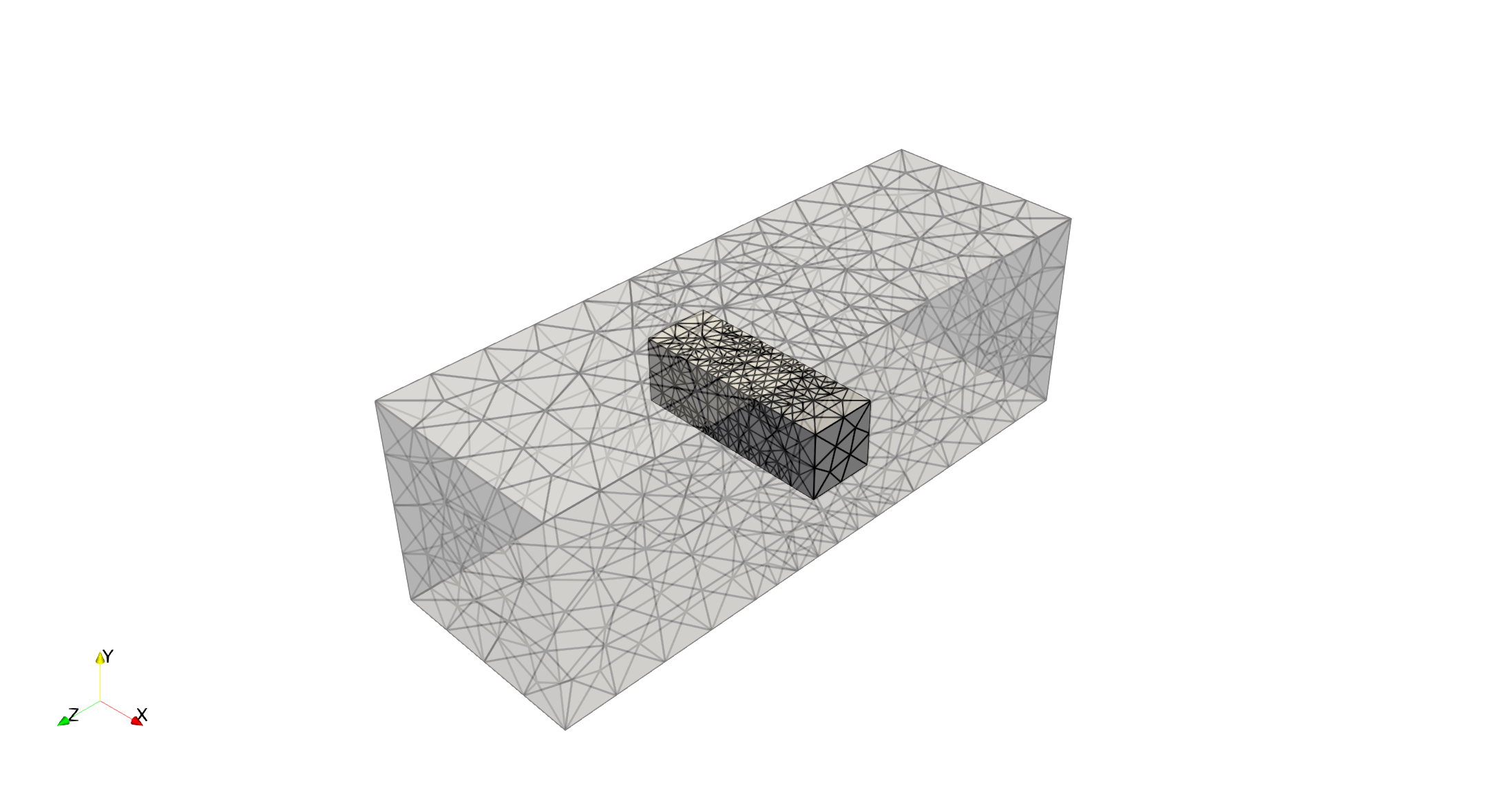}
	\end{minipage}
	\begin{minipage}{0.32\linewidth}\centering
		{\footnotesize $\texttt{dof}=1532895$}\\
		\includegraphics[scale=0.1,trim=18cm 3cm 22cm 7cm,clip]{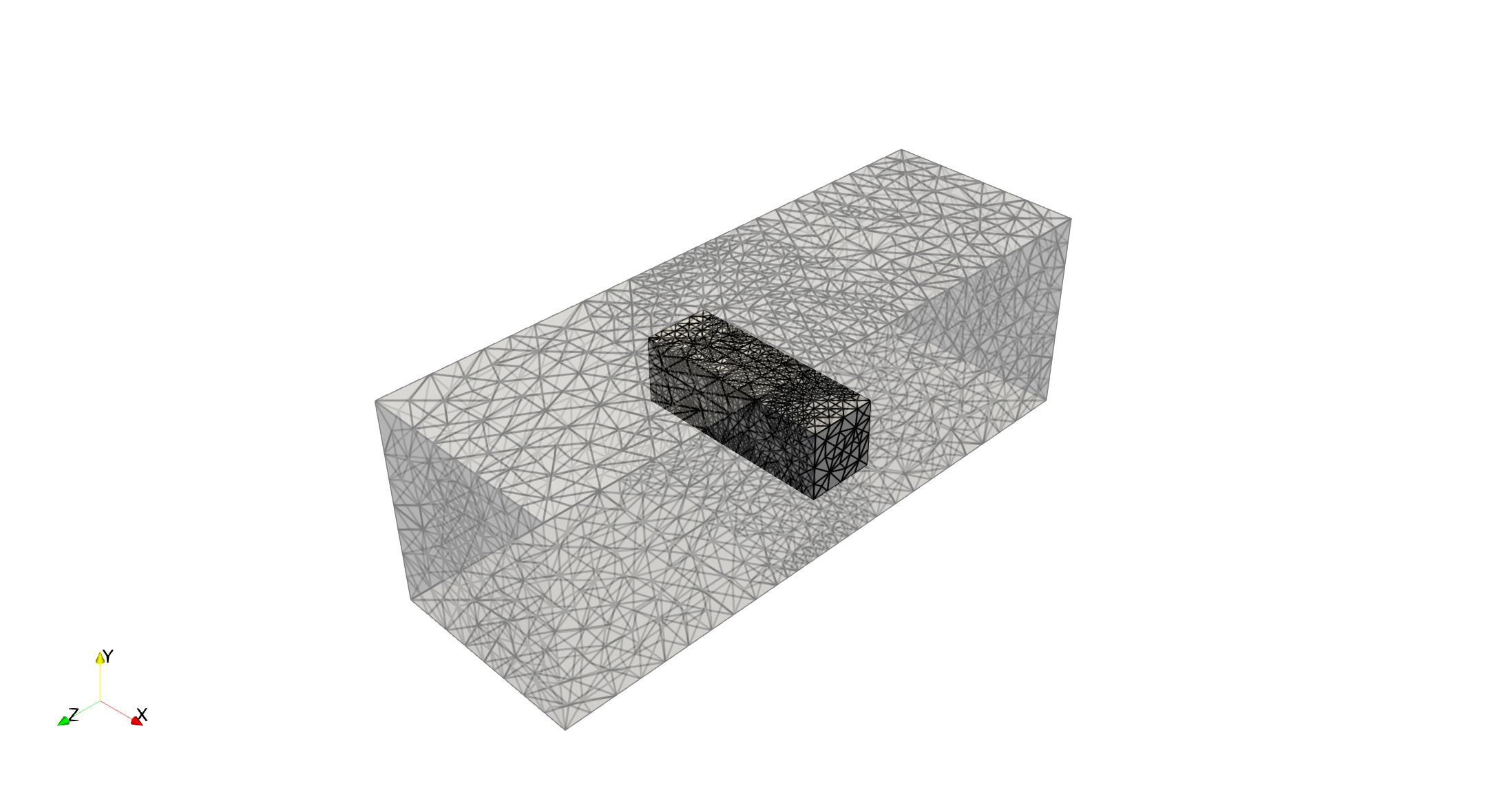}
	\end{minipage}\\
	\begin{minipage}{0.32\linewidth}\centering
		{\footnotesize $\texttt{dof}=166600$}\\
		\includegraphics[scale=0.1,trim=18cm 3cm 22cm 7cm,clip]{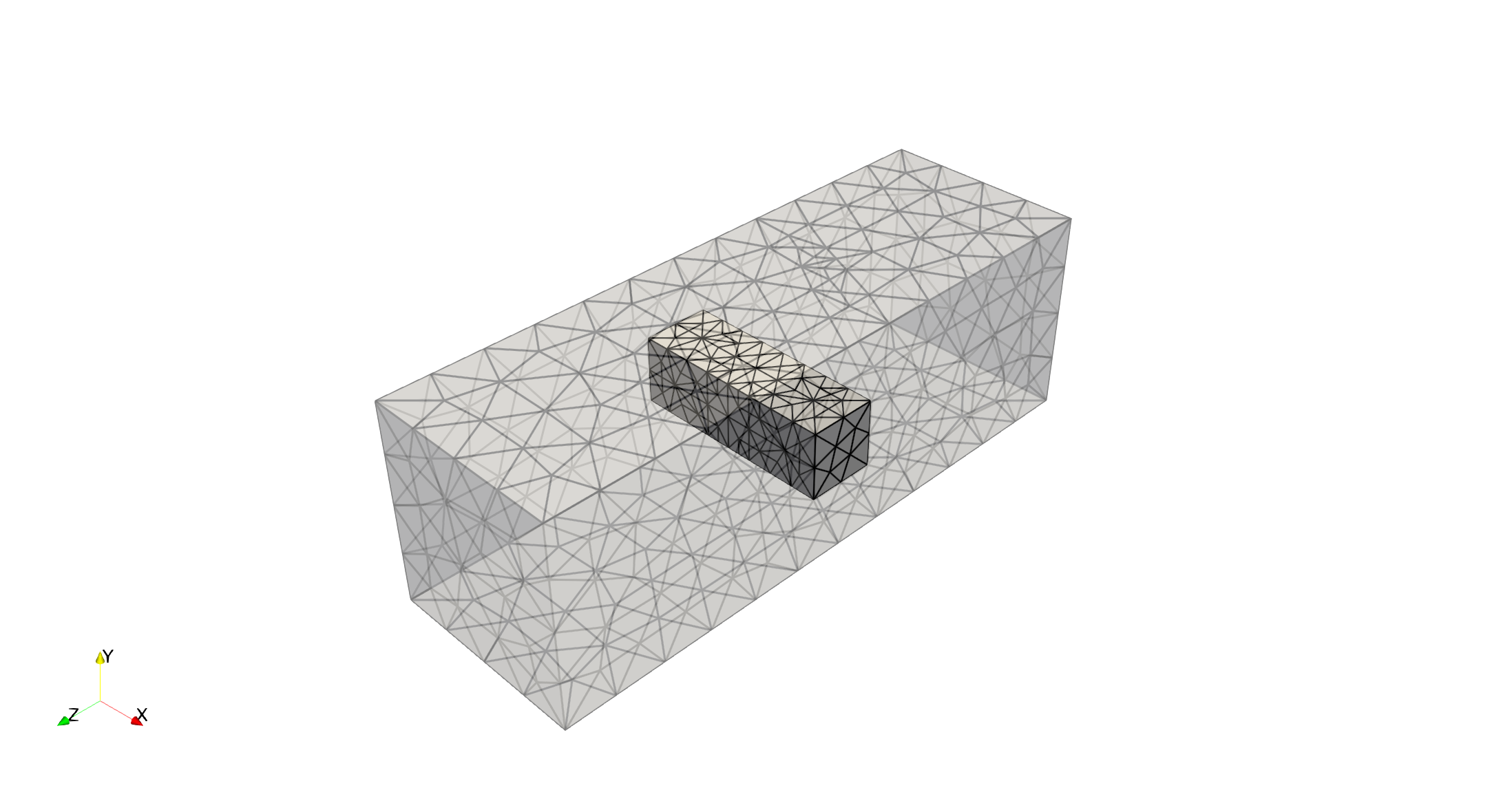}
	\end{minipage}
	\begin{minipage}{0.32\linewidth}\centering
		{\footnotesize $\texttt{dof}=534344$}\\
		\includegraphics[scale=0.1,trim=18cm 3cm 22cm 7cm,clip]{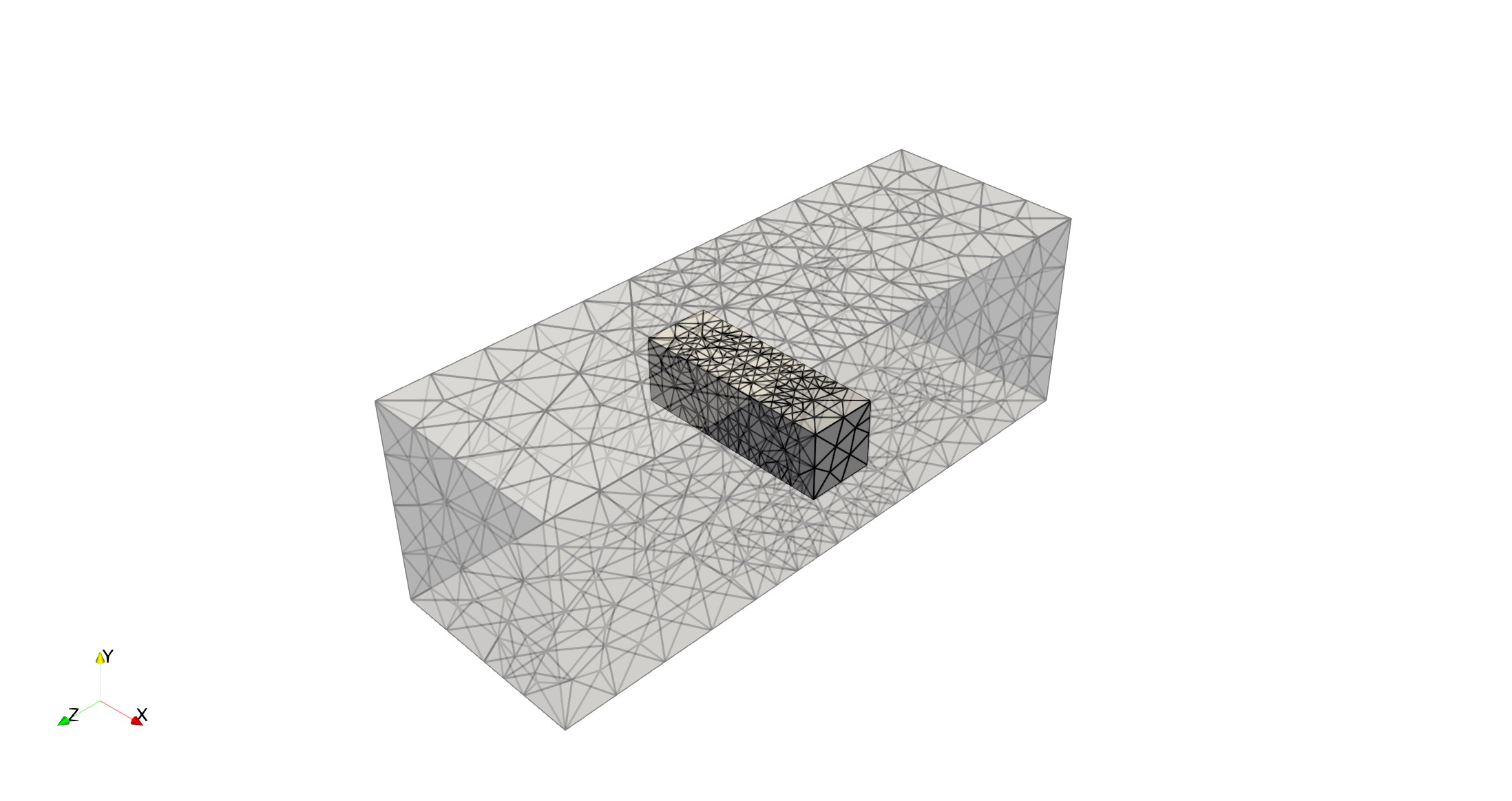}
	\end{minipage}
	\begin{minipage}{0.32\linewidth}\centering
		{\footnotesize $\texttt{dof}=1200982$}\\
		\includegraphics[scale=0.1,trim=18cm 3cm 22cm 7cm,clip]{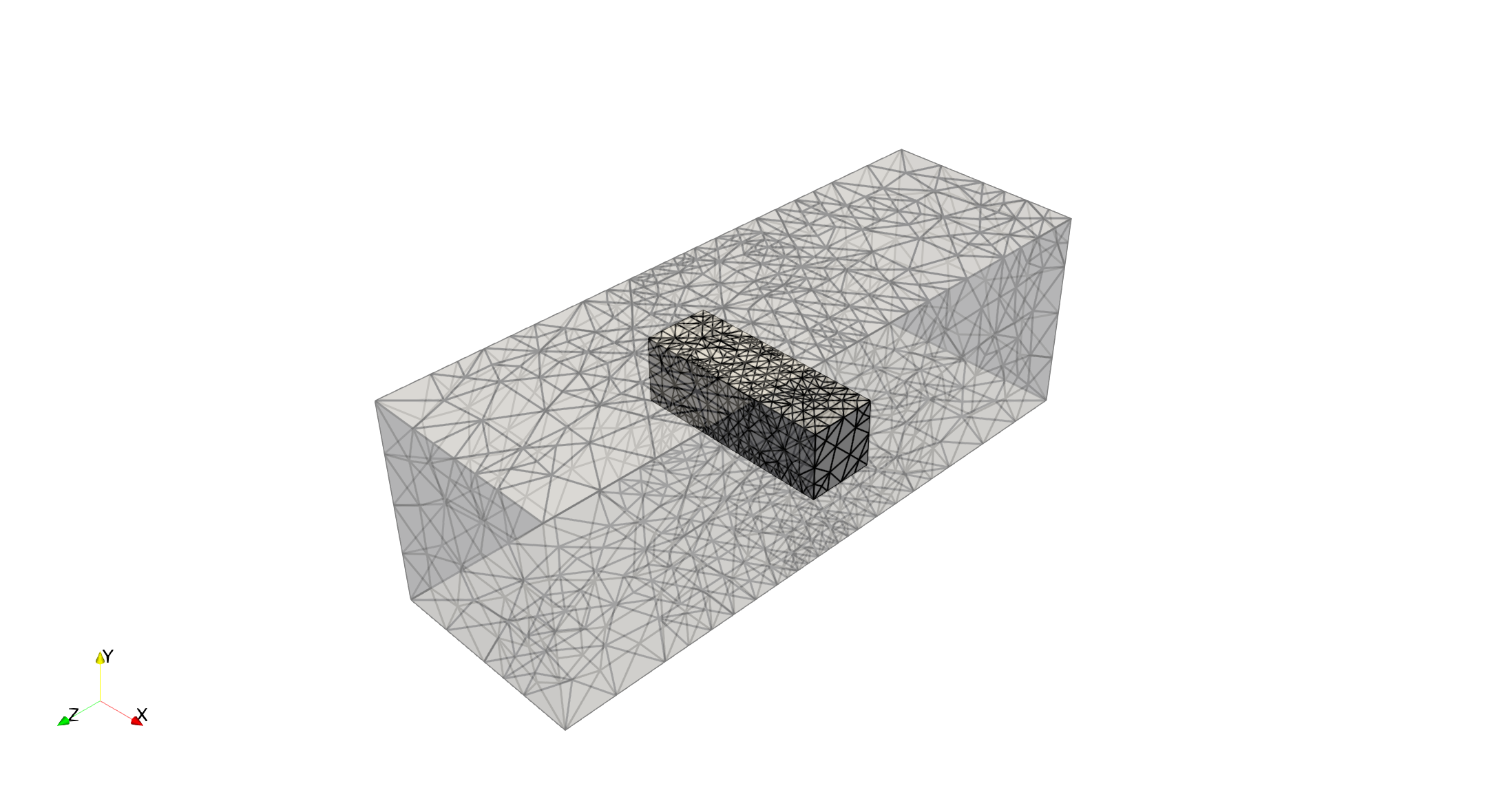}
	\end{minipage}
	\caption{Test \ref{subsec:3d-channel}. Intermediate adaptive meshes for $k=1,2,$ in the symmetric IPDG scheme ($\varepsilon=1$.)}
	\label{fig:3d-channel-adaptive-meshes}
\end{figure}

\begin{figure}[!hpbt]\centering
	\begin{minipage}{0.49\linewidth}\centering
		\includegraphics[scale=0.32,trim=0cm 0cm 2cm 2cm,clip]{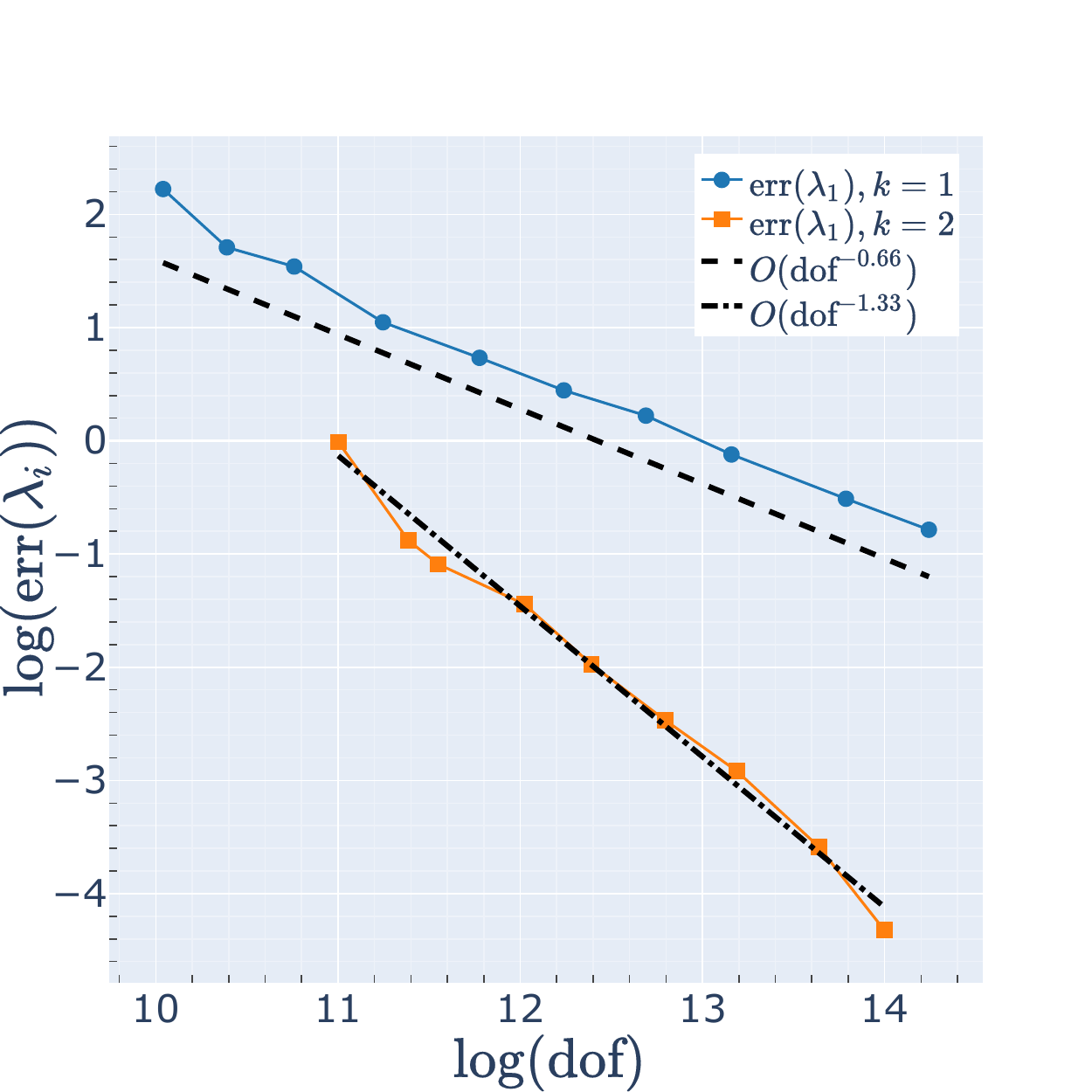}
	\end{minipage}
	\begin{minipage}{0.49\linewidth}\centering
		\includegraphics[scale=0.32,trim=0cm 0cm 2cm 2cm,clip]{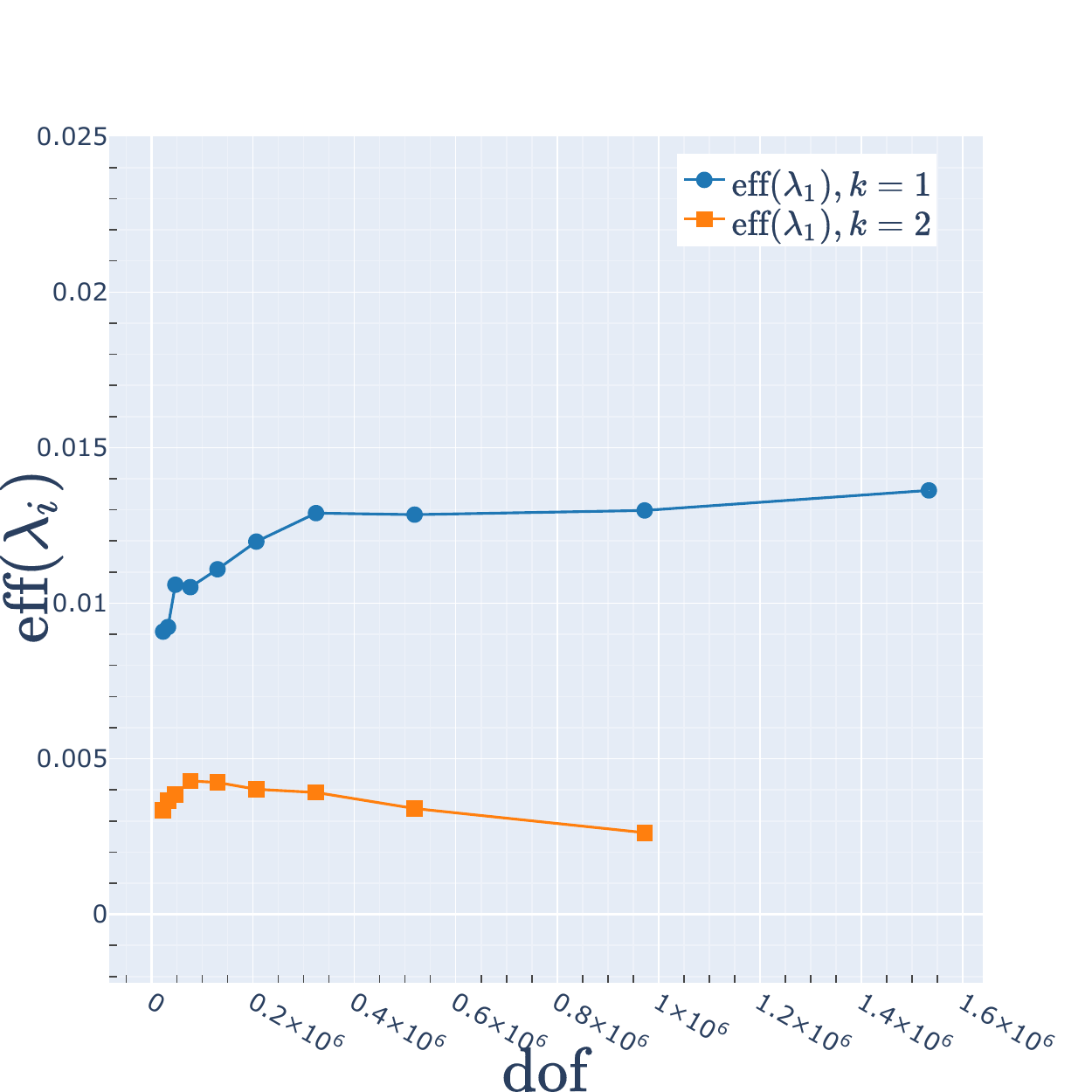}
	\end{minipage}
	\caption{Test \ref{subsec:3d-channel}. Error history for the adaptive refinements when computing the first eigenvalue (left) together with their corresponding effectivity indexes (right) in the symmetric case ($\varepsilon=1$).}
	\label{fig:error-eff-channel3d-symmetric}
\end{figure}

%\subsection*{Availability of Data and Materials} The datasets generated and/or analyzed during the current study are available from the corresponding author on reasonable request.
%
%\section*{Declarationa}
%
%\subsection*{Conflict of interest} The authors have no competing interests to declare that are relevant to the content of this article. Data sharing not applicable to this article as no datasets were generated or analysed during the current study.

%%%%%%%%%%%%%%%%%%%%%%%%%%%%%%%%%%%%%%%%%%%%%%%%%%%%%%%%%%%%%
%%%%%%%%%%%%%%%%%%%%%%%%%%%%%%%%%%%%%%%%%%%%%%%%%%%%%%%%%%%%%
%%%%%%%%%%%%%%%%%%%%%%%%%%%%%%%%%%%%%%%%%%%%%%%%%%%%%%%%%%%%%
%%%%%%%%%%%%%%%%%%%%%%%%%%%%%%%%%%%%%%%%%%%%%%%%%%%%%%%%%%%%%
%%%%%%%%%%%%%%%%%%%%%%%%%%%%%%%%%%%%%%%%%%%%%%%%%%%%%%%%%%%%%
%%%%%%%%%%%%%%%%%%%%%%%%%%%%%%%%%%%%%%%%%%%%%%%%%%%%%%%%%%%%%
%%%%%%%%%%%%%%%%%%%%%%%%%%%%%%%%%%%%%%%%%%%%%%%%%%%%%%%%%%%%%
%%%%%%%%%%%%%%%%%%%%%%%%%%%%%%%%%%%%%%%%%%%%%%%%%%%%%%%%%%%%%
\bibliographystyle{siamplain}
\bibliography{LRV_brinkmannEV}
\end{document}